\documentclass[a4paper,
fontsize=11pt,%
oneside,%
numbers=enddot]{scrartcl}
\KOMAoptions{DIV=13}    
\usepackage[utf8]{inputenc}
\usepackage[T1]{fontenc}
\usepackage{graphicx}
\usepackage{textcomp}
\usepackage[fleqn]{amsmath}
\usepackage{amsthm}
\usepackage{amssymb}
\usepackage{amsfonts}
\usepackage{soul}
\usepackage{pifont}        
\usepackage{libertine}
\usepackage[libertine,
slantedGreek,
nosymbolsc,
nonewtxmathopt,
subscriptcorrection]{newtxmath}
\usepackage[scaled=0.95,varqu,varl]{inconsolata}
\usepackage[scr=boondox]{mathalpha}   
\usepackage{euscript}   
\usepackage{leftindex}
\allowdisplaybreaks[1]  
\numberwithin{equation}{section}    
\usepackage[svgnames,hyperref]{xcolor}
\definecolor{orng}{HTML}{F35400}
\definecolor{bleu}{HTML}{BCE6F2}
\definecolor{dblue}{HTML}{0455BF}
\definecolor{dgreen}{HTML}{02724A}
\definecolor{dgreen2}{HTML}{025951}
\definecolor{dred}{HTML}{D90404}
\definecolor{dviolet}{HTML}{42208C}
\definecolor{labelkey}{HTML}{025951}
\definecolor{refkey}{HTML}{025951}
\definecolor{refkey}{rgb}{0,0.6,0.0}
\definecolor{Brown}{rgb}{0.45,0.0,0.05}
\definecolor{dgreen}{rgb}{0.00,0.49,0.00}
\definecolor{dblue}{rgb}{0,0.18,0.75}
\definecolor{lblue}{rgb}{0,0.7,0.75}
\definecolor{dviolet}{HTML}{9400D3}
\definecolor{pblue}{rgb}{0.1176,0.5647,1}
\definecolor{nblue}{rgb}{0.2,0.3,1}
\definecolor{pgreen}{rgb}{0.1961,0.8039,0.1961}
\definecolor{ngreen}{rgb}{0.0,0.6,0.3}
\definecolor{pred}{rgb}{1.0,0.2706,0.0}
\definecolor{magenta}{HTML}{ff00ff}
\definecolor{hotmagenta}{rgb}{1.0, 0.11, 0.81}
\definecolor{dorng}{rgb}{0.91,0.41,0.17}
\definecolor{dgray}{rgb}{0.41,0.41,0.41}
\definecolor{azure}{rgb}{0.0, 0.5, 1.0}
\usepackage{tikz,tkz-euclide,pgfplots}
\usetikzlibrary{arrows,decorations}
\addtokomafont{section}{\centering}

\usepackage{enumitem}
\setlist{itemsep=-2.0pt}

\usepackage{upref}  
\usepackage{hyperref}
\hypersetup{colorlinks=true,
linktocpage=true,
linkcolor=dblue,
citecolor=dgreen,
urlcolor=dred,
pdfencoding=auto,
hypertexnames=false}
\makeatletter
\g@addto@macro\th@plain{
\thm@headfont{\bfseries\sffamily}
\thm@notefont{}}
\g@addto@macro\th@definition{
\thm@headfont{\bfseries\sffamily}
\thm@notefont{}}
\g@addto@macro\th@remark{
\thm@headfont{\bfseries\sffamily}
\thm@notefont{}}
\makeatother
\theoremstyle{plain}
\newtheorem{theorem}{Theorem}[section]
\newtheorem{proposition}[theorem]{Proposition}
\newtheorem{corollary}[theorem]{Corollary}
\newtheorem{lemma}[theorem]{Lemma}

\theoremstyle{definition}
\newtheorem{definition}[theorem]{Definition}
\newtheorem{example}[theorem]{Example}
\newtheorem{problem}[theorem]{Problem}

\theoremstyle{remark}
\newtheorem{remark}[theorem]{Remark}

\newtheorem{algorithm}[theorem]{Algorithm}

\usepackage{epsfig}
\usepackage[usenames,dvipsnames]{pstricks}
\usepackage{pstricks-add}
\usepackage{pst-grad}
\usepackage{pst-plot} 
\usepackage{pst-node} 
\usepackage{pst-eucl} 
\usepackage{pst-coil} 
\usepackage[extdef=true]{delimset}
\DeclareMathDelimiterSet{\scal}[2]{
\selectdelim[l]<{#1}
\mathpunct{}\selectdelim[p]|
{#2}\selectdelim[r]>}
\DeclareMathDelimiterSet{\EC}[2]{
\mathsf{E}\selectdelim[l]({#1}
\mathpunct{}\selectdelim[p]|
{#2}\selectdelim[r])}

\newcommand{\menge}[2]{\bigl\{{#1}\mid{#2}\bigr\}} 
\DeclareMathDelimiterSet{\Menge}[2]{\selectdelim[l]\{
{#1}\selectdelim[m]|{#2}\selectdelim[r]\}}

\makeatletter
\def\upintkern@{\mkern-7mu\mathchoice{\mkern-3.5mu}{}{}{}}
\def\upintdots@{\mathchoice{\mkern-4mu\@cdots\mkern-4mu}%
{{\cdotp}\mkern1.5mu{\cdotp}\mkern1.5mu{\cdotp}}%
{{\cdotp}\mkern1mu{\cdotp}\mkern1mu{\cdotp}}%
{{\cdotp}\mkern1mu{\cdotp}\mkern1mu{\cdotp}}}
\makeatother
\DeclareFontFamily{OMX}{mdbch}{}
\DeclareFontShape{OMX}{mdbch}{m}{n}{ <->s * [0.8]  mdbchr7v }{}
\DeclareFontShape{OMX}{mdbch}{b}{n}{ <->s * [0.8]  mdbchb7v }{}
\DeclareFontShape{OMX}{mdbch}{bx}{n}{<->ssub * mdbch/b/n}{}
\DeclareSymbolFont{uplargesymbols}{OMX}{mdbch}{m}{n}
\SetSymbolFont{uplargesymbols}{bold}{OMX}{mdbch}{b}{n}
\DeclareMathSymbol{\upintop}{\mathop}{uplargesymbols}{82}
\DeclareMathSymbol{\upointop}{\mathop}{uplargesymbols}{"48}
\makeatletter
\renewcommand{\int}{\DOTSI\upintop\ilimits@}
\renewcommand{\oint}{\DOTSI\upointop\ilimits@}
\makeatother
\newcommand{\RR}{\mathbb{R}}
\newcommand{\NN}{\mathbb{N}}
\newcommand{\XX}{\EuScript{X}}
\newcommand{\YY}{\EuScript{Y}}

\newcommand{\CS}{\mathsf{C}}
\newcommand{\HS}{\mathsf{H}}
\newcommand{\GS}{\mathsf{G}}
\newcommand{\ZS}{\mathsf{Z}}
\newcommand{\zS}{\mathsf{z}}
\newcommand{\nS}{{\mathsf{n}}}
\newcommand{\nnn}{\mathsf{n}\in\mathbb{N}}
\newcommand{\jjj}{\mathsf{j}\in\mathbb{N}}
\newcommand{\iS}{\mathsf{i}}
\newcommand{\dS}{\mathsf{d}}
\newcommand{\jS}{\mathsf{j}}
\newcommand{\kS}{\mathsf{k}}
\newcommand{\xS}{\mathsf{x}}

\newcommand{\LS}{\mathsf{L}}
\newcommand{\TS}{\mathsf{T}}

\newcommand{\BE}{\EuScript{B}}
\newcommand{\FE}{\EuScript{F}}

\newcommand{\unif}{\ensuremath{\text{\rmfamily uniform}}}
\newcommand{\pinf}{{+}\infty}
\newcommand{\minf}{{-}\infty}
\newcommand{\zeroun}{\intv[o]{0}{1}}

\newcommand{\RP}{\intv[r]0{0}{\pinf}}
\newcommand{\RPP}{\intv[o]0{0}{\pinf}}

\newcommand{\emp}{\varnothing}

\newcommand{\WC}{\ensuremath{{\mathfrak W}}}
\newcommand{\SC}{\ensuremath{{\mathfrak S}}}
\newcommand{\Sum}{\displaystyle\sum}

\newcommand{\pushfwd}%
{\ensuremath{\mbox{\Large$\,\triangleright\,$}}}

\DeclareMathOperator{\Argmin}{Argmin}
\DeclareMathOperator{\argmin}{argmin}
\newcommand{\Id}{\mathsf{Id}}

\DeclareMathOperator{\Fix}{Fix}

\DeclareMathOperator{\zer}{zer}
\DeclareMathOperator{\inte}{int}

\DeclareMathOperator{\prox}{prox}
\DeclareMathOperator{\proj}{proj}

\newcommand{\EE}{\mathsf{E}}
\newcommand{\PP}{\mathsf{P}}

\renewcommand{\leq}{\leqslant}
\renewcommand{\geq}{\geqslant}

\newcommand{\weakly}{\rightharpoonup}

\newcommand{\Pas}{\text{\normalfont$\PP$-a.s.}}

\renewenvironment{abstract}{%
\vspace*{-0.50cm}
\small
\quotation%
\noindent%
{\normalfont\bfseries\sffamily
\nobreak\abstractname\ }%
}{%
\endquotation%
\medskip
}
\renewcommand{\abstractname}{Abstract.}
\newcommand\keywordsname{Keywords.}
\newenvironment{keywords}
{\renewcommand\abstractname{\keywordsname}\begin{abstract}}
{\end{abstract}}

\usepackage[auth-sc]{authblk}
\newcommand{\email}[1]{\href{mailto:#1}{\nolinkurl{#1}}}
\renewcommand*\Affilfont{\normalfont\normalsize}
\newcommand\affilcr{\protect\\ \protect\Affilfont}
\makeatletter
\renewcommand\AB@affilsepx{\protect\\[0.5em]}
\makeatother

\author[1]{Patrick L. Combettes}
\affil[1]{North Carolina State University
\affilcr
Department of Mathematics
\affilcr
Raleigh, NC 27695, USA
\affilcr
\email{plc@math.ncsu.edu}
}
\author[2]{Javier I. Madariaga}
\affil[2]{North Carolina State University
\affilcr
Department of Mathematics
\affilcr
Raleigh, NC 27695, USA
\affilcr
\email{jimadari@ncsu.edu}
}

\begin{document}

\title{A Geometric Framework for\\ 
Stochastic Iterations\thanks{Contact author: P. L.
Combettes. Email: \email{plc@math.ncsu.edu}.
This work was supported by the National
Science Foundation under grant DMS-2513409.
}}

\date{~}

\maketitle

\thispagestyle{empty}
\begin{abstract}
This paper concerns models and convergence principles for dealing
with stochasticity in a wide range of algorithms arising in
nonlinear analysis and optimization in Hilbert spaces. It proposes
a flexible geometric framework within which existing solution
methods can be recast and improved, and new ones can be designed.
Almost sure weak, strong, and linear convergence results are
established in particular for stochastic fixed point iterations, 
the stochastic gradient descent method, and stochastic extrapolated
parallel algorithms for feasibility problems. In these areas, 
the proposed algorithms exceed the features and convergence 
guarantees of the state of the art. Numerical applications to 
signal and image recovery are provided.
\end{abstract}


\begin{MSC}
47N10,
90C15,
90C48,
47J26,
62L20.
\end{MSC}
\newpage

\setcounter{page}{1}
\section{Introduction}
\label{sec:1}

The objective of this paper is to propose a general algorithmic
framework and convergence principles for dealing with stochasticity
in a broad class of algorithms arising in optimization and
numerical nonlinear analysis. Throughout, $\HS$ is a separable real
Hilbert space and the underlying probability space 
$(\upOmega,\FE,\PP)$ is complete. 
We denote by $\ZS\subset\HS$ the set of 
solutions to the problem of interest and assume that it is 
nonempty and closed. 

In the recent paper \cite{Acnu24}, we showed that a simple geometry
underlies most deterministic monotone operator splitting algorithms
and that, by exploiting this geometry, the convergence analysis of
existing methods could be simplified and improved. It was also
argued that this geometric framework provides a flexible template
to create new algorithms. The basic idea is to construct the update
at iteration $\nS$ of a deterministic algorithm for finding a point
in the solution set $\mathsf{Z}$ as a relaxed projection
$\xS_{\nS+1}=\xS_{\nS}+\uplambda_{\nS}(\proj_{\HS_{\nS}}\xS_{\nS}
-\xS_{\nS})$ onto a half-space 
$\HS_{\nS}=\menge{\zS\in\HS}{\scal{\zS}{\mathsf{t}_{\nS}^*}_{\HS}
\leq\upeta_{\nS}}$
containing $\mathsf{Z}$ as follows (see Fig.~\ref{fig:1}(a)). 

\begin{algorithm}
\label{algo:1}
Let $\mathsf{x}_{0}\in\HS$ and iterate
\vspace{-2mm}
\begin{equation}
\label{e:0}
\begin{array}{l}
\text{for}\;\nS=0,1,\ldots\\
\left\lfloor
\begin{array}{l}
\text{take}\;\mathsf{t}^*_{\nS}\in\HS\;
\text{and}\;\upeta_{\nS}\in\RR
\;\text{such that}\;(\forall\zS\in\mathsf{Z})\;\;
\scal{\mathsf{z}}{\mathsf{t}_{\nS}^*}_{\HS}\leq\upeta_{\nS}\\[1mm]
\upalpha_{\nS}=
\begin{cases}
\dfrac{\scal{\xS_{\nS}}
{\mathsf{t}_{\nS}^*}_{\HS}-\upeta_{\nS}}
{\|\mathsf{t}_{\nS}^*\|_{\HS}^2},
&\text{if}\;\;
\scal{\xS_{\nS}}{\mathsf{t}_{\nS}^*}_{\HS}>\upeta_{\nS};\\
0,&\text{otherwise}
\end{cases}\\[7mm]
\mathsf{d}_{\nS}=\upalpha_{\nS}\mathsf{t}_{\nS}^*\\
\text{take}\;\uplambda_{\nS}\in\left]0,2\right[\\
\mathsf{x}_{\nS+1}=\mathsf{x}_{\nS}-
\uplambda_{\nS}\mathsf{d}_{\nS}.
\end{array}
\right.\\
\end{array}
\end{equation} 
\end{algorithm}

Our approach consists in extending the above geometric construction
to a general stochastic environment by making the following changes
at iteration $\nS$:
\begin{itemize}
\item
The deterministic quantities $\mathsf{t}_{\nS}^*$ and 
$\upeta_{\nS}$ are replaced by random ones.
\item
A stochastic tolerance is added in the construction of the
outer approximation.
\item
The relaxation parameter $\uplambda_{\nS}$ is now random and no 
longer restricted to the interval $\left]0,2\right[$.
\end{itemize}
This leads to the following algorithmic scheme 
(see Section~\ref{sec:21} for notation).

\begin{algorithm}
\label{algo:2}
Let $x_{0}\in L^2(\upOmega,\FE,\PP;\HS)$ and iterate
\begin{equation}
\label{e:a2}
\begin{array}{l}
\text{for}\;\nS=0,1,\ldots\\
\left\lfloor
\begin{array}{l}
\XX_{\nS}=\upsigma(x_{0},\dots,x_{\nS})\\
\text{take}\;
t^*_{\nS}\in L^2(\upOmega,\FE,\PP;\HS)\; 
\text{and}\;\eta_{\nS}\in{L^1(\upOmega,\FE,\PP;\RR)}
\;\text{such that}\\[2mm]
\qquad\begin{cases}
\dfrac{\mathsf{1}_{[t_{\nS}^*\neq0]}\mathsf{1}_{
\left[\scal{x_{\nS}}{t_{\nS}^*}_{\HS}>\eta_{\nS}\right]}\eta_{\nS}}
{\|t_{\nS}^*\|_{\HS}+\mathsf{1}_{[t_{\nS}^*=0]}}\in
L^2(\upOmega,\FE,\PP;\RR);\\[3mm]
\alpha_{\nS}=
\dfrac{\mathsf{1}_{[t_{\nS}^*\neq0]}\mathsf{1}_{
\left[\scal{x_{\nS}}{t_{\nS}^*}_{\HS}>\eta_{\nS}\right]}
\bigl(\scal{x_{\nS}}{t_{\nS}^*}_{\HS}-\eta_{\nS}\bigr)}
{\|t_{\nS}^*\|_{\HS}^2+\mathsf{1}_{[t_{\nS}^*=0]}};\\
(\forall\mathsf{z}\in\ZS)\;\;
\scal{\mathsf{z}}{\EC{\alpha_{\nS}t^*_{\nS}}{\XX_{\nS}}}_{\HS}\leq
\EC{\alpha_{\nS}\eta_{\nS}}{\XX_{\nS}}+
\varepsilon_{\nS}(\cdot,\mathsf{z})\;\Pas\\
\hspace{36mm}\text{where}\;\varepsilon_{\nS}(\cdot,\mathsf{z})\in 
\RP\;\Pas
\end{cases}\\
d_{\nS}=\alpha_{\nS}t_{\nS}^*\\
\text{take}\;\lambda_{\nS}\in L^\infty(\upOmega,\FE,\PP;\RPP)\\
x_{\nS+1}=x_{\nS}-\lambda_{\nS} d_{\nS}.
\end{array}
\right.\\
\end{array}
\end{equation}
\end{algorithm}

Implicitly, Algorithm~\ref{algo:2} constructs a random outer
approximation $S_{\nS}$ to $\ZS$, namely 
\begin{equation}
\ZS\subset S_{\nS}=\menge{\mathsf{z}\in\HS}
{\scal{\mathsf{z}}{\EC{\alpha_{\nS}t^*_{\nS}}{\XX_{\nS}}}_{\HS}
\leq\EC{\alpha_{\nS}\eta_{\nS}}{\XX_{\nS}}+
\varepsilon_{\nS}(\cdot,\mathsf{z})}\;\Pas 
\end{equation} 
and the
update $x_{\nS+1}$ is obtained by performing a relaxed projection
of the current iterate $x_{\nS}$ onto the simpler set
\begin{equation}
H_{\nS}=\menge{\mathsf{z}\in\HS}
{\scal{\mathsf{z}}{t_{\nS}^*}_{\HS}\leq\eta_{\nS}},
\end{equation} 
which is a random affine half-space. It should be noted that, while
$\ZS\subset S_{\nS}$, the more restrictive inclusion
$\ZS\subset H_{\nS}$ does not hold in general 
(see Fig.~\ref{fig:1}(b)). In terms
of modeling, choosing $t^*_{\nS}$ and $\eta_{\nS}$ such that
$\ZS\subset H_{\nS}$ would restrict the scope of the processes we
intend to model, whereas the more general inclusion 
$\ZS\subset S_{\nS}$ offers more flexibility. Let us observe that,
if $\varepsilon_{\nS}=0$, $S_{\nS}$ is also a random half-space.
However, as the following example shows, projecting onto it is not
judicious.

\begin{example}
\label{ex:1}
For every $\kS\in\{1,\ldots,\mathsf{p}\}$, let $\CS_{\kS}$ be a
closed convex subset of $\HS$. Suppose that
$\ZS=\bigcap_{\kS=1}^{\mathsf{p}}\CS_{\kS}\neq\emp$ and 
implement Algorithm~\ref{algo:2} with $\lambda_{\nS}=1$, 
$\varepsilon_{\nS}=0$,
$t_{\nS}^*=x_{\nS}-\proj_{\CS_{k}}x_{\nS}$, and
$\eta_{\nS}=\scal{\proj_{\CS_{k}}x_{\nS}}{t_{\nS}^*}_{\HS}$, where 
the random variable $k$ is uniformly distributed in 
$\{1,\ldots,\mathsf{p}\}$. Then
$\EC{t_{\nS}^*}{\XX_{\nS}}=x_{\nS}-
\mathsf{p}^{-1}\sum_{\mathsf{k}=1}^\mathsf{p}
\proj_{\CS_{\mathsf{k}}}x_{\nS}$ and therefore 
\begin{align}
\mathsf{Z}
&\subset\Menge3{\mathsf{z}\in\HS}
{\sum_{\mathsf{k}=1}^\mathsf{p}
\scal1{\mathsf{z}-\proj_{\CS_{\mathsf{k}}}x_{\nS}}
{x_{\nS}-\proj_{\CS_{\mathsf{k}}}x_{\nS}}_{\HS}\leq 0}\nonumber\\
&=\Menge1{\mathsf{z}\in\HS}
{\scal1{\mathsf{z}}{\EC{t_{\nS}^*}{\XX_{\nS}}}_{\HS}
\leq\EC1{\eta_{\nS}}{\XX_{\nS}}}\nonumber\\
&=S_{\nS}\;\;\Pas
\end{align}
Thus, Algorithm~\ref{algo:2} yields the random iteration process
$x_{\nS+1}=\proj_{\CS_{k}}x_{\nS}$ in which a single, randomly
selected set is projected onto at iteration $\nS$. By contrast,
projecting onto $S_{\nS}$ would yield the barycentric projection
method $x_{\nS+1}=\mathsf{p}^{-1}\sum_{\kS=1}^\mathsf{p}
\proj_{\CS_{\kS}}x_{\nS}$, which is deterministic and imposes the
computation of all $\mathsf{p}$ projections at each iteration. 
\end{example}

Another new feature of Algorithm~\ref{algo:2} is that the
relaxation parameters $(\lambda_{\nS})_{\nnn}$ are random. In
addition, they need not be confined to the range $\left]0,2\right[$
imposed in deterministic algorithms
\cite{Livre1,Byrn14,Aiep96,Acnu24,Dong22}. We call such relaxations
\emph{super relaxations}. 

\begin{figure}[ht]
\begin{center}
\scalebox{0.7} 
{
\begin{pspicture}(-1.5,-1.3)(12.0,5.5) 
\definecolor{color0}{rgb}{0.7,0.7,1.0}
\definecolor{color0b}{rgb}{0.53,0.729,0.808}
\definecolor{color98b}{rgb}{0.8,0.0,0.0}
\rput{5.0}(-1.58,2.6){\psellipse[linewidth=0.06,dimen=inner,%
fillstyle=solid,fillcolor=color0](4.1,-2.06)(3.0,0.8)}
\psline[linewidth=0.06cm,linestyle=solid,linecolor=dred]%
(5.0,-1.2)(8.2,4.9)
\psline[linewidth=0.04cm,arrowsize=0.26cm,linestyle=solid]{->}%
(7.3,0.8)(6.35,1.27)
\rput(4.85,-0.5){\color{dred}\Large${\mathsf{H}_{\mathrm{n}}}$}
\rput(7.3,0.8){\Large$\bullet$}
\rput(2.4,0.9){\Large$\boldsymbol{\mathsf{Z}}$}
\rput(7.7,0.56){\Large$\mathsf{x}_{\mathrm{n}}$}
\rput(6.29,1.27){\Large$\bullet$}
\rput(6.0,1.8){\Large$\mathsf{x}_{\mathrm{n}+1}$}
\end{pspicture} 

\begin{pspicture}(1.0,-1.3)(12.0,5.5) 
\definecolor{color0}{rgb}{0.7,0.7,1.0}
\definecolor{color0b}{rgb}{0.53,0.729,0.808}
\definecolor{color98b}{rgb}{0.8,0.0,0.0}
\rput{5.0}(-1.58,2.6){\psellipse[linewidth=0.06,dimen=inner,%
fillstyle=solid,fillcolor=color0](4.1,-2.06)(3.0,0.8)}
\psline[linewidth=0.06cm,linestyle=solid,linecolor=dgreen]%
(5.6,-1.2)(8.3,4.9)
\psline[linewidth=0.06cm,linestyle=solid,linecolor=dblue]%
(3.1,-1.2)(8.6,4.9)
\psline[linewidth=0.05cm,arrowsize=0.26cm,linestyle=solid]{->}%
(7.3,0.8)(6.07,1.96)
\rput(5.4,-0.5){\color{dgreen}\Large $S_{\mathrm{n}}$}
\rput(3.1,-0.5){\color{dblue}\Large${H_{\mathrm{n}}}$}
\rput(7.3,0.8){\Large$\bullet$}
\rput(2.4,0.9){\Large$\boldsymbol{\mathsf{Z}}$}
\rput(7.7,0.56){\Large$x_{\mathrm{n}}$}
\rput(6.0,2.0){\Large$\bullet$}
\rput(5.1,2.2){\Large$x_{\mathrm{n}+1}$}
\end{pspicture} 
}
\caption{Geometry of algorithms for finding a point in
$\mathsf{Z}$ with $\uplambda_{\nS}=1$.
(a) Left: Iteration $\nS$ of the deterministic 
Algorithm~\ref{algo:1}.
(b) Right: Iteration $\nS$ of the stochastic 
Algorithm~\ref{algo:2} with $\varepsilon_{\nS}=0$.}
\label{fig:1}
\end{center}
\end{figure}

The deterministic setting of Algorithm~\ref{algo:1} is known to
capture a vast array of iterative methods in nonlinear analysis and
optimization \cite{Acnu24}. Our premise is that
Algorithm~\ref{algo:2} can serve the same purpose for their
stochastic counterparts. Weak, strong, and linear convergence
results will be established for Algorithm~\ref{algo:2}. In turn,
these results will be applied to fixed point and feasibility
problems, where they will be shown to provide new stochastic
algorithms that go beyond the state of the art not only in terms of
convergence guarantees but also of flexibility of implementation
and scope.

The remainder of the paper is organized as follows. Notation and
preliminary results are introduced in Section~\ref{sec:2}. The main
theorems are presented in Section~\ref{sec:3}, where the asymptotic
properties of Algorithm~\ref{algo:2} are established.
Section~\ref{sec:4} is devoted to an application of the proposed
theory to a randomly relaxed Krasnosel'ski\u\i--Mann iteration
process and includes new results on the convergence of the
stochastic gradient method. Section~\ref{sec:5} concerns an
application to randomly activated and relaxed extrapolated fixed
point methods for common fixed point problems in the presence of
possibly uncountably many operators. These results significantly
improve existing ones. Section~\ref{sec:6} concludes the paper with
applications to signal and image recovery. Applications of
the results of Section~\ref{sec:3} to the design and the 
analysis of stochastic splitting algorithms for monotone 
inclusion problems are addressed in the companion paper
\cite{Sadd25}.

\section{Notation and background}
\label{sec:2}

\subsection{Notation}
\label{sec:21}
We use sans-serif letters to denote deterministic variables and
italicized serif letters to denote random variables. 

The Hilbert space $\HS$ has identity
operator $\Id$, scalar product $\scal{\cdot}{\cdot}_{\HS}$,
and associated norm $\|\cdot\|_{\HS}$. The symbols $\weakly$ and
$\to$ denote weak and strong convergence in $\HS$, respectively.
The sets of strong and weak sequential cluster points of a sequence
$(\mathsf{x}_{\nS})_{\nnn}$ in $\HS$ are denoted by
$\SC(\mathsf{x}_{\nS})_{\nnn}$ and
$\WC(\mathsf{x}_{\nS})_{\nnn}$, respectively. The distance
function of a set $\CS\subset\HS$ is 
$\dS_{\CS}\colon\mathsf{x}\mapsto
\inf_{\mathsf{y}\in\CS}\|\mathsf{y}-\mathsf{x}\|_{\HS}$ and the
projection onto a nonempty closed convex set $\CS\subset\HS$ is
denoted by $\proj_{\CS}$. The fixed
point set of an operator $\TS\colon\HS\to\HS$ is 
$\Fix\TS=\menge{\mathsf{x}\in\HS}{\TS\mathsf{x}=\mathsf{x}}$. 
The following notion will play an important role in
Sections~\ref{sec:4} and \ref{sec:5}; see
\cite[Proposition~2.4]{Sico10} for examples of demiregular
operators.

\begin{definition}{\protect{\cite{Sico10}}}
\label{d:10}
$\mathsf{T}\colon\HS\to\HS$ is demiregular at $\xS\in\HS$
if, for every sequence $(\xS_{\nS})_{\nnn}$ in $\HS$ such that
$\xS_{\nS}\weakly\xS$ and $\TS\xS_{\nS}\to\TS\xS$, we have
$\xS_{\nS}\to\xS$.
\end{definition}

Let $(\upXi,\EuScript{G})$ be a measurable space. A $\upXi$-valued
random variable is a measurable mapping
$x\colon(\upOmega,\FE,\PP)\to(\upXi,\EuScript{G})$. 
Given $x\colon\upOmega\to\upXi$ and $\mathsf{S}\in\EuScript{G}$, we
set $[x\in\mathsf{S}]=\menge{\upomega\in\upOmega}
{x(\upomega)\in\mathsf{S}}$. Let $x$ and $y$ be random variables 
from $(\upOmega,\FE,\PP)$ to $(\upXi,\EuScript{G})$. Then $y$ is a
copy of $x$ if, for every $\mathsf{S}\in\EuScript{G}$, 
$\PP([x\in\mathsf{S}])=\PP([y\in\mathsf{S}])$.
The Borel $\upsigma$-algebra of $\HS$ is denoted by $\BE_{\HS}$.
An $\HS$-valued random variable is a measurable 
mapping $x\colon(\upOmega,\FE)\to(\HS,\BE_{\HS})$. 
Let $\mathsf{p}\in\left[1,\pinf\right[$ and let $\XX$ be a
sub $\upsigma$-algebra of $\FE$. Then
$L^\mathsf{p}(\upOmega,\XX,\PP;\HS)$ denotes the space of 
equivalence classes of $\Pas$ equal $\HS$-valued random variables
$x\colon(\upOmega,\XX,\PP)\to(\HS,\BE_{\HS})$ such that 
$\EE\|x\|^\mathsf{p}_{\HS}<\pinf$. Endowed with the norm 
\begin{equation}
\|\cdot\|_{L^\mathsf{p}(\upOmega,\XX,\PP;\HS)}\colon
x\mapsto\EE^{1/\mathsf{p}}\|x\|_{\HS}^\mathsf{p}
=\brk3{\int_{\upOmega}\|x(\upomega)\|_{\HS}^{\mathsf{p}}
\PP(d\upomega)}^{1/\mathsf{p}},
\end{equation}
$L^\mathsf{p}(\upOmega,\XX,\PP;\HS)$ is a
real Banach space and $L^2(\upOmega,\XX,\PP;\HS)$ is
a real Hilbert space with scalar product 
\begin{equation}
\scal{\cdot}{\cdot}_{L^2(\upOmega,\XX,\PP;\HS)}\colon
(x,y)\mapsto \EE\scal{x}{y}_{\HS}=\int_{\upOmega}
\scal1{x(\upomega)}{y(\upomega)}_{\HS}\PP(d\upomega).
\end{equation}
Further,
\begin{equation}
(\forall\mathsf{S}\in\BE_{\HS})\quad
L^\mathsf{p}(\upOmega,\XX,\PP;\mathsf{S})=\Menge1{x\in
L^\mathsf{p}(\upOmega,\XX,\PP;\HS)}{x\in\mathsf{S}\;\Pas}.
\end{equation}
The $\upsigma$-algebra generated by a family $\upPhi$ of random 
variables is denoted by $\upsigma(\upPhi)$. Let 
$\mathfrak{F}=(\FE_{\nS})_{\nnn}$ be a sequence of sub 
$\upsigma$-algebras of $\FE$ such that $(\forall\nnn)$ 
$\FE_{\nS}\subset\FE_{\nS+1}$. Then $\ell_+(\mathfrak{F})$ is the
set of sequences of $\RP$-valued random variables
$(\xi_{\nS})_{\nnn}$ such that, for every $\nnn$, $\xi_{\nS}$ is
$\FE_{\nS}$-measurable. We set
\begin{equation}
(\forall\mathsf{p}\in\RPP)\quad\ell^\mathsf{p}_+(\mathfrak{F})=
\Menge3{(\xi_{\nS})_{\nnn}\in\ell_+(\mathfrak{F})}
{\sum_{\nnn}\xi_{\nS}^\mathsf{p}<\pinf\;\Pas}.
\end{equation}
We say that $\varphi\colon\upOmega\times\HS\to\RR$ is an
$\XX$-Carath\'eodory integrand if
\begin{equation}
\begin{cases}
\text{for}\;\PP\text{-almost every}\;\upomega\in\upOmega,
&\varphi(\upomega,\cdot)\;\text{is continuous};\\
\text{for every}\;\xS\in\HS,&\varphi(\cdot,\xS)\;
\text{is}\;\XX\text{-measurable.}
\end{cases}
\end{equation}
We denote by $\mathfrak{C}(\upOmega,\XX,\PP;\HS)$ the class of 
$\XX$-Carath\'eodory integrands
$\varphi\colon\upOmega\times\HS\to\RP$.

The reader is referred to \cite{Livre1} for background on convex
analysis and fixed point theory, and to \cite{Hyto16,Ledo91} for
background on probability in Hilbert spaces. 

\subsection{Preliminary results}
\label{sec:22}

\begin{definition}
Let $\XX$ be a sub $\upsigma$-algebra of $\FE$, 
$\CS\in\BE_{\HS}$, and $x$ be an $\HS$-valued random 
variable. Then $x$ is a $\CS$-valued 
$\XX$-simple mapping if there exist a finite family of 
disjoint sets $(\mathsf{F}_{\iS})_{1\leq\iS\leq\mathsf{N}}$ in 
$\XX$ and a family of vectors
$(\mathsf{z}_{\iS})_{1\leq\iS\leq\mathsf{N}}$ in $\CS$ such
that 
\begin{equation}
\bigcup_{\iS=1}^{\mathsf{N}}\mathsf{F}_{\iS}=\upOmega\;\;
\text{and}\;\;
x=\sum_{\iS=1}^{\mathsf{N}}\mathsf{1}_{\mathsf{F}_{\iS}}
\mathsf{z}_{\iS}\;\;\Pas
\end{equation}
\end{definition}

\begin{remark}
Let $\mathsf{p}\in[1,\pinf[$. Then every $\CS$-valued
$\XX$-simple mapping is in $L^\mathsf{p}(\upOmega,\XX,\PP;\CS)$.
\end{remark}

The following proposition is an adaptation of 
\cite[Corollary~1.1.7]{Hyto16}.

\begin{proposition}
\label{p:seq.simple}
Let $\CS$ be a nonempty closed subset of $\HS$, $\XX$ be a
sub $\upsigma$-algebra of $\FE$, $\mathsf{p}\in[1,\pinf[$,
and $x\in L^{\mathsf{p}}(\upOmega,\XX,\PP;\CS)$. Then 
there exists a sequence $(x_{\nS})_{\nnn}$ of $\CS$-valued 
$\XX$-simple mappings that converges strongly $\Pas$ to $x$ with 
$\sup_{\nnn}\|x_{\nS}\|_{\HS}^{\mathsf{p}}
\leq\|x\|_{\HS}^{\mathsf{p}}+1\;\Pas$ 
\end{proposition}
\begin{proof}
Let $\mathsf{z}\in\CS$ be such that 
$\|\mathsf{z}\|_{\HS}^{\mathsf{p}}
\leq\inf_{\mathsf{y}\in\CS}\|\mathsf{y}\|_{\HS}^{\mathsf{p}}+1$
and let $\{\mathsf{z}_{\nS}\}_{\nnn}$ be a countable dense subset
of $\CS$ with $\mathsf{z}_{0}=\mathsf{z}$. For every 
$\nnn$ and every $\mathsf{y}\in\CS$, define
$\mathsf{I}_{\nS,\mathsf{y}}=\menge{\iS\in\{0,\ldots,\nS\}}
{\|\mathsf{z}_{\iS}\|_{\HS}^{\mathsf{p}}\leq 
\|\mathsf{y}\|_{\HS}^{\mathsf{p}}+1}$ and let 
$\iS_{\nS,\mathsf{y}}$ be the smallest integer 
$\iS\in\mathsf{I}_{\nS,\mathsf{y}}$ such that 
$\|\mathsf{y}-\mathsf{z}_{\iS}\|_{\HS}
=\min_{\jS\in\mathsf{I}_{\nS,\mathsf{y}}}
\|\mathsf{y}-\mathsf{z}_{\jS}\|_{\HS}$. Define, for every $\nnn$, 
$\TS_{\nS}\colon\CS\to\CS\colon
\mathsf{y}\mapsto\mathsf{z}_{\iS_{\nS,\mathsf{y}}}$. It
follows from the density of $\{\mathsf{z}_{\nS}\}_{\nnn}$ in $\CS$ 
that, for every $\mathsf{y}\in\CS$,
$\TS_{\nS}\mathsf{y}\to\mathsf{y}$ and 
\begin{equation}
(\forall\nnn)\quad\|\TS_{\nS}\mathsf{y}\|_{\HS}^{\mathsf{p}}
\leq\|\mathsf{y}\|_{\HS}^{\mathsf{p}}+1.
\end{equation}
Set, for every $\nnn$, $x_{\nS}=\TS_{\nS}x$. Then
$(x_{\nS})_{\nnn}$ converges strongly $\Pas$ to $x$ and 
\begin{equation}
(\forall\nnn)\quad\|x_{\nS}\|_{\HS}^{\mathsf{p}}
\leq\|x\|_{\HS}^{\mathsf{p}}+1\;\Pas
\end{equation}
It remains to show that $(x_{\nS})_{\nnn}$ is a sequence of 
$\CS$-valued $\XX$-simple mappings. Fix
$\nnn$. Then
\begin{equation}
\label{e:meas0}
[x_{\nS}=\mathsf{z}_{0}]=\Menge2{\upomega\in\upOmega}
{\|x(\upomega)-\mathsf{z}_{0}\|_{\HS}
=\min_{\jS\in\mathsf{I}_{\nS,x(\upomega)}}
\|x(\upomega)-\mathsf{z}_{\jS}\|_{\HS}
}
\end{equation}
and, for every $\iS\in\{1,\ldots,\nS\}$,
\begin{equation}
\label{e:meas1}
[x_{\nS}=\mathsf{z}_{\iS}]=\Menge2{\upomega\in\upOmega}
{\iS\in\mathsf{I}_{\nS,x(\upomega)}\;\text{and}\;
\|x(\upomega)-\mathsf{z}_{\iS}\|_{\HS}
=\min_{\jS\in\mathsf{I}_{\nS,x(\upomega)}}
\|x(\upomega)-\mathsf{z}_{\jS}\|_{\HS}
<\min_{\jS\in\mathsf{I}_{\iS-1,x(\upomega)}}
\|x(\upomega)-\mathsf{z}_{\jS}\|_{\HS}
}.
\end{equation}
By construction, \eqref{e:meas0} and \eqref{e:meas1} define sets 
in $\XX$. Further, 
\begin{equation}
\bigcup_{\iS=0}^{\nS}[x_{\nS}=\mathsf{z}_{\iS}]=\upOmega\;\;
\text{and}\;\;
x_{\nS}=\sum_{\iS=0}^{\nS}\mathsf{1}_{[x_{\nS}=\mathsf{z}_{\iS}]}
\mathsf{z}_{\iS},
\end{equation}
which confirms that $x_{\nS}$ is a $\CS$-valued
$\XX$-simple mapping.
\end{proof}

\begin{lemma}
\label{l:4}
Let $\mathfrak{F}=(\FE_{\nS})_{\nnn}$ be a sequence of
sub $\upsigma$-algebras of $\FE$ such that $(\forall\nnn)$
$\FE_{\nS}\subset\FE_{\nS+1}$. Let
$(\alpha_{\nS})_{\nnn}\in\ell_+(\mathfrak{F})$,
$(\theta_{\nS})_{\nnn}\in\ell_+(\mathfrak{F})$, and
$(\eta_{\nS})_{\nnn}\in\ell_+(\mathfrak{F})$. 
Then the following hold:
\begin{enumerate}
\item
\label{l:4i}
Suppose that $(\eta_{\nS})_{\nnn}\in\ell_+^1(\mathfrak{F})$ and
there exists a sequence
$(\chi_{\nS})_{\nnn}\in\ell_+^1(\mathfrak{F})$ satisfying
\begin{equation}
(\forall\nnn)\quad\EC{\alpha_{\nS+1}}{\FE_{\nS}}+\theta_{\nS}
\leq(1+\chi_{\nS})\alpha_{\nS}+\eta_{\nS}\;\Pas
\end{equation}
Then $(\theta_{\nS})_{\nnn}\in\ell_+^1(\mathfrak{F})$ and
$(\alpha_{\nS})_{\nnn}$ converges $\Pas$ to a $\RP$-valued random
variable.
\item
\label{l:4ii}
Suppose that $\EE\alpha_{0}<\pinf$, 
$\sum_{\nnn}\EE\eta_{\nS}<\pinf$,
and there exists a sequence $(\upchi_{\nS})_{\nnn}$ in $\RP$
satisfying $\varlimsup\upchi_{\nS}<1$ and
\begin{equation}
(\forall\nnn)\quad\EC{\alpha_{\nS+1}}{\FE_{\nS}}+\theta_{\nS}
\leq\upchi_{\nS}\alpha_{\nS}+\eta_{\nS}\;\Pas
\end{equation}
Then $\sum_{\nnn}\EE\theta_{\nS}<\pinf$ and
$\sum_{\nnn}\EE\alpha_{\nS}<\pinf$.
\end{enumerate}
\end{lemma}
\begin{proof}
\ref{l:4i}: This follows from \cite[Theorem~1]{Robb71}.

\ref{l:4ii}: This follows from \cite[Lemma~2.1(ii)]{MaPr19}. 
\end{proof}

\begin{corollary}
\label{l:5}
Let
$(\upalpha_{\nS})_{\nnn}$,
$(\uptheta_{\nS})_{\nnn}$,
$(\upeta_{\nS})_{\nnn}$, and
$(\upchi_{\nS})_{\nnn}$ be sequences in $\RP$.
Then the following hold:
\begin{enumerate}
\item
\label{l:5i}
Suppose that $\sum_{\nnn}\upeta_{\nS}<\pinf$,
$\sum_{\nnn}\upchi_{\nS}<\pinf$, and
\begin{equation}
(\forall\nnn)\quad\upalpha_{\nS+1}+\uptheta_{\nS}
\leq(1+\upchi_{\nS})\upalpha_{\nS}+\upeta_{\nS}.
\end{equation}
Then $\sum_{\nnn}\uptheta_{\nS}<\pinf$ and
$(\upalpha_{\nS})_{\nnn}$ converges 
to some $\upalpha\in\RP$.
\item
\label{l:5ii}
Suppose that $\sum_{\nnn}\upeta_{\nS}<\pinf$,
$\varlimsup\upchi_{\nS}<1$, and
\begin{equation}
(\forall\nnn)\quad\upalpha_{\nS+1}+\uptheta_{\nS}
\leq\upchi_{\nS}\upalpha_{\nS}+\upeta_{\nS}\;\Pas
\end{equation}
Then $\sum_{\nnn}\uptheta_{\nS}<\pinf$ and
$\sum_{\nnn}\upalpha_{\nS}<\pinf$.
\end{enumerate}
\end{corollary}
\begin{proof}
An application of Lemma~\ref{l:4} with
$(\forall\nnn)$ $\FE_{\nS}=\{\emp,\upOmega\}$. 
\end{proof}

\begin{lemma}
\label{l:6}
Let $\xi\in L^1(\upOmega,\FE,\PP;\RR)$, let $\upPhi$ be a family of
random variables, set $\XX=\upsigma(\upPhi)$, and let 
$\eta\in L^1(\upOmega,\FE,\PP;\RR)$ be independent of 
$\upsigma(\{\xi\}\cup\upPhi)$. Then
$\EC{\eta\xi}{\XX}=\EE\eta\EC{\xi}{\XX}$.
\end{lemma}
\begin{proof}
Note that $\XX\subset\upsigma(\{\xi\}\cup\upPhi)$ and
that $\xi$ is $\upsigma(\{\xi\}\cup\upPhi)$-measurable. Hence, it
follows from 
\cite[Properties H$^*$, K$^*$, and J$^*$ in Section~2.7.4]{Shir16}
that
\begin{equation}
\EC{\eta\xi}{\XX}
=\EC2{\EC1{\eta\xi}{\upsigma(\{\xi\}\cup\upPhi)}}{\XX}
=\EC2{\xi\EC1{\eta}{\upsigma(\{\xi\}\cup\upPhi)}}{\XX}
=\EC1{\xi\EE\eta}{\XX}
=\EE\eta\EC{\xi}{\XX},
\end{equation}
which proves the identity.
\end{proof}

\begin{lemma}
\label{l:7}
Let $\boldsymbol{x}=(x_1,\ldots,x_{\mathsf{N}})$ be an
$\HS^\mathsf{N}$-valued random variable,
let $(\mathsf{K},\EuScript{K})$ be a measurable space, and suppose
that the random variable
$k\colon(\upOmega,\FE,\PP)\to(\mathsf{K},\EuScript{K})$ is
independent of $\upsigma(\boldsymbol{x})$. Let 
$\mathsf{f}\colon(\mathsf{K}\times\HS,\EuScript{K}\otimes\BE_{\HS})
\to\RR$ be measurable and such 
that $\EE\abs{\mathsf{f}(k,x_1)}<\pinf$, and define 
$\mathsf{g}\colon\HS\to\RR\colon\xS\mapsto\EE\mathsf{f}(k,\xS)$.
Then, for $\PP$-almost every $\upomega'\in\upOmega$,
\begin{equation}
\label{e:cex}
\EC1{\mathsf{f}(k,x_1)}{\upsigma(\boldsymbol{x})}(\upomega')
=\int_{\upOmega}\mathsf{f}\brk1{k(\upomega),x_1(\upomega')}
\PP(d\upomega)
=\mathsf{g}\brk1{x_1(\upomega')}.
\end{equation}
\end{lemma}
\begin{proof}
Define 
$\boldsymbol{\mathsf{f}}\colon\mathsf{K}\times\HS^{\mathsf{N}}
\to\RR\colon(\kS,\boldsymbol{\xS})\mapsto\mathsf{f}(\kS,\xS_1)$.
Then $\boldsymbol{\mathsf{f}}$ is an $\RR$-valued measurable
function. Let $\mathsf{S}\in\upsigma(\boldsymbol{x})$. Then there
exists $\mathsf{A}\in\bigotimes_{1\leq\iS\leq\mathsf{N}}\BE_{\HS}$
such that $\mathsf{S}=[\boldsymbol{x}\in\mathsf{A}]$. Thus, it
follows from the image measure theorem
\cite[Theorem~2.6.7]{Shir16}, the independence of $k$ and
$\upsigma(\boldsymbol{x})$, and Fubini's theorem
\cite[Theorem~2.6.8]{Shir16} that
\begin{align}
\int_{\mathsf{S}}\mathsf{f}\brk1{k(\upomega),x_1(\upomega)}
\PP(d\upomega)
&=\int_{\upOmega}\boldsymbol{\mathsf{f}}\brk1{k(\upomega),
\boldsymbol{x}(\upomega)}
\mathsf{1}_{\mathsf{A}}\brk1{\boldsymbol{x}(\upomega)}
\PP(d\upomega)\nonumber\\
&=\int_{\mathsf{K}\times\HS^{\mathsf{N}}}
\boldsymbol{\mathsf{f}}(\kS,\boldsymbol{\xS})
\mathsf{1}_{\mathsf{A}}\brk{\boldsymbol{\xS}}
\brk1{\PP\circ(k,\boldsymbol{x})^{-1}}
(d\kS,d\boldsymbol{\mathsf{\xS}})\nonumber\\
&=\int_{\mathsf{K}\times\HS^{\mathsf{N}}}
\boldsymbol{\mathsf{f}}(\kS,\boldsymbol{\xS})
\mathsf{1}_{\mathsf{A}}\brk{\boldsymbol{\xS}}
\brk1{(\PP\circ k^{-1})\otimes(\PP\circ\boldsymbol{x}^{-1})}
(d\kS,d\boldsymbol{\mathsf{\xS}})\nonumber\\
&=\int_{\HS^{\mathsf{N}}}
\mathsf{1}_{\mathsf{A}}\brk{\boldsymbol{\xS}}
\brk3{\int_{\mathsf{K}}
\mathsf{\boldsymbol{\mathsf{f}}}(\kS,\boldsymbol{\xS})
(\PP\circ k^{-1})(d\kS)}
(\PP\circ\boldsymbol{x}^{-1})(d\boldsymbol{\xS})\nonumber\\
&=\int_{\HS^{\mathsf{N}}}
\mathsf{1}_{\mathsf{A}}\brk{\boldsymbol{\xS}}
\brk3{\int_{\mathsf{K}}
\mathsf{f}(\kS,\xS_1)(\PP\circ k^{-1})(d\kS)}
(\PP\circ\boldsymbol{x}^{-1})(d\boldsymbol{\xS})\nonumber\\
&=\int_{\HS^{\mathsf{N}}}
\mathsf{1}_{\mathsf{A}}\brk{\boldsymbol{\xS}}
\mathsf{g}(\xS_1)
(\PP\circ\boldsymbol{x}^{-1})(d\boldsymbol{\xS})\nonumber\\
&=\int_{\upOmega}
\mathsf{1}_{\mathsf{A}}\brk1{\boldsymbol{x}(\upomega)}
\mathsf{g}\brk1{x_1(\upomega)}\PP(d\upomega)\nonumber\\
&=\int_{\mathsf{S}}
\mathsf{g}\brk1{x_1(\upomega)}\PP(d\upomega).
\end{align} 
Therefore
$\mathsf{g}(x_1)=\EC{\mathsf{f}(k,x_1)}{\upsigma(\boldsymbol{x})}$
\Pas
\end{proof}

\begin{lemma}
\label{l:3}
Let $\mathsf{p}\in\left]1,\pinf\right[$, let 
$(\xi_{\nS})_{\nnn}$ be a sequence in 
$L^\mathsf{p}(\upOmega,\FE,\PP;\RR)$ such that 
$\sup_{\nnn}\EE|\xi_{\nS}|^\mathsf{p}<\pinf$, and let
$\xi\colon\upOmega\to\RR$ be measurable. Suppose that 
$\xi_{\nS}\to\xi\;\Pas$ Then $\EE|\xi|<\pinf$, 
$\xi_{\nS}\to\xi$ in $L^1(\upOmega,\FE,\PP;\RR)$, and 
$\EE\xi_{\nS}\to\EE\xi$.
\end{lemma}
\begin{proof}
Set $\mathsf{q}=(\mathsf{p}-1)/\mathsf{p}$. 
It follows from the H\"older and Markov inequalities that
\begin{align}
0&\leq
\lim_{\upbeta\to\pinf}\sup_{\nnn}
\int_{[\abs{\xi_{\nS}}\geq\upbeta]}\abs{\xi_{\nS}}d\PP\nonumber\\
&\leq\lim_{\upbeta\to\pinf}\sup_{\nnn}\brk2{
\EE^{1/\mathsf{p}}\abs{\xi_{\nS}}^{\mathsf{p}}
\EE^{1/\mathsf{q}}
\mathsf{1}_{[\abs{\xi_{\nS}}\geq\upbeta]}^{\mathsf{q}}}
\nonumber\\
&\leq\sup_{\nnn}\EE^{1/\mathsf{p}}\abs{\xi_{\nS}}^\mathsf{p}
\lim_{\upbeta\to\pinf}\sup_{\nnn}
\brk2{\PP\brk1{[\abs{\xi_{\nS}}\geq\upbeta]}}^\mathsf{1/\mathsf{q}}
\nonumber\\
&\leq\sup_{\nnn}\EE^{1/\mathsf{p}}\abs{\xi_{\nS}}^\mathsf{p}
\lim_{\upbeta\to\pinf}\sup_{\nnn}
\frac{\EE^{1/\mathsf{q}}\abs{\xi_{\nS}}^{\mathsf{p}}}
{\upbeta^{\mathsf{p}/\mathsf{q}}}\nonumber\\
&=0,
\end{align}
which shows that 
$(\xi_{\nS})_{\nnn}$ is uniformly integrable. We can therefore
invoke \cite[Theorem~2.6.4(b)]{Shir16}, which 
asserts that $\xi\in L^1(\upOmega,\FE,\PP;\RR)$, 
$\EE\xi_{\nS}\to\EE\xi$, and 
$\xi_{\nS}\to\xi$ in $L^1(\upOmega,\FE,\PP;\RR)$.
\end{proof}

\begin{lemma}
\textup{\cite[Proposition~2.6.31]{Hyto16}}
\label{l:2}
Let $x\in L^2(\upOmega,\FE,\PP;\HS)$, let $\XX$ be a sub 
$\upsigma$-algebra of $\FE$, and let 
$y\in L^2(\upOmega,\XX,\PP;\HS)$. Then
$\EC1{\scal{x}{y}_{\HS}}{\XX}
=\scal1{\EC{x}{\XX}}{y}_{\HS}$.
\end{lemma}

\begin{lemma}
\label{l:1}
Let $\CS$ be a nonempty closed subset of $\HS$ and let
$(x_{\nS})_{\nnn}$ be a sequence of $\HS$-valued random
variables. Define 
\begin{equation}
\mathfrak{X}=(\XX_{\nS})_{\nnn},\;\;\text{where}\;\;
(\forall\nnn)\;\;\XX_{\nS}=\upsigma(x_{0},\dots,x_{\nS}).
\end{equation}
Let $\mathsf{p}\in\left[1,\pinf\right[$ and suppose that, for every
$\mathsf{z}\in\CS$, there exist 
$(\mu_{\nS}(\mathsf{z}))_{\nnn}\in\ell^1_+(\mathfrak{X})$, 
$(\theta_{\nS}(\mathsf{z}))_{\nnn}\in\ell_+(\mathfrak{X})$, and 
$(\nu_{\nS}(\mathsf{z}))_{\nnn}\in\ell^1_+(\mathfrak{X})$ 
such that 
\begin{equation}
\label{e:l1}
(\forall\nnn)\quad
\EC1{\|x_{\nS+1}-\mathsf{z}\|_{\HS}^{\mathsf{p}}}{\XX_{\nS}}
+\theta_{\nS}(\mathsf{z}) \leq\big(1+
\mu_{\nS}(\mathsf{z})\big)
\|x_{\nS}-\mathsf{z}\|_{\HS}^{\mathsf{p}}+
\nu_{\nS}(\mathsf{z})\;\;\Pas
\end{equation}
Then the following hold:
\begin{enumerate}
\item
\label{l:1i}
Let $\mathsf{z}\in\CS$. Then 
$\sum_{\nnn}\theta_{\nS}(\mathsf{z})<\pinf\;\Pas$
\item
\label{l:1ii}
$(\|x_{\nS}\|_{\HS})_{\nnn}$ is bounded $\Pas$
\item
\label{l:1ii2} 
$\mathfrak{W}(x_{\nS})_{\nnn}\neq\varnothing\;\Pas$
\item
\label{l:1iii}
There exists $\upOmega'\in\FE$ such that
$\PP(\upOmega')=1$ and, for every $\upomega\in
\upOmega'$ and every $\mathsf{z}\in\CS$, 
$(\|x_{\nS}(\upomega)-\mathsf{z}\|_{\HS})_{\nnn}$ converges.
\item
\label{l:1iv}
Suppose that $\mathfrak{W}(x_{\nS})_{\nnn}\subset\CS\;\Pas$ 
Then $(x_{\nS})_{\nnn}$ converges weakly $\Pas$ to a 
$\CS$-valued random variable.
\item
\label{l:1v}
Suppose that $\mathfrak{S}(x_{\nS})_{\nnn}\cap\CS\neq
\varnothing\;\Pas$ Then $(x_{\nS})_{\nnn}$ converges strongly
$\Pas$ to a $\CS$-valued random variable.
\item
\label{l:1vi}
Suppose that $\mathfrak{S}(x_{\nS})_{\nnn}\neq\emp\;\Pas$ and
that $\mathfrak{W}(x_{\nS})_{\nnn}\subset\CS\;\Pas$ Then 
$(x_{\nS})_{\nnn}$ converges strongly $\Pas$ to an 
$\CS$-valued random variable.
\item
\label{l:1vii}
Suppose that $\mathsf{z}\in\CS$ and $(\upchi_{\nS})_{\nnn}$ 
in $\RP$ satisfy
\begin{equation}
\label{e:l2}
(\forall\nnn)\quad
\EC1{\|x_{\nS+1}-\mathsf{z}\|_{\HS}^{\mathsf{p}}}{\XX_{\nS}}
\leq
\upchi_{\nS}\|x_{\nS}-\mathsf{z}\|_{\HS}^{\mathsf{p}}
\;\Pas\;\;\text{and}\;\;\varlimsup\upchi_{\nS}<1.
\end{equation}
Then the following hold:
\begin{enumerate}
\item
Let $\nnn$. Then
$\EC{\|x_{\nS+1}-\mathsf{z}\|_{\HS}^{\mathsf{p}}}{\XX_{0}}
\leq(\prod_{\jS=0}^{\nS}
\upchi_{\jS})\|x_{0}-\mathsf{z}\|_{\HS}^{\mathsf{p}}\;
\Pas$
\item
Suppose that 
$x_{0}\in L^{\mathsf{p}}(\upOmega,\FE,\PP;\HS)$. 
Then $(x_{\nS})_{\nnn}$ converges strongly in 
$L^{\mathsf{p}}(\upOmega,\FE,\PP;\HS)$ and $\Pas$ to 
$\mathsf{z}$.
\end{enumerate}
\end{enumerate}
\end{lemma}
\begin{proof}
\ref{l:1i}-\ref{l:1vi}: Apply \cite[Proposition~2.3]{Siop15} with 
$\phi=|\cdot|^\mathsf{p}$. The measurability of the weak limit in
\ref{l:1iv} relies on \cite[Proposition~2.3(iv)]{Siop15} which
invokes \cite[Corollary~1.13]{Pett38}. The applicability 
of the latter follows from the separability of $\HS$ and the
completeness of $(\upOmega,\FE,\PP)$; see
\cite[Sections~1.1a--b]{Hyto16} for details. 

\ref{l:1vii}: Apply
\cite[Lemma~2.2]{MaPr19} with $\phi=|\cdot|^{\mathsf{p}}$.
\end{proof}

\section{Main results}
\label{sec:3}

\subsection{An abstract stochastic algorithm}
\label{sec:31}

The analysis of the asymptotic behavior of the following
algorithm will serve as the backbone of subsequent convergence
results. We recall from Section~\ref{sec:1} that $\ZS$ is the
solution set of the problem under consideration and that it is
assumed to be nonempty and closed. 

\begin{algorithm}
\label{algo:3}
Let $x_{0}\in L^2(\upOmega,\FE,\PP;\HS)$. Iterate
\begin{equation}
\label{e:a1}
\begin{array}{l}
\text{for}\;\nS=0,1,\ldots\\
\left\lfloor
\begin{array}{l}
\XX_{\nS}=\upsigma(x_{0},\dots,x_{\nS})\\
\text{take}\;
\lambda_{\nS}\in L^\infty(\upOmega,\FE,\PP;\RPP),\;
d_{\nS}\in L^2(\upOmega,\FE,\PP;\HS),\;\text{and}\;
\delta_{\nS}\in\mathfrak{C}(\upOmega,\XX_{\nS},\PP;\HS)\;\text{such
that}\\
\qquad\begin{cases}
\EC1{\lambda_{\nS}(2-\lambda_{\nS})\|d_{\nS}\|_{\HS}^2}
{\XX_{\nS}}\geq 0\;\Pas;\\
(\forall\mathsf{z}\in\ZS)\;\;
\EC1{\lambda_{\nS}
\scal{\mathsf{z}+d_{\nS}-x_{\nS}}{d_{\nS}}_{\HS}}{\XX_{\nS}}
\leq\delta_{\nS}(\cdot,\mathsf{z})/2\;\Pas\\
\end{cases}\\
x_{\nS+1}=x_{\nS}-\lambda_{\nS} d_{\nS}.
\end{array}
\right.\\
\end{array}
\end{equation}
\end{algorithm}

Let us outline the weak and strong convergence properties of
Algorithm~\ref{algo:3}.

\begin{theorem}
\label{t:3}
Let $(x_{\nS})_{\nnn}$ be the sequence generated by
Algorithm~\ref{algo:3}. Then the following hold:
\begin{enumerate}
\item
\label{t:3i-}
$(x_{\nS})_{\nnn}$ is a well-defined sequence in 
$L^2(\upOmega,\FE,\PP;\HS)$.
\item
\label{t:3i}
Let $\nnn$ and $\mathsf{z}\in\ZS$. Then
\begin{equation}
\nonumber
\EC1{\|x_{\nS+1}-\mathsf{z}\|_{\HS}^2}{\XX_{\nS}}
\leq\|x_{\nS}-\mathsf{z}\|_{\HS}^2
-\EC1{\lambda_{\nS}(2-\lambda_{\nS})\|d_{\nS}\|_{\HS}^2}{\XX_{\nS}}
+\delta_{\nS}(\cdot,\mathsf{z})\;\;\Pas
\end{equation}
\item
\label{t:3ii}
Let $\nnn$ and $z\in L^2(\upOmega,\XX_{\nS},\PP;\ZS)$. Then
\begin{equation}
\nonumber
\EC1{\|x_{\nS+1}-z\|_{\HS}^2}{\XX_{\nS}}
\leq\|x_{\nS}-z\|_{\HS}^2
-\EC1{\lambda_{\nS}(2-\lambda_{\nS})\|d_{\nS}\|_{\HS}^2}{\XX_{\nS}}
+\delta_{\nS}(\cdot,z)\;\;\Pas
\end{equation}
\item
\label{t:3v}
Let $\nnn$ and $z\in L^2(\upOmega,\XX_{\nS},\PP;\ZS)$. Then
\begin{equation}
\nonumber
\|x_{\nS+1}-z\|_{L^2(\upOmega,\FE,\PP;\HS)}^2
\leq\|x_{\nS}-z\|_{L^2(\upOmega,\FE,\PP;\HS)}^2
-\EE\brk1{\lambda_{\nS}(2-\lambda_{\nS})\|d_{\nS}\|_{\HS}^2}
+\EE\delta_{\nS}(\cdot,z).
\end{equation}
\item
Suppose that, for every $\mathsf{z}\in\ZS$, 
$\sum_{\nnn}\delta_{\nS}(\cdot,\mathsf{z})<\pinf\;\Pas$ 
Then the following hold:
\begin{enumerate}
\item
\label{t:3iva}
$(\|x_{\nS}\|_{\HS})_{\nnn}$ is bounded $\Pas$
\item
\label{t:3ivb}
Let $z$ be a $\ZS$-valued random variable. Then
$(\|x_{\nS}-z\|_{\HS})_{\nnn}$ converges $\Pas$
\item
\label{t:3ivd}
$\sum_{\nnn}\EC{\lambda_{\nS}(2-\lambda_{\nS})
\|d_{\nS}\|_{\HS}^2}{\XX_{\nS}}<\pinf\;\Pas$
\item
\label{t:3ivf}
Suppose that $\mathfrak{W}(x_{\nS})_{\nnn}\subset\ZS\;\Pas$ 
Then $(x_{\nS})_{\nnn}$ converges weakly $\Pas$ to a 
$\ZS$-valued random variable.
\item
\label{t:3ivg}
Suppose that 
$\mathfrak{S}(x_{\nS})_{\nnn}\cap\ZS\neq\emp\;\Pas$ 
Then $(x_{\nS})_{\nnn}$ converges strongly $\Pas$ to a 
$\ZS$-valued random variable.
\item
\label{t:3ivh}
Suppose that $\mathfrak{S}(x_{\nS})_{\nnn}\neq\emp\;\Pas$
and that $\mathfrak{W}(x_{\nS})_{\nnn}\subset\ZS\;\Pas$ Then 
$(x_{\nS})_{\nnn}$ converges strongly $\Pas$ to a $\ZS$-valued 
random variable.
\end{enumerate}
\item
Suppose that, for every 
$\zS\in\ZS$,
$\sum_{\nnn}\EE\delta_{\nS}(\cdot,\zS)<\pinf$.
Then the following hold:
\begin{enumerate}
\item
\label{t:3viia}
$(\|x_{\nS}\|_{L^2(\upOmega,\FE,\PP;\HS)})_{\nnn}$ is bounded.
\item
\label{t:3csii}
Let $z\in L^2(\upOmega,\FE,\PP;\ZS)$. Then 
$(\|x_{\nS}-z\|_{L^1(\upOmega,\FE,\PP;\HS)})_{\nnn}$ converges.
\item
\label{t:3viic}
$\sum_{\nnn}\EE\brk{\lambda_{\nS}\left(2-\lambda_{\nS}\right)
{\|d_{\nS}\|_{\HS}^2}}<\pinf$. 
\item
\label{t:3weak}
Suppose that $(x_{\nS})_{\nnn}$ converges weakly $\Pas$ to an
$\HS$-valued random variable $x$. Then 
$x\in L^2(\upOmega,\FE,\PP;\HS)$ and $(x_{\nS})_{\nnn}$ converges
weakly in $L^2(\upOmega,\FE,\PP;\HS)$ to $x$.
\item
\label{t:3viig}
Let $x$ be a $\ZS$-valued random variable. Then
$(x_{\nS})_{\nnn}$ converges strongly $\Pas$ to $x$ if and only if 
$(x_{\nS})_{\nnn}$ converges strongly in 
$L^1(\upOmega,\FE,\PP;\HS)$ to $x$. In this case, 
$x\in L^2(\upOmega,\FE,\PP;\ZS)$ and
$(x_{\nS})_{\nnn}$ converges weakly in $L^2(\upOmega,\FE,\PP;\HS)$ 
to $x$.
\end{enumerate}
\end{enumerate}
\end{theorem}
\begin{proof}
\ref{t:3i-}: By assumption, 
$x_{0}\in L^2(\upOmega,\FE,\PP;\HS)$.
Now suppose that $x_{\nS}\in L^2(\upOmega,\FE,\PP;\HS)$. Then,
since $d_{\nS}\in L^2(\upOmega,\FE,\PP;\HS)$ and
$\lambda_{\nS}\in L^\infty(\upOmega,\FE,\PP;\RPP)$, 
$x_{\nS+1}=x_{\nS}-\lambda_{\nS} d_{\nS}\in 
L^2(\upOmega,\FE,\PP;\HS)$. This establishes the claim 
by induction. 

\ref{t:3i}: 
We derive from \eqref{e:a1} that
\begin{align}
\label{e:3.2}
&\quad
\EC1{\|x_{\nS+1}-\mathsf{z}\|_{\HS}^2}{\XX_{\nS}}\nonumber\\ 
&\qquad=\EC1{\|x_{\nS}-\mathsf{z}\|_{\HS}^2-2\lambda_{\nS}
\scal{x_{\nS}-\mathsf{z}}{d_{\nS}}_{\HS}+
\lambda_{\nS}^2\|d_{\nS}\|_{\HS}^2}
{\XX_{\nS}}\nonumber\\
&\qquad=\|x_{\nS}-\mathsf{z}\|_{\HS}^2-\EC1{\lambda_{\nS}(2-
\lambda_{\nS})\|d_{\nS}\|_{\HS}^2}{\XX_{\nS}}+2\EC1{\lambda_{\nS}
\scal{\mathsf{z}+d_{\nS}-x_{\nS}}{d_{\nS}}_{\HS}}
{\XX_{\nS}}\nonumber\\
&\qquad\leq\|x_{\nS}-\mathsf{z}\|_{\HS}^2-\EC1{\lambda_{\nS}(2-
\lambda_{\nS})\|d_{\nS}\|_{\HS}^2}{\XX_{\nS}}+
\delta_{\nS}(\cdot,\mathsf{z})\;\;\Pas
\end{align}

\ref{t:3ii}: 
First, let $s$ be a $\ZS$-valued $\XX_{\nS}$-simple mapping. Then
there exists a finite family of disjoint sets 
$(\mathsf{F}_{\iS})_{\iS\in\mathsf{I}}$ in 
$\XX_{\nS}$ such that $\bigcup_{\iS\in\mathsf{I}} 
\mathsf{F}_{\iS}=\upOmega$, and a family 
$(\mathsf{z}_{\iS})_{\iS\in\mathsf{I}}$ in 
$\ZS$ such that $s=\sum_{\iS\in\mathsf{I}}
\mathsf{1}_{\mathsf{F}_{\iS}}\mathsf{z}_{\iS}$. Then,
by \ref{t:3i},
\begin{align}
\EC1{\|x_{\nS+1}-s\|_{\HS}^2}{\XX_{\nS}}&=
\EC3{\bigg\|\sum_{\iS\in\mathsf{I}}
\mathsf{1}_{\mathsf{F}_{\iS}}
(x_{\nS+1}-\mathsf{z}_{\iS})\bigg\|_{\HS}^2}
{\XX_{\nS}}\nonumber\\ 
&=\EC3{\sum_{\iS\in\mathsf{I}}
\|x_{\nS+1}-\mathsf{z}_{\iS}
\|_{\HS}^2\mathsf{1}_{\mathsf{F}_{\iS}}}{\XX_{\nS}}
\nonumber\\
&=\sum_{\iS\in\mathsf{I}}
\EC1{\|x_{\nS+1}-\mathsf{z}_{\iS}\|_{\HS}^2}
{\XX_{\nS}}\mathsf{1}_{\mathsf{F}_{\iS}}\nonumber\\
&\leq\sum_{\iS\in\mathsf{I}}
\|x_{\nS}-\mathsf{z}_{\iS}\|_{\HS}^2
\mathsf{1}_{\mathsf{F}_{\iS}}+\sum_{\iS\in\mathsf{I}}
\brk2{-\EC1{\lambda_{\nS}(2-\lambda_{\nS})\|d_{\nS}\|_{\HS}^2}
{\XX_{\nS}}+\delta_{\nS}(\cdot,\mathsf{z}_{\iS})}
\mathsf{1}_{\mathsf{F}_{\iS}}
\nonumber\\
&=\bigg\|\sum_{\iS\in\mathsf{I}}
\mathsf{1}_{\mathsf{F}_{\iS}}
(x_{\nS}-\mathsf{z}_{\iS})\bigg\|_{\HS}^2-\EC1{\lambda_{\nS}
(2-\lambda_{\nS})\|d_{\nS}\|_{\HS}^2}{\XX_{\nS}}+
\sum_{\iS\in\mathsf{I}}\delta_{\nS}(\cdot,\mathsf{z}_{\iS})
\mathsf{1}_{\mathsf{F}_{\iS}} 
\nonumber\\
&=\|x_{\nS}-s\|_{\HS}^2-\EC1{\lambda_{\nS}(2-
\lambda_{\nS})\|d_{\nS}\|_{\HS}^2}{\XX_{\nS}}+
\delta_{\nS}(\cdot,s)\;\;\Pas
\label{e:3.11}
\end{align}
Next, Proposition~\ref{p:seq.simple} guarantees the existence of 
a sequence of $\ZS$-valued $\XX_{\nS}$-simple mappings
$(s_{\jS})_{\jS\in\NN}$ such that $s_{\jS}\to z\;\Pas$ and
$\sup_{\jS\in\NN}\|s_{\jS}\|^2_{\HS}\leq\|z\|^2_{\HS}+1\;\Pas$
Thus, we derive from \eqref{e:3.11} that 
\begin{align}
(\forall\jS\in\NN)\quad
\EC1{\|x_{\nS+1}-s_{\jS}\|_{\HS}^2}{\XX_{\nS}}
&\leq\|x_{\nS}-s_{\jS}\|_{\HS}^2
-\EC1{\lambda_{\nS}(2-\lambda_{\nS})\|d_{\nS}\|_{\HS}^2}{\XX_{\nS}}
+\delta_{\nS}(\cdot,s_{\jS})\;\;\Pas\label{e:3.4}
\end{align}
Additionally,
\begin{equation}
(\forall\jS\in\NN)\quad
\|x_{\nS+1}-s_{\jS}\|_{\HS}^2
\leq 2\|x_{\nS+1}\|_{\HS}^2+2\|s_{\jS}\|_{\HS}^2
\leq 2\|x_{\nS+1}\|_{\HS}^2+2\|z\|_{\HS}^2+2\;\;\Pas
\label{e:3.5}
\end{equation}
Note that the right-hand term in \eqref{e:3.5} is integrable and
that $\|x_{\nS+1}-s_{\jS}\|_{\HS}^2\to
\|x_{\nS+1}-z\|_{\HS}^2\;\Pas$ as $\jS\to\pinf$. 
Therefore, by the conditional dominated convergence 
theorem \cite[Theorem~2.7.2(a)]{Shir16}, 
\begin{equation}
\EC{\|x_{\nS+1}-s_{\jS}\|_{\HS}^2}{\XX_{\nS}}\to 
\EC{\|x_{\nS+1}-z\|_{\HS}^2}{\XX_{\nS}}\;\Pas 
\;\;\text{as}\;\;\jS\to\pinf.
\end{equation}
On the other hand, the continuity of $\delta_{\nS}$ with respect to
the $\HS$-variable implies that
$\delta_{\nS}(\cdot,s_{\jS})\to\delta_{\nS}(\cdot,z)\;\Pas$ as
$\jS\to\pinf$.
Altogether, taking the limit as $\jS\to\pinf$ in \eqref{e:3.4}
yields 
\begin{equation}
\label{e:ref37}
\EC1{\|x_{\nS+1}-z\|_{\HS}^2}{\XX_{\nS}}
\leq\|x_{\nS}-z\|_{\HS}^2
-\EC1{\lambda_{\nS}(2-\lambda_{\nS})\|d_{\nS}\|_{\HS}^2}{\XX_{\nS}}
+\delta_{\nS}(\cdot,z)\;\;\Pas
\end{equation}

\ref{t:3v}: Take the expected value in \ref{t:3ii}.

\ref{t:3iva}: This follows from \ref{t:3i} and 
Lemma~\ref{l:1}\ref{l:1ii}.

\ref{t:3ivb}: Let $\upOmega''\in\FE$ be such that
$\PP(\upOmega'')=1$
and, for every $\upomega\in\upOmega''$, $z(\upomega)\in\ZS$. 
Further, let $\upOmega'\in\FE$ be given as in
Lemma~\ref{l:1}\ref{l:1iii}, which holds as a consequence of 
\ref{t:3i}. Then
\begin{equation}
\brk1{\forall\upomega\in\upOmega'\cap\upOmega''}\quad 
\brk1{\|x_{\nS}(\upomega)-z(\upomega)\|_{\HS}}_{\nnn}\;
\text{converges,}
\end{equation}
which confirms that
$(\|x_{\nS}-z\|_{\HS})_{\nnn}$ converges $\Pas$ since 
$\PP(\upOmega'\cap\upOmega'')=1$. 

\ref{t:3ivd}: Let $\mathsf{z}\in\ZS$. In view of
\ref{t:3i} and Lemma~\ref{l:1}\ref{l:1i},
\begin{equation}
\label{e:3.6}
\sum_{\nnn}\EC1{\lambda_{\nS}(2-\lambda_{\nS})
\|d_{\nS}\|_{\HS}^2}{\XX_{\nS}}<\pinf\;\Pas
\end{equation}

\ref{t:3ivf}--\ref{t:3ivh}: These follow from \ref{t:3i} and 
Lemma~\ref{l:1}\ref{l:1iv}--\ref{l:1vi}.

\ref{t:3viia}: Note that
$\{\emp,\upOmega\}\subset\bigcap_{\nnn}\XX_{\nS}$. 
It follows from \ref{t:3v} and the assumption that,
for every $\zS\in\ZS$, 
$\sum_{\nnn}\EE\delta_{\nS}(\cdot,\zS)<\pinf$,
that $(x_{\nS})_{\nnn}$ is quasi-Fej\'er of Type III in
$L^2(\upOmega,\FE,\PP;\HS)$ relative to 
$L^2(\upOmega,\{\emp,\upOmega\},\PP;\ZS)$ 
\cite[Definition~1.1]{Else01}.
Hence, \cite[Proposition~3.3(i)]{Else01} implies
that $(x_{\nS})_{\nnn}$ is bounded in $L^2(\upOmega,\FE,\PP;\HS)$.

\ref{t:3csii}: It follows from \ref{t:3viia} that 
$\sup_{\nnn}\EE\|x_{\nS}-z\|_{\HS}^2<\pinf$ and from \ref{t:3ivb} 
that $(\|x_{\nS}-z\|_{\HS})_{\nnn}$ converges $\Pas$
We then invoke Lemma~\ref{l:3} to deduce that 
$\EE\|x_{\nS}-z\|_{\HS}\to\EE\brk1{\lim\|x_{\nS}-z\|_{\HS}}<\pinf$.

\ref{t:3viic}: Let $\zS\in\ZS$. 
Then, in view of \ref{t:3v} and Corollary~\ref{l:5}\ref{l:5i},
\begin{equation}
\sum_{\nnn}\EE\brk1{\lambda_{\nS}(2-\lambda_{\nS})
\norm{d_{\nS}}_{\HS}^2}<\pinf.
\end{equation}

\ref{t:3weak}: In view of \ref{t:3viia}, 
$(x_{\nS})_{\nnn}$ possesses a weak sequential cluster point
$w\in L^2(\upOmega,\FE,\PP;\HS)$, i.e., there exists a strictly 
increasing sequence $({\kS_{\nS}})_{\nnn}$ in $\NN$ such that 
$x_{\kS_{\nS}}\weakly w$ in $L^2(\upOmega,\FE,\PP;\HS)$. 
However, since $\HS$ is separable, it contains a 
countable dense set $\{\mathsf{y}_{\jS}\}_{\jjj}$. Let us fix
temporarily $\jjj$ and identify $\mathsf{y}_{\jS}$ with a constant
random variable in $L^2(\upOmega,\FE,\PP;\HS)$. Then
$\EE\scal{x_{\kS_{\nS}}-w}{\mathsf{y}_{\jS}}_{\HS}\to 0$ and we can
therefore extract a further subsequence 
$(x_{\mathsf{l}_{\kS_{\nS}}})_{\nnn}$ such that
$\scal{x_{\mathsf{l}_{\kS_{\nS}}}-w}
{\mathsf{y}_{\jS}}_{\HS}\to 0\;\Pas$
On the other hand, the assumption yields 
$\scal{x_{\mathsf{l}_{\kS_{\nS}}}-x}
{\mathsf{y}_{\jS}}_{\HS}\to 0\;\Pas$ 
We deduce from the $\PP$-almost sure uniqueness of
the limit that there exists $\upOmega_{\jS}\in\FE$ such that
$\PP(\upOmega_{\jS})=1$ and 
\begin{equation}
\label{e:kj1}
(\forall\upomega\in\upOmega_{\jS})\quad
\scal{x(\upomega)}{\mathsf{y}_{\jS}}_{\HS}
=\scal{w(\upomega)}{\mathsf{y}_{\jS}}_{\HS}.
\end{equation}
Let us set $\upOmega''=\bigcap_{\jjj}\upOmega_{\jS}$
and note that $\PP(\upOmega'')=1$. 
Then \eqref{e:kj1} yields
\begin{equation}
\label{e:3.7}
(\forall\jjj)(\forall\upomega\in\upOmega'')
\quad\scal{x(\upomega)-w(\upomega)}{\mathsf{y}_{\jS}}_{\HS}=0.
\end{equation}
By density, for every $\upomega\in\upOmega''$, there exists
a strictly increasing sequence $(\iS_{\jS})_{\jjj}$ such that
$\mathsf{y}_{\iS_{\jS}}\to{x(\upomega)-w(\upomega)}$ and it results
from \eqref{e:3.7} that 
\begin{equation}
\|x(\upomega)-w(\upomega)\|_{\HS}^2=
\scal{x(\upomega)-w(\upomega)}{x(\upomega)-w(\upomega)}_{\HS}=0,
\end{equation}
which shows that $x(\upomega)=w(\upomega)$. Thus, $x=w\;\Pas$ and
it follows from \cite[Lemma~2.46]{Livre1} that
$x_{\nS}\weakly x$ in $L^2(\upOmega,\FE,\PP;\HS)$.

\ref{t:3viig}:
Suppose that $x_{\nS}\to x\;\Pas$ Then it follows from 
\ref{t:3viia} and Lemma~\ref{l:3} that 
$x\in L^1(\upOmega,\FE,\PP;\ZS)$ and $x_{\nS}\to x$ in
$L^1(\upOmega,\FE,\PP;\HS)$. Conversely, suppose that 
$x\in L^1(\upOmega,\FE,\PP;\ZS)$ and $x_{\nS}\to x$ in
$L^1(\upOmega,\FE,\PP;\HS)$, i.e., $\EE\|x_{\nS}-x\|_{\HS}\to 0$.
Then there exists a strictly increasing sequence
$(\kS_{\nS})_{\nnn}$ in $\NN$ such that
$\|x_{\kS_{\nS}}-x\|_{\HS}\to 0\;\Pas$ On the other hand, 
\ref{t:3ivb} guarantees that $(\|x_{\nS}-x\|_{\HS})_{\nnn}$ 
converges $\Pas$ Since the $\PP$-almost sure limit of any
subsequence coincides with the $\PP$-almost sure limit of the
sequence, we conclude that $\|x_{\nS}-x\|_{\HS}\to0\;\Pas$
Additionally, as $\PP$-almost sure strong convergence implies
$\PP$-almost sure weak convergence, we deduce from \ref{t:3weak}
that $x\in L^2(\upOmega,\FE,\PP;\HS)$ and $x_{\nS}\weakly x$ in
$L^2(\upOmega,\FE,\PP;\HS)$.
\end{proof}

We now assume that the tolerance variables $(\delta_{\nS})_{\nnn}$
are constant with respect to the $\HS$-variable and depend only on
the $\upOmega$-variable.

\begin{theorem}
\label{t:3.5}
Let $(x_{\nS})_{\nnn}$ be the sequence generated by
Algorithm~\ref{algo:3}. For every $\nnn$, assume that
$\delta_{\nS}$ is constant with respect to the $\HS$-variable and
set, for every $\upomega\in\upOmega$,
$\vartheta_{\nS}(\upomega)=\delta_{\nS}(\upomega,0)$.
Then the following hold:
\begin{enumerate}
\item
\label{t:3iii}
Let $\nnn$. Then $\EC{\dS_{\ZS}^2(x_{\nS+1})}{\XX_{\nS}}\leq 
\dS_{\ZS}^2(x_{\nS})+\vartheta_{\nS}\;\Pas$
\item
\label{t:3vi}
Let $\nnn$. Then 
$\EE{\dS_{\ZS}^2(x_{\nS+1})}\leq\EE \dS_{\ZS}^2(x_{\nS})
+\EE\vartheta_{\nS}$.
\item
\label{t:3ivc}
Suppose that
$\sum_{\nnn}\vartheta_{\nS}<\pinf\;\Pas$ Then
$(\dS_{\ZS}(x_{\nS}))_{\nnn}$ converges $\Pas$
\item
Suppose that
$\sum_{\nnn}\EE\vartheta_{\nS}<\pinf$.
Then the following hold:
\begin{enumerate}
\item
\label{t:3viib}
$(\EE\dS_{\ZS}^2(x_{\nS}))_{\nnn}$ converges.
\item
\label{t:3viih}
Suppose that $\ZS$ is convex and that
$\varliminf\EE{\dS_{\ZS}^2(x_{\nS})}=0$. Then $(x_{\nS})_{\nnn}$
converges strongly in $L^2(\upOmega,\FE,\PP;\HS)$ and $\Pas$ to a 
$\ZS$-valued random variable.
\item
\label{t:3viiI}
Suppose that $\ZS$ is convex and that there exists 
$\upchi\in\zeroun$ such that 
\begin{equation}
\label{e:kj2}
(\forall\nnn)\quad\EC1{\dS_{\ZS}^2(x_{\nS+1})}{\XX_{\nS}}
\leq\upchi\dS_{\ZS}^2(x_{\nS})+\vartheta_{\nS}\;\;\Pas
\end{equation}
Then the following are satisfied:
\begin{enumerate}[label=\normalfont{[\Alph*]}]
\item
\label{t:3viiI1}
Let $\nnn$. Then $\EE\dS_{\ZS}^2(x_{\nS+1})
\leq\upchi^{\nS+1}\EE\dS_{\ZS}^2(x_{0})
+\sum_{\jS=0}^{\nS}\upchi^{\nS-\jS}\EE\vartheta_{\jS}$.
\item
\label{t:3viiI2}
There exists $x\in L^2(\upOmega,\FE,\PP;\ZS)$ such 
that $(x_{\nS})_{\nnn}$ converges strongly in 
$L^2(\upOmega,\FE,\PP;\HS)$ and $\Pas$ to $x$, and
\begin{equation}
\label{e:3.1}
(\forall\nnn)\quad \EE\|x_{\nS}-x\|_{\HS}^2
\leq4\upchi^{\nS}\EE\dS^2_{\ZS}(x_{0})
+4\sum_{\jS=0}^{\nS-1}\upchi^{\nS-\jS-1}\EE\vartheta_{\jS}
+2\sum_{\jS\geq\nS}\EE\vartheta_{\jS}.
\end{equation}
\end{enumerate}
\end{enumerate}
\end{enumerate}
\end{theorem}
\begin{proof}
\ref{t:3iii}: Let $\zS\in\ZS$. Then Theorem~\ref{t:3}\ref{t:3i} 
yields $\EC{\|x_{\nS+1}-\zS\|_{\HS}^2}{\XX_{\nS}}
\leq\|x_{\nS}-\zS\|_{\HS}^2+\vartheta_{\nS}\;\Pas$ 
On the other hand,
$\dS_{\ZS}(x_{\nS+1})\leq\|x_{\nS+1}-\zS\|_{\HS}\;\Pas$ Thus,
\begin{equation}
\EC1{\dS_{\ZS}^2(x_{\nS+1})}{\XX_{\nS}}
\leq\EC1{\|x_{\nS+1}-\zS\|_{\HS}^2}{\XX_{\nS}}
\leq\|x_{\nS}-\zS\|_{\HS}^2+\vartheta_{\nS}\;\Pas
\end{equation}
Taking the infimum over $\zS\in\ZS$ yields the claim.

\ref{t:3vi}: Take the expected value in \ref{t:3iii}.

\ref{t:3ivc}: This follows from \ref{t:3iii} and
Lemma~\ref{l:4}\ref{l:4i}.

\ref{t:3viib}: This follows from \ref{t:3vi} and 
Corollary~\ref{l:5}\ref{l:5i}.

\ref{t:3viih}: 
Let $\nnn$, $\mathsf{m}\in\NN\smallsetminus\{0\}$, 
and $z\in L^2(\upOmega,\XX_{\nS},\PP;\ZS)$. Then
$z\in\bigcap_{1\leq\jS\leq\mathsf{m}}
L^2(\upOmega,\XX_{\nS+\jS},\PP;\HS)$ and we derive inductively from
\eqref{e:a1} and Theorem~\ref{t:3}\ref{t:3ii} that
\begin{align}
\EC1{\|x_{\nS}-x_{\nS+\mathsf{m}}\|_{\HS}^2}{\XX_{\nS}}
&\leq 2\EC2{\|x_{\nS}-z\|_{\HS}^2+
\|x_{\nS+\mathsf{m}}-z\|_{\HS}^2}{\XX_{\nS}}
\nonumber\\
&\leq 2\|x_{\nS}-z\|_{\HS}^2+2\EC2{
\EC1{\|x_{\nS+\mathsf{m}}-z\|_{\HS}^2}
{\XX_{\nS+\mathsf{m}-1}}}{\XX_{\nS}}
\nonumber\\
&\leq 4\|x_{\nS}-z\|_{\HS}^2+2\sum_{\jS=\nS}^
{\nS+\mathsf{m}-1}\vartheta_{\jS}\;\;\Pas 
\label{e:3.8}
\end{align}
Now assume that $z=\proj_{\ZS}x_{\nS}$ and recall that
$\proj_{\ZS}$ is nonexpansive \cite[Proposition~4.16]{Livre1}
while $x_{\nS}$ is $(\XX_{\nS},\BE_{\HS})$-measurable. 
Consequently, $z$ is $(\XX_{\nS},\BE_{\HS})$-measurable.
Given $y\in L^2(\upOmega,\XX_{\nS},\PP;\ZS)$,
\begin{multline}
\frac{1}{2}\EE\norm{z}^2_{\HS}
=\frac{1}{2}\EE\norm{z-y+y}^2_{\HS}
\leq\EE\norm{\proj_{\ZS}x_{\nS}-\proj_{\ZS}y}^2_{\HS}
+\EE\norm{y}^2_{\HS}\\
\leq\norm{x_{\nS}-y}^2_{L^2(\upOmega,\XX_{\nS},\PP;\ZS)}+
\norm{y}^2_{L^2(\upOmega,\XX_{\nS},\PP;\ZS)},
\end{multline}
which shows that $z\in L^2(\upOmega,\XX_{\nS},\PP;\ZS)$.
Further,
\eqref{e:3.8} yields
\begin{equation}
\label{e:3.9}
\EC1{\|x_{\nS}-x_{\nS+\mathsf{m}}\|_{\HS}^2}{\XX_{\nS}}
\leq4\dS_{\ZS}^2(x_{\nS})+2\sum_{\jS=\nS}^
{\nS+\mathsf{m}-1}{\vartheta_{\jS}}\;\;\Pas 
\end{equation}
Therefore, upon taking expectations, we get
\begin{equation}
\label{e:3.10}
\EE{\|x_{\nS}-x_{\nS+\mathsf{m}}\|_{\HS}^2}
\leq 4\EE\dS_{\ZS}^2(x_{\nS})+2\sum_{\jS=\nS}^
{\nS+\mathsf{m}-1}\EE\vartheta_{\jS}. 
\end{equation}
The assumption $\varliminf \EE\dS_{\ZS}^2(x_{\nS})=0$ and 
\ref{t:3viib} yield $\lim\EE\dS_{\ZS}^2(x_{\nS})=0$. 
In addition, 
\begin{equation}
(\forall\mathsf{m}\in\NN\smallsetminus\{0\})\quad
0\leq\sum_{\jS=\nS}^{\nS+\mathsf{m}-1}
\EE\vartheta_{\jS}
\leq\sum_{\jS\geq\nS}\EE\vartheta_{\jS}\to0\;\text{as}\;\nS\to\pinf.
\end{equation}
We thus infer from \eqref{e:3.10} that
$(x_{\nS})_{\nnn}$ is a Cauchy sequence in 
$L^2(\upOmega,\FE,\PP;\HS)$, which implies that there exists
$x\in L^2(\upOmega,\FE,\PP;\HS)$ such that $x_{\nS}\to x$ in
$L^2(\upOmega,\FE,\PP;\HS)$. Further, since
$\dS_{\ZS}^2\colon\HS\to\RP$ is continuous, 
$\dS_{\ZS}^2(x_{\nS})\to\dS_{\ZS}^2(x)\;\Pas$ In addition, 
it follows from Fatou's lemma that
\begin{equation}
0\leq \EE\dS_{\ZS}^2(x)\leq\varliminf\EE\dS_{\ZS}^2(x_{\nS})=0.
\end{equation}
Hence $\EE\dS_{\ZS}^2(x)=0$, $\dS_{\ZS}^2(x)=0\;\Pas$, and 
$x\in\ZS\;\Pas$ 
Finally, Theorem~\ref{t:3}\ref{t:3viig} yields 
$x_{\nS}\to x\;\Pas$ 

\ref{t:3viiI}:

\ref{t:3viiI1}: Taking expectations in \eqref{e:kj2} yields
$\EE\dS_{\ZS}^2(x_{\nS+1})\leq\upchi
\EE\dS_{\ZS}^2(x_{\nS})+\EE\vartheta_{\nS}$. The claim follows by
induction. 

\ref{t:3viiI2}: It follows from 
Corollary~\ref{l:5}\ref{l:5ii} that 
$\lim\EE\dS_{\ZS}^2(x_{\nS})=0$. Therefore, 
\ref{t:3viih} implies that $(x_{\nS})_{\nnn}$ converges 
strongly in $L^2(\upOmega,\FE,\PP;\HS)$ and $\Pas$ to a
$\ZS$-valued random variable. Finally,
arguing as in \cite[Theorem~3.13(ii)]{Else01}, we obtain 
\eqref{e:3.1}.
\end{proof}

\subsection{A stochastic algorithm with super
relaxations}
\label{sec:32}

We study an implementation of Algorithm~\ref{algo:2} in which the
standard condition that the relaxations are deterministic and
bounded above by $2$ is not imposed. In Section~\ref{sec:1} we
called such relaxations super relaxations. 

\begin{algorithm}
\label{algo:4}
In Algorithm~\ref{algo:2} assume that, for every $\nnn$, 
$\varepsilon_{\nS}\in\mathfrak{C}(\upOmega,\XX_{\nS},\PP;\HS)$,
$\lambda_{\nS}$ is independent of 
$\upsigma(\{x_{0},\ldots,x_{\nS},d_{\nS}\})$, and
$\EE(\lambda_{\nS}(2-\lambda_{\nS}))\geq 0$.
\end{algorithm}

\begin{proposition}
\label{p:1}
Algorithm~\ref{algo:4} is a special case of 
Algorithm~\ref{algo:3} where, for every $\nnn$, 
$\delta_{\nS}=2\varepsilon_{\nS}\EE\lambda_{\nS}$. 
\end{proposition}
\begin{proof}
Let $(x_{\nS})_{\nnn}$ be the sequence generated by 
Algorithm~\ref{algo:4}. Let us first show by induction that it is a
well-defined sequence in $L^2(\upOmega,\FE,\PP;\HS)$. By 
assumption, $x_{0}\in L^2(\upOmega,\FE,\PP;\HS)$.
Fix $\nnn$ and note that $d_{\nS}$ is measurable as a combination 
of measurable functions. Additionally, \eqref{e:a2} yields
\begin{align}
\dfrac{1}{2}\EE\|d_{\nS}\|_{\HS}^2
&=\dfrac{1}{2}\EE\|\alpha_{\nS}t_{\nS}^*\|_{\HS}^2
\nonumber\\
&\leq\dfrac{1}{2}\EE\Biggl\|\dfrac{\mathsf{1}_{[t_{\nS}^*\neq0]}
\mathsf{1}_{
\left[\scal{x_{\nS}}{t_{\nS}^*}_{\HS}>\eta_{\nS}\right]}
\bigl(\scal{x_{\nS}}{t_{\nS}^*}_{\HS}-\eta_{\nS}\bigr)}
{\|t_{\nS}^*\|_{\HS}^2+\mathsf{1}_{[t_{\nS}^*=0]}}\,
t_{\nS}^*\Biggr\|_{\HS}^2\nonumber\\ 
&=\dfrac{1}{2}\EE\Biggl|\dfrac{\mathsf{1}_{[t_{\nS}^*\neq0]}
\mathsf{1}_{\left[\scal{x_{\nS}}{t_{\nS}^*}_{\HS}>\eta_{\nS}\right]
}\bigl(\scal{x_{\nS}}{t_{\nS}^*}_{\HS}-\eta_{\nS}\bigr)}
{\|t_{\nS}^*\|_{\HS}+\mathsf{1}_{[t_{\nS}^*=0]}}\Biggr|^2
\nonumber\\
&\leq\EE\Biggl|\dfrac{\mathsf{1}_{[t_{\nS}^*\neq0]}\mathsf{1}_{
\left[\scal{x_{\nS}}{t_{\nS}^*}_{\HS}>\eta_{\nS}\right]}
\scal{x_{\nS}}{t_{\nS}^*}_{\HS}}
{\|t_{\nS}^*\|_{\HS}+\mathsf{1}_{[t_{\nS}^*=0]}}\Biggr|^2
+\EE\Biggl|\dfrac{\mathsf{1}_{[t_{\nS}^*\neq0]}\mathsf{1}_{
\left[\scal{x_{\nS}}{t_{\nS}^*}_{\HS}>\eta_{\nS}\right]}\eta_{\nS}}
{\|t_{\nS}^*\|_{\HS}+\mathsf{1}_{[t_{\nS}^*=0]}}\Biggr|^2
\nonumber\\
&\leq\EE\Biggl|\dfrac{\|x_{\nS}\|_{\HS}\,\|t_{\nS}^*\|_{\HS}}
{\|t_{\nS}^*\|_{\HS}+\mathsf{1}_{[t_{\nS}^*=0]}}\Biggr|^2
+\EE\Biggl|\dfrac{\mathsf{1}_{[t_{\nS}^*\neq0]}\mathsf{1}_{
\left[\scal{x_{\nS}}{t_{\nS}^*}_{\HS}>\eta_{\nS}\right]}\eta_{\nS}}
{\|t_{\nS}^*\|_{\HS}+\mathsf{1}_{[t_{\nS}^*=0]}}\Biggr|^2
\nonumber\\
&\leq\EE\|x_{\nS}\|_{\HS}^2+\EE\Biggl|
\dfrac{\mathsf{1}_{[t_{\nS}^*\neq0]}\mathsf{1}_{
\left[\scal{x_{\nS}}{t_{\nS}^*}_{\HS}>\eta_{\nS}\right]}\eta_{\nS}}
{\|t_{\nS}^*\|_{\HS}+\mathsf{1}_{[t_{\nS}^*=0]}}\Biggr|^2
\nonumber\\
&<\pinf.
\label{e:999}
\end{align}
Thus, $d_{\nS}\in L^2(\upOmega,\FE,\PP;\HS)$ and, since 
$\lambda_{\nS}\in L^\infty(\upOmega,\FE,\PP;\RPP)$, 
$x_{\nS+1}=x_{\nS}-\lambda_{\nS}d_{\nS}\in 
L^2(\upOmega,\FE,\PP;\HS)$, which completes the induction argument.
The fact that $\lambda_{\nS}\in L^\infty(\upOmega,\FE,\PP;\RPP)$
also guarantees the integrability of $\lambda_{\nS}$ and 
$\lambda_{\nS}(2-\lambda_{\nS})$. Further, since $\lambda_{\nS}$ is
independent of
$\upsigma(\{x_{0},\ldots,x_{\nS},d_{\nS}\})$ and
$\EE\brk{\lambda_{\nS}(2-\lambda_{\nS})}\geq0$, it follows from
Lemma~\ref{l:6} that
\begin{equation}
\EC1{\lambda_{\nS}(2-\lambda_{\nS})\norm{d_{\nS}}_{\HS}^2}
{\XX_{\nS}}
=\EE\brk1{\lambda_{\nS}(2-\lambda_{\nS})}
\EC1{\norm{d_{\nS}}_{\HS}^2}{\XX_{\nS}}
\geq 0\;\;\Pas
\end{equation}
Next, we infer from \eqref{e:a2} that
\begin{equation}
\label{e:ye7}
(\forall\nnn)\quad\alpha_{\nS}\eta_{\nS}
=\scal{x_{\nS}}{\alpha_{\nS}t_{\nS}^*}_{\HS}
-\alpha_{\nS}^2\|t_{\nS}^*\|_{\HS}^2
=\scal{x_{\nS}}{d_{\nS}}_{\HS}-\|d_{\nS}\|_{\HS}^2,
\end{equation}
which shows that 
$\alpha_{\nS}\eta_{\nS}\in L^1(\upOmega,\FE,\PP;\RR)$.
Now set $\delta_{\nS}=2\varepsilon_{\nS}
\EE\lambda_{\nS}\in\mathfrak{C}(\upOmega,\XX_{\nS},\PP;\HS)$ and 
let $\mathsf{z}\in\ZS$. Then we deduce from \eqref{e:a2}, 
Lemma~\ref{l:2}, and \eqref{e:ye7} that
\begin{align}
\label{e:ye9}
(\forall\nnn)\quad
\EC1{\scal{\mathsf{z}}{d_{\nS}}_{\HS}}{\XX_{\nS}}
&=\scal{\mathsf{z}}{\EC{\alpha_{\nS}t_{\nS}^*}{\XX_{\nS}}}_{\HS}
\nonumber\\
&\leq\EC{\alpha_{\nS}\eta_{\nS}}{\XX_{\nS}}+\varepsilon_{\nS}
(\cdot,\mathsf{z})\nonumber\\
&=\EC1{\scal{x_{\nS}}{d_{\nS}}_{\HS}
-\|d_{\nS}\|_{\HS}^2}{\XX_{\nS}}+\varepsilon_{\nS}
(\cdot,\mathsf{z})\;\;\Pas
\end{align}
Finally, we derive from \eqref{e:ye9} and Lemma~\ref{l:6} that
\begin{align}
\label{e:t129}
(\forall\nnn)\quad
\EC1{\lambda_{\nS}\scal{\mathsf{z}+d_{\nS}-x_{\nS}}{d_{\nS}}_{\HS}}
{\XX_{\nS}}
&=
\EC1{\scal{\mathsf{z}+d_{\nS}-x_{\nS}}{d_{\nS}}_{\HS}}{\XX_{\nS}}
\EE\lambda_{\nS}\nonumber\\
&=\EC1{\scal{\mathsf{z}}{d_{\nS}}_{\HS}
+\|d_{\nS}\|_{\HS}^2-\scal{x_{\nS}}{d_{\nS}}_{\HS}}{\XX_{\nS}}
\EE\lambda_{\nS}\nonumber\\
&\leq\varepsilon_{\nS}(\cdot,\mathsf{z})\EE\lambda_{\nS}
\nonumber\\
&=\dfrac{\delta_{\nS}(\cdot,\mathsf{z})}{2}\;\;\Pas,
\end{align}
which yields the claim. 
\end{proof}

The asymptotic behavior of Algorithm~\ref{algo:4} is our next
topic. We leverage Proposition~\ref{p:1} and Theorems~\ref{t:3} and
\ref{t:3.5} to obtain the following properties.

\begin{theorem}
\label{t:2}
Let $(x_{\nS})_{\nnn}$ be the sequence generated by
Algorithm~\ref{algo:4}. 
\begin{enumerate}
\item
Suppose that, for every $\mathsf{z}\in\ZS$,
$\sum_{\nnn}\varepsilon_{\nS}(\cdot,\mathsf{z})
\EE\lambda_{\nS}<\pinf\;\Pas$ 
Then the following hold:
\begin{enumerate}
\item
\label{t:2ivd}
$\sum_{\nnn}\EE(\lambda_{\nS}\left(2-\lambda_{\nS}\right))
\EC{\|d_{\nS}\|_{\HS}^2}{\XX_{\nS}}<\pinf\;\Pas$
\item
\label{t:2ive}
Suppose that 
$\inf_{\nnn}\EE(\lambda_{\nS}(2-\lambda_{\nS}))>0$ and
there exists $\uprho\in\left[1,\pinf\right[$ such that 
$\sup_{\nnn}\lambda_{\nS}<\uprho$ $\Pas$ Then
$\sum_{\nnn}\EC{\|x_{\nS+1}-x_{\nS}\|_{\HS}^2}{\XX_{\nS}}<\pinf\;
\Pas$
and $\sum_{\nnn}\|x_{\nS+1}-x_{\nS}\|_{\HS}^2<\pinf\;\Pas$
\item
\label{t:2ivf}
Suppose that $\mathfrak{W}(x_{\nS})_{\nnn}\subset\ZS\;\Pas$ 
Then $(x_{\nS})_{\nnn}$ converges weakly $\Pas$ to a 
$\ZS$-valued random variable.
\item
\label{t:2ivg}
Suppose that 
$\mathfrak{S}(x_{\nS})_{\nnn}\cap\ZS\neq\emp\;\Pas$ 
Then $(x_{\nS})_{\nnn}$ converges strongly $\Pas$ to a
$\ZS$-valued random variable.
\item
\label{t:2ivh}
Suppose that $\mathfrak{S}(x_{\nS})_{\nnn}\neq\emp\;\Pas$
and that $\mathfrak{W}(x_{\nS})_{\nnn}\subset\ZS\;\Pas$ Then 
$(x_{\nS})_{\nnn}$ converges strongly $\Pas$ to a $\ZS$-valued 
random variable.
\end{enumerate}
\item
Suppose that, for every 
$\zS\in\ZS$,
$\sum_{\nnn}\EE\varepsilon_{\nS}(\cdot,\zS)\EE\lambda_{\nS}<\pinf$.
Then the following hold:
\begin{enumerate}
\item
\label{t:2viic}
$\sum_{\nnn}\EE(\lambda_{\nS}(2-\lambda_{\nS}))
\EE{\|d_{\nS}\|_{\HS}^2}<\pinf$.
\item
\label{t:2viie}
Suppose that 
$\inf_{\nnn}\EE(\lambda_{\nS}(2-\lambda_{\nS}))>0$ and
there exists $\uprho\in\left[1,\pinf\right[$ such that 
$\sup_{\nnn}\lambda_{\nS}<\uprho$ $\Pas$ Then
$\sum_{\nnn}\EE{\|x_{\nS+1}-x_{\nS}\|_{\HS}^2}<\pinf$.
\item
\label{t:2weak}
Suppose that $\mathfrak{W}(x_{\nS})_{\nnn}\subset\ZS\;\Pas$ 
Then $(x_{\nS})_{\nnn}$ converges weakly $\Pas$ and weakly in
$L^2(\upOmega,\FE,\PP;\HS)$ to a random variable
$x\in L^2(\upOmega,\FE,\PP;\ZS)$.
\item
\label{t:2viig}
Suppose that $\mathfrak{S}(x_{\nS})_{\nnn}\cap\ZS\neq\emp\;\Pas$ 
Then $(x_{\nS})_{\nnn}$ converges strongly $\Pas$ and strongly in
$L^1(\upOmega,\FE,\PP;\HS)$ to a random variable 
$x\in L^2(\upOmega,\FE,\PP;\ZS)$. Additionally, 
$(x_{\nS})_{\nnn}$ converges weakly in $L^2(\upOmega,\FE,\PP;\HS)$
to $x$.
\item
\label{t:2viih}
Suppose that $\ZS$ is convex, that, for every $\nnn$,
$\varepsilon_{\nS}$ is constant with respect to the $\HS$-variable,
and that $\varliminf\EE{\dS_{\ZS}^2(x_{\nS})}=0$. Then 
$(x_{\nS})_{\nnn}$ converges strongly in 
$L^2(\upOmega,\FE,\PP;\HS)$ and $\Pas$ to a $\ZS$-valued random 
variable.
\item
\label{t:2viiI}
Suppose that $\ZS$ is convex, that, for every $\nnn$,
$\varepsilon_{\nS}$ is constant with respect to the $\HS$-variable,
and that there exists $\upchi\in\zeroun$ such that 
\begin{equation}
(\forall\nnn)\quad\EC1{\dS_{\ZS}^2(x_{\nS+1})}{\XX_{\nS}}
\leq\upchi \dS_{\ZS}^2(x_{\nS})+2\varepsilon_{\nS}\EE\lambda_{\nS}
\;\;\Pas
\end{equation}
Set, for every $\nnn$ and for every $\upomega\in\upOmega$,
$\vartheta_{\nS}(\upomega)=\varepsilon_{\nS}(\upomega,0)$.
Then the following are satisfied:
\begin{enumerate}[label=\normalfont{[\Alph*]}]
\item
\label{t:2viiI1}
Let $\nnn$. Then $\EE\dS_{\ZS}^2(x_{\nS+1})\leq\upchi^{\nS+1}\EE
\dS_{\ZS}^2(x_{0})+2\sum_{\jS=0}^{\nS}\upchi^{\nS-\jS}
\EE\vartheta_{\jS}\EE\lambda_{\jS}$.
\item
\label{t:2viiI2}
There exists $x\in L^2(\upOmega,\FE,\PP;\ZS)$ such 
that $(x_{\nS})_{\nnn}$ converges strongly in 
$L^2(\upOmega,\FE,\PP;\HS)$ and $\Pas$ to $x$, and
\begin{equation}
(\forall\nnn)\quad \EE\|x_{\nS}-x\|_{\HS}^2\leq4\upchi^{\nS}
\EE\dS^2_{\ZS}(x_{0})+
8\sum_{\jS=0}^{\nS-1}\upchi^{\nS-\jS-1}
\EE\vartheta_{\jS}\EE\lambda_{\jS}+4\sum_{\jS\geq\nS}
\EE\vartheta_{\jS}\EE\lambda_{\jS}.
\end{equation}
\end{enumerate}
\end{enumerate}
\end{enumerate}
\end{theorem}
\begin{proof}
In view of Proposition~\ref{p:1}, we appeal to Theorems~\ref{t:3}
and \ref{t:3.5} to establish the claims.

\ref{t:2ivd}: It follows from Theorem~\ref{t:3}\ref{t:3ivd} and
Lemma~\ref{l:6} that
\begin{equation}
\sum_{\nnn}\EE\brk1{\lambda_{\nS}\left(2-\lambda_{\nS}\right)}
\EC1{\|d_{\nS}\|_{\HS}^2}{\XX_{\nS}}
=\sum_{\nnn}
\EC1{\lambda_{\nS}\left(2-\lambda_{\nS}\right)\|d_{\nS}\|_{\HS}^2}
{\XX_{\nS}}<\pinf\;\Pas
\end{equation}

\ref{t:2ivd}$\Rightarrow$\ref{t:2ive}: It follows from
\eqref{e:a2} that 
\begin{equation}
\label{e:2233}
\sum_{\nnn}\EE\brk1{\lambda_{\nS}(2-\lambda_{\nS})}
\EC3{\frac{1}{\lambda_{\nS}^2}\|x_{\nS+1}-x_{\nS}\|_{\HS}^2}
{\XX_{\nS}}
=\sum_{\nnn}\EE\brk1{\lambda_{\nS}(2-\lambda_{\nS})}
\EC1{\|d_{\nS}\|_{\HS}^2}{\XX_{\nS}}
<\pinf\;\Pas
\end{equation}
Hence, the assumption 
$\inf_{\nnn}\EE\brk1{\lambda_{\nS}(2-\lambda_{\nS})}>0$ yields 
$\sum_{\nnn}\EC1{
\|x_{\nS+1}-x_{\nS}\|_{\HS}^2/\lambda_{\nS}^2}{\XX_{\nS}}
<\pinf\;\Pas$ Further, 
\begin{equation}
(\forall\nnn)\quad0<\frac{1}{\uprho^2}
\leq\frac{1}{\lambda_{\nS}^2}\;\;\Pas
\end{equation} 
Thus,
\begin{equation}
\label{e:654}
\sum_{\nnn}\EC1{\|x_{\nS+1}-x_{\nS}\|_{\HS}^2}{\XX_{\nS}}
<\pinf\;\Pas
\end{equation} 
In addition, 
\begin{equation}
(\forall\nnn)\quad
\EC4{\sum_{\kS=0}^{\nS+1}\norm{x_{\kS+1}-x_{\kS}}_{\HS}^2}
{\XX_{\nS+1}}
=\sum_{\kS=0}^{\nS}\norm{x_{\kS+1}-x_{\kS}}_{\HS}^2
+\EC1{\norm{x_{\nS+2}-x_{\nS+1}}_{\HS}^2}{\XX_{\nS+1}}\;\;\Pas
\end{equation}
It then follows from \eqref{e:654} and Lemma~\ref{l:4}\ref{l:4i}
that $(\sum_{\kS=0}^{\nS}\norm{x_{\kS+1}-x_{\kS}}_{\HS}^2)_{\nnn}$ 
converges $\Pas$ to a $\RP$-valued random variable, hence 
$\sum_{\nnn}\norm{x_{\nS+1}-x_{\nS}}_{\HS}^2<\pinf\;\Pas$

\ref{t:2ivf}--\ref{t:2ivh}: These follow from 
Theorem~\ref{t:3}\ref{t:3ivf}--\ref{t:3ivh}.

\ref{t:2viic}: It follows from Theorem~\ref{t:3}\ref{t:3viic} and
Lemma~\ref{l:6} that
\begin{align}
\sum_{\nnn}\EE\brk1{\lambda_{\nS}(2-\lambda_{\nS})}
\EE\norm{d_{\nS}}_{\HS}^2
&=\sum_{\nnn}\EE\brk1{\lambda_{\nS}(2-\lambda_{\nS})}
\EE\brk2{\EC1{\norm{d_{\nS}}_{\HS}^2}{\XX_{\nS}}}\nonumber\\
&=\sum_{\nnn}\EE\brk2{\EE\brk1{\lambda_{\nS}(2-\lambda_{\nS})}
\EC1{\norm{d_{\nS}}_{\HS}^2}{\XX_{\nS}}}\nonumber\\
&=\sum_{\nnn}\EE\brk2{\EC2{\lambda_{\nS}(2-\lambda_{\nS})
\norm{d_{\nS}}_{\HS}^2}{\XX_{\nS}}}\nonumber\\
&=\sum_{\nnn}\EE\brk2{\lambda_{\nS}(2-\lambda_{\nS})
\norm{d_{\nS}}_{\HS}^2}\nonumber\\
&<\pinf.
\end{align}
Hence $\sum_{\nnn}\EE\brk{\lambda_{\nS}(2-\lambda_{\nS})}
\EE\norm{d_{\nS}}_{\HS}^2<\pinf$. 

\ref{t:2viic}$\Rightarrow$\ref{t:2viie}:
It follows from \eqref{e:a1} that 
\begin{equation}
\label{e:2234}
\sum_{\nnn}\EE\brk1{\lambda_{\nS}(2-\lambda_{\nS})}\EE\brk3{
\frac{1}{\lambda_{\nS}^2}\|x_{\nS+1}-x_{\nS}\|_{\HS}^2}
=\sum_{\nnn}\EE\brk1{\lambda_{\nS}(2-\lambda_{\nS})}
\EE\|d_{\nS}\|_{\HS}^2<\pinf.
\end{equation}
Thus, as in \ref{t:2ive}, 
$\sum_{\nnn}\EE\|x_{\nS+1}-x_{\nS}\|_{\HS}^2<\pinf$. 

\ref{t:2weak}--\ref{t:2viig}: These follow from
\ref{t:2ivf}--\ref{t:2ivg} and
Theorem~\ref{t:3}\ref{t:3weak}--\ref{t:3viig}.

\ref{t:2viih}--\ref{t:2viiI}: These follow from
Theorem~\ref{t:3.5}\ref{t:3viih}--\ref{t:3viiI}.
\end{proof}

\subsection{A stochastic algorithm with random
relaxations bounded by 2}
\label{sec:33}

We present an implementation of Algorithm~\ref{algo:2} with an 
alternative relaxation strategy.

\begin{algorithm}
\label{algo:5}
In Algorithm~\ref{algo:2}, for every $\nnn$, 
$\varepsilon_{\nS}\in\mathfrak{C}(\upOmega,\XX_{\nS},\PP;\HS)$ and
$\lambda_{\nS}\in 
L^\infty(\upOmega,\XX_{\nS},\PP;\left]0,2\right[)$.
\end{algorithm}

\begin{proposition}
\label{p:2}
Algorithm~\ref{algo:5} is a special case of 
Algorithm~\ref{algo:3} where, for every $\nnn$,
$\delta_{\nS}=2\lambda_{\nS}\varepsilon_{\nS}$.
\end{proposition}
\begin{proof}
Set $(\forall\nnn)$ $\delta_{\nS}=2\lambda_{\nS}\varepsilon_{\nS}$.
Following the proof of Proposition~\ref{p:1}, it is enough to show
that 
\begin{equation}
(\forall\nnn)\quad
\begin{cases}
\delta_{\nS}\in\mathfrak{C}(\upOmega,\XX_{\nS},\PP;\HS);\\[2mm]
\EC1{\lambda_{\nS}(2-\lambda_{\nS})\norm{d_{\nS}}_{\HS}^2}
{\XX_{\nS}}\geq 0\;\;\Pas;\\[2mm]
(\forall\mathsf{z}\in\ZS)\;\;
\EC1{\lambda_{\nS}\scal{\mathsf{z}+d_{\nS}-x_{\nS}}{d_{\nS}}_{\HS}}
{\XX_{\nS}}
\leq\delta_{\nS}(\cdot,\mathsf{z})/2\;\;\Pas
\end{cases}
\end{equation}
Let $\nnn$. It follows from the positivity and measurability of
$\lambda_{\nS}$, as well as the fact that 
$\varepsilon_{\nS}\in\mathfrak{C}(\upOmega,\XX_{\nS},\PP;\HS)$,
that $\delta_{\nS}\in\mathfrak{C}(\upOmega,\XX_{\nS},\PP;\HS)$.
Next, since $\lambda_{\nS}\in\left]0,2\right[\;\Pas$, we have 
$\lambda_{\nS}(2-\lambda_{\nS})>0\;\Pas$ and hence
\begin{equation}
\EC1{\lambda_{\nS}(2-\lambda_{\nS})\norm{d_{\nS}}_{\HS}^2}
{\XX_{\nS}}\geq0\;\;\Pas
\end{equation}
Finally, let $\mathsf{z}\in\ZS$. It then follows from 
\eqref{e:ye9} and the fact that 
$\lambda_{\nS}$ is positive and $\XX_{\nS}$-measurable that
\begin{multline}
\EC1{\lambda_{\nS}\scal{\mathsf{z}+d_{\nS}-x_{\nS}}{d_{\nS}}_{\HS}}
{\XX_{\nS}}\\
=\lambda_{\nS}\EC1{\scal{\mathsf{z}}{d_{\nS}}_{\HS}
+\|d_{\nS}\|_{\HS}^2-\scal{x_{\nS}}{d_{\nS}}_{\HS}}{\XX_{\nS}}
\leq\lambda_{\nS}\varepsilon_{\nS}(\cdot,\mathsf{z})
=\dfrac{\delta_{\nS}(\cdot,\mathsf{z})}{2}\;\;\Pas,
\end{multline}
which completes the proof.
\end{proof}

As in Section~\ref{sec:32}, we can derive weak, strong, and
linear convergence results from Theorems~\ref{t:3} and \ref{t:3.5}.
For brevity, we provide below only the weak convergence results 
but, as in Theorem~\ref{t:2}, strong and linear convergence results
can also be obtained.

\begin{theorem}
\label{t:5}
Let $(x_{\nS})_{\nnn}$ be the sequence generated by
Algorithm~\ref{algo:5}. Suppose that, for every $\mathsf{z}\in\ZS$,
$\sum_{\nnn}\lambda_{\nS}\varepsilon_{\nS}(\cdot,\mathsf{z})
<\pinf\;\Pas$
and that $\mathfrak{W}(x_{\nS})_{\nnn}\subset\ZS\;\Pas$ 
Then $(x_{\nS})_{\nnn}$ converges weakly $\Pas$ to a 
$\ZS$-valued random variable $x$. If, in addition, for every 
$\zS\in\ZS$,
$\sum_{\nnn}\EE\brk{\lambda_{\nS}\varepsilon_{\nS}(\cdot,\zS)}
<\pinf$, then
$x\in L^2(\upOmega,\FE,\PP;\HS)$ and $(x_{\nS})_{\nnn}$ converges 
weakly in $L^2(\upOmega,\FE,\PP;\HS)$ to $x$.
\end{theorem}
\begin{proof}
In view of Proposition~\ref{p:2}, the claim follows from
Theorem~\ref{t:3}\ref{t:3ivf} and \ref{t:3}\ref{t:3weak}.
\end{proof}

\section{Randomly relaxed Krasnosel'ski\u\i--Mann iterations}
\label{sec:4}

Let us first recall some definitions about an operator
$\TS\colon\HS\to\HS$ \cite[Chapter~4]{Livre1}. First,
$\TS\colon\HS\to\HS$ is
nonexpansive if it is $1$-Lipschitzian and $\upalpha$-averaged for 
some $\upalpha\in\zeroun$ if $\Id+\upalpha^{-1}(\TS-\Id)$ is 
nonexpansive \cite{Bail78}. On the other hand, 
$\TS$ is $\upbeta$-cocoercive for some $\upbeta\in\RPP$ if 
\begin{equation}
(\forall\xS\in\HS)(\forall\mathsf{y}\in\HS)\quad
\scal{\xS-\mathsf{y}}{\mathsf{T}\xS-\mathsf{T}\mathsf{y}}_{\HS}
\geq\upbeta\norm{\mathsf{T}\xS-\mathsf{T}\mathsf{y}}_{\HS}^2
\end{equation}
and it is firmly nonexpansive if it is $1$-cocoercive.

The Krasnosel'ski\u\i--Mann iterative process is a basic algorithm
to construct fixed points of nonexpansive operators
\cite{Livre1,Dong22,Groe72,Kras55,Mann53,Reic79}. We propose a
study of its asymptotic behavior in a novel environment featuring
random relaxations and stochastic errors.

\begin{theorem}
\label{t:25}
Let $\TS\colon\HS\to\HS$ be a nonexpansive operator such that
$\Fix\TS\neq\emp$ and let 
$x_{0}\in L^2(\upOmega,\FE,\PP;\HS)$. Iterate
\begin{equation}
\label{e:p25}
\begin{array}{l}
\textup{for}\;\nS=0,1,\ldots\\
\left\lfloor
\begin{array}{l}
\textup{take}\;e_{\nS}\in L^2(\upOmega,\FE,\PP;\HS)\;
\textup{and}\;
\mu_{\nS}\in L^\infty(\upOmega,\FE,\PP;\left]0,1\right[)\\
x_{\nS+1}=x_{\nS}+\mu_{\nS}\brk1{\TS x_{\nS}+e_{\nS}-x_{\nS}}.
\end{array}
\right.\\
\end{array}
\end{equation}
Set $(\forall\nnn)$ $\upPhi_{\nS}=\{x_{0},\dots,x_{\nS}\}$
and $\XX_{\nS}=\upsigma(\upPhi_{\nS})$.
Suppose that 
$\sum_{\nnn}\EE\brk{\mu_{\nS}(1-\mu_{\nS})}=\pinf$ and, for every
$\nnn$, $\mu_{\nS}$ is independent of
$\upsigma(\{e_{\nS}\}\cup\upPhi_{\nS})$. Then the following hold 
for some $\Fix\TS$-valued random variable $x$:
\begin{enumerate}
\item
Suppose that 
$\EC{\norm{e_{\nS}}_{\HS}^2}{\XX_{\nS}}\to 0\;\Pas$ and
$\sum_{\nnn}\EE\mu_{\nS}\sqrt{\EC{\norm{e_{\nS}}_{\HS}^2}
{\XX_{\nS}}}<\pinf\;\Pas$ Then the following hold:
\begin{enumerate}
\item
\label{t:25ia-}
$(\TS x_{\nS}-x_{\nS})_{\nnn}$ converges strongly $\Pas$ to $0$.
\item
\label{t:25ia}
$(x_{\nS})_{\nnn}$ converges weakly $\Pas$ to $x$.
\item
\label{t:25ib}
Suppose that $\TS-\Id$ is demiregular at every point in $\Fix\TS$.
Then $(x_{\nS})_{\nnn}$ converges strongly $\Pas$ to $x$.
\end{enumerate}
\item
\label{t:25ii}
Suppose that $\EE\norm{e_{\nS}}^2_{\HS}\to 0$ and
$\sum_{\nnn}\sqrt{\EE\mu_{\nS}^2\EE\norm{e_{\nS}}_{\HS}^2}<\pinf$.
Then the following hold:
\begin{enumerate}
\item
\label{t:25ic-}
$(\TS x_{\nS}-x_{\nS})_{\nnn}$ converges strongly in
$L^1(\upOmega,\FE,\PP;\HS)$ and strongly $\Pas$ to $0$.
\item
\label{t:25ic}
$x\in L^2(\upOmega,\FE,\PP;\Fix\TS)$ and $(x_{\nS})_{\nnn}$
converges weakly in $L^2(\upOmega,\FE,\PP;\HS)$ and weakly $\Pas$
to $x$.
\item
\label{t:25id}
Suppose that $\TS-\Id$ is demiregular at every point in $\Fix\TS$. 
Then $(x_{\nS})_{\nnn}$ converges strongly in
$L^1(\upOmega,\FE,\PP;\HS)$ and strongly $\Pas$ to $x$.
\end{enumerate}
\end{enumerate}
\end{theorem}
\begin{proof}
Let us show that the sequence $(x_{\nS})_{\nnn}$
constructed by \eqref{e:p25} corresponds to a sequence generated by
Algorithm~\ref{algo:4}. To see this, set $\ZS=\Fix\TS$ and observe
that, since $\TS$ is nonexpansive, 
\begin{equation}
\brk1{\forall\nnn}\brk1{\forall z\in L^2(\upOmega,\FE,\PP;\ZS)}
\quad\EE\|\TS x_{\nS}-z\|_{\HS}^2\leq\EE\|x_{\nS}-z\|_{\HS}^2.
\end{equation}
Thus if, for some $\nnn$, $x_{\nS}\in L^2(\upOmega,\FE,\PP;\HS)$,
then $\TS x_{\nS}\in L^2(\upOmega,\FE,\PP;\HS)$ and \eqref{e:p25}
yields $x_{\nS+1}\in L^2(\upOmega,\FE,\PP;\HS)$. This shows by
induction that $(x_{\nS})_{\nnn}$ and $(\TS x_{\nS})_{\nnn}$ lie in
$L^2(\upOmega,\FE,\PP;\HS)$. Let
us define
\begin{equation}
\label{e:cl}
(\forall\nnn)\quad
\begin{cases}
t_{\nS}^*=\dfrac{x_{\nS}-\TS x_{\nS}-e_{\nS}}{2}\in 
L^2(\upOmega,\FE,\PP;\HS);\\[3mm]
\eta_{\nS}=
\dfrac{\norm{x_{\nS}}_{\HS}^2-\norm{\TS
x_{\nS}+e_{\nS}}_{\HS}^2}{4}\in L^1(\upOmega,\FE,\PP;\RR);\\
\alpha_{\nS}=\mathsf{1}_{[t_{\nS}^*\neq 0]};\\
(\forall\zS\in\ZS)\;\varepsilon_{\nS}(\cdot,\zS)
=\dfrac{1}{4}\EC1{\norm{e_{\nS}}_{\HS}^2}{\XX_{\nS}}
+\dfrac{1}{2}
\norm{x_{\nS}-\mathsf{z}}_{\HS}
\sqrt{\EC1{\norm{e_{\nS}}_{\HS}^2}{\XX_{\nS}}};\\
d_{\nS}=t_{\nS}^*;\\
\lambda_{\nS}=2\mu_{\nS}\in\left]0,2\right[\;\Pas
\end{cases}
\end{equation}
Now set $\mathsf{F}=(\TS+\Id)/2$. Since $\TS$ is nonexpansive,
$\mathsf{F}$ is firmly nonexpansive (see
\cite[Proposition~4.4]{Livre1} or
\cite[Proposition~1.11.2]{Goeb84}).
Hence, we deduce from Lemma~\ref{l:2} and \eqref{e:cl} that, for
every $\zS\in\ZS$ and every $\nnn$, 
\begin{align}
\scal1{\mathsf{z}}{\EC{\alpha_{\nS}t_{\nS}^*}{\XX_{\nS}}}_{\HS}
&=\EC3{\scal2{\mathsf{z}}{x_{\nS}-\mathsf{F}x_{\nS}
-\frac{1}{2}e_{\nS}}_{\HS}}{\XX_{\nS}}
\nonumber\\
&=\scal1{\mathsf{z}}{x_{\nS}-\mathsf{F}x_{\nS}}_{\HS}
-\frac{1}{2}\EC1{\scal{\mathsf{z}}{e_{\nS}}_{\HS}}{\XX_{\nS}}
\nonumber\\
&\leq\scal1{\mathsf{F}x_{\nS}}
{x_{\nS}-\mathsf{F}x_{\nS}}_{\HS}
-\frac{1}{2}\EC1{\scal{\mathsf{z}}{e_{\nS}}_{\HS}}{\XX_{\nS}}
\nonumber\\
&=\EC{\alpha_{\nS}\eta_{\nS}}{\XX_{\nS}}
+\frac{1}{4}\EC1{\norm{e_{\nS}}_{\HS}^2}{\XX_{\nS}}
+\frac{1}{2}\EC1{\scal{\TS
x_{\nS}-\mathsf{z}}{e_{\nS}}_{\HS}}{\XX_{\nS}}\nonumber\\
&\leq\EC{\alpha_{\nS}\eta_{\nS}}{\XX_{\nS}}
+\frac{1}{4}\EC1{\norm{e_{\nS}}_{\HS}^2}{\XX_{\nS}}
+\frac{1}{2}\norm{\TS x_{\nS}-\mathsf{z}}_{\HS}
\EC1{\norm{e_{\nS}}_{\HS}}{\XX_{\nS}}\nonumber\\
&\leq\EC{\alpha_{\nS}\eta_{\nS}}{\XX_{\nS}}
+\frac{1}{4}\EC1{\norm{e_{\nS}}_{\HS}^2}{\XX_{\nS}}
+\frac{1}{2}\norm{x_{\nS}-\mathsf{z}}_{\HS}
\EC1{\norm{e_{\nS}}_{\HS}}{\XX_{\nS}}\nonumber\\
&\leq\EC{\alpha_{\nS}\eta_{\nS}}{\XX_{\nS}}+
\varepsilon_{\nS}(\cdot,\zS)\;\;\Pas
\end{align}
Next, it is clear from \eqref{e:cl} that
$\varepsilon_{\nS}\in\mathfrak{C}(\upOmega,\XX_{\nS},\PP;\HS)$.
Furthermore, in view of \eqref{e:cl}, \eqref{e:p25} can be written
as 
\begin{equation}
(\forall\nnn)\quad x_{\nS+1}=x_{\nS}-\lambda_{\nS}d_{\nS}.
\end{equation}
On the other hand, since $(\mu_{\nS})_{\nnn}$ lies almost surely 
in $\zeroun$, we have $\EE(\lambda_{\nS}(2-\lambda_{\nS}))\geq 0$.
Additionally, for every $\nnn$, $\mu_{\nS}$ is independent of 
$\upsigma(\{e_{\nS}\}\cup\upPhi_{\nS})$ and, by \eqref{e:cl}, 
$d_{\nS}$ is $\upsigma(\{e_{\nS}\}\cup\upPhi_{\nS})$-measurable. 
Hence, $\upsigma(\{d_{\nS}\}\cup\upPhi_{\nS})\subset
\upsigma(\{e_{\nS}\}\cup\upPhi_{\nS})$
and $\lambda_{\nS}$ is independent of
$\upsigma(\{d_{\nS}\}\cup\upPhi_{\nS})$. Altogether, 
$(x_{\nS})_{\nnn}$ is a 
sequence generated by Algorithm~\ref{algo:4}. Finally, it follows 
from \eqref{e:p25} and Lemma~\ref{l:6} that
\begin{align}
(\forall\nnn)(\forall\zS\in\ZS)\quad&
\EC{\norm{x_{\nS+1}-\zS}_{\HS}}{\XX_{\nS}}\nonumber\\
&\quad\leq\EC1{(1-\mu_{\nS})\norm{x_{\nS}-\zS}_{\HS}+
\mu_{\nS}\norm{\TS x_{\nS}-\zS}_{\HS}+
\mu_{\nS}\norm{e_{\nS}}_{\HS}}{\XX_{\nS}}\nonumber\\
&\quad =(1-\EE\mu_{\nS})\norm{x_{\nS}-\zS}_{\HS}+
\EE\mu_{\nS}\norm{\TS x_{\nS}-\zS}_{\HS}+
\EE\mu_{\nS}\EC{\norm{e_{\nS}}_{\HS}}{\XX_{\nS}}\nonumber\\
&\quad\leq\norm{x_{\nS}-\zS}_{\HS}+
\EE\mu_{\nS}\EC{\norm{e_{\nS}}_{\HS}}{\XX_{\nS}}
\nonumber\\
&\quad\leq\norm{x_{\nS}-\zS}_{\HS}+
\EE\mu_{\nS}\sqrt{\EC{\norm{e_{\nS}}_{\HS}^2}{\XX_{\nS}}}\;\;\Pas
\label{e:km1}
\end{align}
Further, by invoking the nonexpansiveness of $\TS$ and the
Cauchy--Schwarz inequality, we obtain
\begin{align}
(\forall\nnn)(\forall\zS\in\ZS)\quad
&\norm{x_{\nS+1}-\zS}_{L^2(\upOmega,\FE,\PP;\HS)}\nonumber\\
&\quad=\sqrt{\EE\norm{x_{\nS+1}-\zS}_{\HS}^2}\nonumber\\
&\quad=\sqrt{\EE\norm{(1-\upmu_{\nS})(x_{\nS}-\zS)
+\upmu_{\nS}(\TS x_{\nS}-\TS\zS)+\upmu_{\nS}e_{\nS}}_{\HS}^2}
\nonumber\\
&\quad\leq\sqrt{\EE\bigl|\norm{x_{\nS}-\zS}_{\HS}+
\mu_{\nS}\norm{e_{\nS}}_{\HS}\bigr|^2}\nonumber\\
&\quad=\sqrt{\EE\norm{x_{\nS}-\zS}_{\HS}^2+
2\EE\mu_{\nS}\EE\brk{\norm{x_{\nS}-\zS}_{\HS}\norm{e_{\nS}}_{\HS}}+
\EE\mu_{\nS}^2\EE\norm{e_{\nS}}_{\HS}^2}\nonumber\\
&\quad\leq\sqrt{\EE\norm{x_{\nS}-\zS}_{\HS}^2+
2\sqrt{\EE\mu_{\nS}^2}\sqrt{\EE\norm{x_{\nS}-\zS}_{\HS}^2}
\sqrt{\norm{e_{\nS}}_{\HS}^2}+
\EE\mu_{\nS}^2\EE\norm{e_{\nS}}_{\HS}^2}\nonumber\\
&\quad=\sqrt{\biggl|\sqrt{\EE\norm{x_{\nS}-\zS}_{\HS}^2}+
\sqrt{\EE\mu_{\nS}^2\EE\norm{e_{\nS}}_{\HS}^2}\,\biggr|^2}
\nonumber\\
&\quad=\norm{x_{\nS}-\zS}_{L^2(\upOmega,\FE,\PP;\HS)}+
\sqrt{\EE\mu_{\nS}^2\EE\norm{e_{\nS}}^2_{\HS}}.
\label{e:FejL2}
\end{align}

\ref{t:25ia-}: We derive from \eqref{e:km1} and
Lemma~\ref{l:1}\ref{l:1ii} that $(\norm{x_{\nS}}_{\HS})_{\nnn}$ is
bounded $\Pas$ Hence, for every $\zS\in\ZS$,
$\sum_{\nnn}\norm{x_{\nS}-\zS}_{\HS}\EE\mu_{\nS}
\sqrt{\EC{\norm{e_{\nS}}_{\HS}^2}{\XX_{\nS}}}<\pinf\;\Pas$ On the 
other hand, the assumptions
$\lim\EC{\norm{e_{\nS}}_{\HS}^2}{\XX_{\nS}}=0$ and
$\sum_{\nnn}
\EE\mu_{\nS}\sqrt{\EC{\norm{e_{\nS}}_{\HS}^2}{\XX_{\nS}}}<\pinf\;
\Pas$ yield
\begin{equation}
\sum_{\nnn}
\EE\mu_{\nS}\EC{\norm{e_{\nS}}_{\HS}^2}{\XX_{\nS}}
<\pinf\;\;\Pas
\end{equation}
Therefore, for every $\zS\in\ZS$,
$\sum_{\nnn}
\varepsilon_{\nS}(\cdot,\zS)\EE\lambda_{\nS}
=2\sum_{\nnn}
\varepsilon_{\nS}(\cdot,\zS)\EE\mu_{\nS}<\pinf\;\Pas$
It then follows from 
Theorem~\ref{t:2}\ref{t:2ivd} and the assumption
$\sum_{\nnn}\EE\brk{\mu_{\nS}(1-\mu_{\nS})}=\pinf$ that 
$\varliminf\EC{\norm{d_{\nS}}_{\HS}^2}{\XX_{\nS}}=0\;\Pas$ 
Hence,
\begin{align} 
\label{e:fc}
0&\leq\frac{1}{2}
\varliminf\norm{\TS x_{\nS}-x_{\nS}}_{\HS}^2\nonumber\\
&\leq\varliminf\EC1{\norm{\TS
x_{\nS}+e_{\nS}-x_{\nS}}_{\HS}^2+\norm{e_{\nS}}_{\HS}^2}{\XX_{\nS}}
\nonumber\\
&=\varliminf\EC1{\norm{\TS
x_{\nS}+e_{\nS}-x_{\nS}}_{\HS}^2}{\XX_{\nS}}
+\lim\EC{\norm{e_{\nS}}_{\HS}^2}{\XX_{\nS}}\nonumber\\
&=4\varliminf\EC1{\norm{d_{\nS}}_{\HS}^2}{\XX_{\nS}}
+\lim\EC{\norm{e_{\nS}}_{\HS}^2}{\XX_{\nS}}\nonumber\\
&=0\;\;\Pas
\end{align}
Thus, Lemma~\ref{l:6} implies that, for every $\nnn$,
\begin{align}
&\EC1{\norm{\TS x_{\nS+1}-x_{\nS+1}}_{\HS}}{\XX_{\nS}}\nonumber\\
&\quad=\EC1{\|\TS x_{\nS+1}-\TS x_{\nS}+(1-\mu_{\nS})(\TS 
x_{\nS}-x_{\nS})-\mu_{\nS}e_{\nS}\|_{\HS}}{\XX_{\nS}}\nonumber\\
&\quad\leq\EC1{\norm{\TS x_{\nS+1}-\TS x_{\nS}}_{\HS}}{\XX_{\nS}}
+\EC1{(1-\mu_{\nS})\norm{\TS x_{\nS}-x_{\nS}}_{\HS}}{\XX_{\nS}}
+\EC1{\mu_{\nS}\norm{e_{\nS}}_{\HS}}{\XX_{\nS}}
\nonumber\\
&\quad\leq\EC1{\norm{x_{\nS+1}-x_{\nS}}_{\HS}}{\XX_{\nS}}
+(1-\EE\mu_{\nS})\norm{\TS x_{\nS}-x_{\nS}}_{\HS}
+\EC1{\mu_{\nS}\norm{e_{\nS}}_{\HS}}{\XX_{\nS}}\nonumber\\
&\quad=\EC1{\mu_{\nS}\norm{\TS x_{\nS}+e_{\nS}-x_{\nS}}_{\HS}}
{\XX_{\nS}}
+(1-\EE\mu_{\nS})\norm{\TS x_{\nS}-x_{\nS}}_{\HS}
+\EC1{\mu_{\nS}\norm{e_{\nS}}_{\HS}}{\XX_{\nS}}\nonumber\\
&\quad=\EC1{\mu_{\nS}\norm{\TS x_{\nS}-x_{\nS}}_{\HS}}
{\XX_{\nS}}
+(1-\EE\mu_{\nS})\norm{\TS x_{\nS}-x_{\nS}}_{\HS}
+2\EC1{\mu_{\nS}\norm{e_{\nS}}_{\HS}}{\XX_{\nS}}\nonumber\\
&\quad=\brk1{\EE{\mu_{\nS}}}\norm{\TS x_{\nS}-x_{\nS}}_{\HS}
+(1-\EE\mu_{\nS})\norm{\TS x_{\nS}-x_{\nS}}_{\HS}
+2\EE\mu_{\nS}\EC{\norm{e_{\nS}}_{\HS}}{\XX_{\nS}}\nonumber\\
&\quad\leq\norm{\TS x_{\nS}-x_{\nS}}_{\HS}
+2\EE\mu_{\nS}\sqrt{\EC{\norm{e_{\nS}}_{\HS}^2}{\XX_{\nS}}}
\;\Pas\label{e:tn}
\end{align}
Consequently, Lemma~\ref{l:4}\ref{l:4i} secures the convergence
$\Pas$ of the sequence $(\norm{\TS x_{\nS}-x_{\nS}}_{\HS})_{\nnn}$,
which, in view of \eqref{e:fc}, forces 
\begin{equation}
\label{e:fc1}
\lim\norm{\TS x_{\nS}-x_{\nS}}_{\HS}=0\;\Pas
\end{equation}

\ref{t:25ia}: Let us show that 
$\mathfrak{W}(x_{\nS})_{\nnn}\subset\ZS\;\Pas$ 
Let $\upomega\in\upOmega$ be such that
$\mathfrak{W}\brk{x_{\nS}(\upomega)}_{\nnn}\neq\emp$ and
$\lim\norm{\TS x_{\nS}(\upomega)-x_{\nS}(\upomega)}=0$. Let
$\xS\in\mathfrak{W}(x_{\nS}(\upomega))_{\nnn}$, say
$x_{\kS_{\nS}}(\upomega)\weakly\xS$. The nonexpansiveness of $\TS$ 
implies that $\Id-\TS$ is demiclosed at $0$ 
\cite[Theorem~4.27]{Livre1}. In turn,
$\TS\xS=\xS$ and $\mathfrak{W}(x_{\nS}(\upomega))\subset\ZS$. Since
$\mathfrak{W}\brk{x_{\nS}}_{\nnn}\neq\emp\;\Pas$ and
$\lim\norm{\TS x_{\nS}-x_{\nS}}=0\;\Pas$, we conclude that
$\mathfrak{W}(x_{\nS})_{\nnn}\subset\ZS\;\Pas$ Thus, the claim
follows from Theorem~\ref{t:2}\ref{t:2ivf}. 

\ref{t:25ib}: By \ref{t:25ia-} and \ref{t:25ia}, there exists
$\upOmega'\in\FE$ such that $\PP(\upOmega')=1$ and 
\begin{equation}
\brk1{\forall\upomega\in\upOmega'}\quad 
\TS x_{\nS}(\upomega)-x_{\nS}(\upomega)\to 0\;\;
\text{and}\;\;
x_{\nS}(\upomega)\weakly x(\upomega).
\end{equation}
It then follows from the demiregularity of $\TS-\Id$ that, for \
every $\upomega\in\upOmega'$, $x_{\nS}(\upomega)\to x(\upomega)$. 
Hence, $(x_{\nS})_{\nnn}$ converges strongly $\Pas$ to $x$.

\ref{t:25ic-}: We derive from \eqref{e:FejL2} and
Corollary~\ref{l:5}\ref{l:5i} that 
$(\norm{x_{\nS}}_{L^2(\upOmega,\FE,\PP;\HS)})_{\nnn}$ is bounded. 
Therefore, 
\begin{equation}
\brk1{\forall z\in L^2(\upOmega,\FE,\PP;\HS)}\quad
\sup_{\nnn}\EE\norm{x_{\nS}-z}_{\HS}^2<\pinf.
\end{equation}
In turn, for every $z\in L^2(\upOmega,\FE,\PP;\HS)$,
\begin{equation}
\sum_{\nnn}\EE\mu_{\nS}\sqrt{\EE\norm{x_{\nS}-z}_{\HS}^2}
\sqrt{\EE\norm{e_{\nS}}_{\HS}^2}
\leq\sum_{\nnn}\sqrt{\EE\norm{x_{\nS}-z}_{\HS}^2}
\sqrt{\EE\mu_{\nS}^2\EE\norm{e_{\nS}}_{\HS}^2}
<\pinf.
\end{equation}
On the other hand, since
$\lim\EE\norm{e_{\nS}}_{\HS}^2=0$ and
$\sum_{\nnn}\EE\mu_{\nS}\sqrt{\EE\norm{e_{\nS}}_{\HS}^2}<\pinf$, 
we have
\begin{equation}
\sum_{\nnn}
\EE\mu_{\nS}\EE\norm{e_{\nS}}_{\HS}^2<\pinf.
\end{equation}
Altogether, we deduce that
\begin{equation}
\brk1{\forall\zS\in\ZS}\quad
\sum_{\nnn}\EE\varepsilon_{\nS}(\cdot,\zS)\EE\lambda_{\nS}=
2\sum_{\nnn}\EE\varepsilon_{\nS}(\cdot,\zS)\EE\mu_{\nS}<\pinf,
\end{equation}
which shows in particular that, for every $\zS\in\ZS$,
$\sum_{\nnn}\varepsilon_{\nS}(\cdot,\zS)\EE\lambda_{\nS}<\pinf\;
\Pas$ 
Thus $\lim\norm{\TS x_{\nS}-x_{\nS}}_{\HS}=0\;\Pas$ On the 
other hand, it follows from Theorem~\ref{t:2}\ref{t:2viic} and the 
assumptions that $\varliminf\EE\norm{d_{\nS}}_{\HS}^2=0$. Hence, 
proceeding as in \eqref{e:fc}, we obtain
$\varliminf\EE\norm{\TS x_{\nS}-x_{\nS}}_{\HS}^2=0$.
Moreover, taking 
expectations in \eqref{e:tn} yields
\begin{align}
(\forall\nnn)\quad
\EE\norm{\TS x_{\nS+1}-x_{\nS+1}}_{\HS}
&\leq\EE\norm{\TS x_{\nS}-x_{\nS}}_{\HS}
+2\EE\mu_{\nS}\sqrt{\EE\norm{e_{\nS}}^2_{\HS}}\nonumber\\
&\leq\EE\norm{\TS x_{\nS}-x_{\nS}}_{\HS}
+2\sqrt{\EE\mu_{\nS}^2\EE\norm{e_{\nS}}^2_{\HS}}.
\label{e:tn2}
\end{align}
It then follows from Corollary~\ref{l:5}\ref{l:5i} that
$(\EE\norm{\TS x_{\nS}-x_{\nS}}_{\HS})_{\nnn}$ converges. Since
$\varliminf\EE\norm{\TS x_{\nS}-x_{\nS}}_{\HS}^2=0$, this implies
that $\lim\EE\norm{\TS x_{\nS}-x_{\nS}}_{\HS}^2=0$. Hence,
appealing to \ref{t:25ia-}, $(\TS x_{\nS}-x_{\nS})_{\nnn}$
converges strongly in $L^1(\upOmega,\FE,\PP;\HS)$ and strongly
$\Pas$ to $0$.

\ref{t:25ic}: We deduce weak convergence $\Pas$ by arguing as
in the proof of \ref{t:25ia}, while weak convergence in
$L^2(\upOmega,\FE,\PP;\HS)$ follows from
Theorem~\ref{t:2}\ref{t:2weak}.

\ref{t:25id}: As in the proof of \ref{t:25ib}, it follows from 
\ref{t:25ic} that $(x_{\nS})_{\nnn}$ converges strongly $\Pas$ to
$x$. Further, strong convergence in $L^1(\upOmega,\FE,\PP;\HS)$
follows from Theorem~\ref{t:2}\ref{t:2viig}.
\end{proof}

\begin{remark}
Theorem~\ref{t:25}\ref{t:25ii} extends
\cite[Corollary~2.7]{Siop15}, where the relaxations are only
deterministic and the weak limit is not shown to be in
$L^2(\upOmega,\FE,\PP;\HS)$. Another connected result is
\cite[Theorem~2.8]{Brav24}, which focuses on the finite-dimensional
setting (which implies that demiregularity holds
\cite[Proposition~2.4]{Sico10}) with deterministic relaxations and
the weaker summability condition
$\sum_{\nnn}\upmu_{\nS}\EC{\|e_{\nS}\|_{\HS}}{\XX_{\nS}}<\pinf\;
\Pas$ The case 
of deterministic relaxations and deterministic errors was 
considered in \cite[Theorem~5.5(i)]{Else01}, as an extension of the
classical result error-free result of \cite[Corollary~3]{Groe72}.
\end{remark}

The following application of Theorem~\ref{t:25} concerns
averaged operators. 

\begin{corollary}
\label{c:25}
Let $\upalpha\in\zeroun$, let $\TS\colon\HS\to\HS$ be an 
$\upalpha$-averaged operator such that
$\Fix\TS\neq\emp$, and let 
$x_{0}\in L^2(\upOmega,\FE,\PP;\HS)$. Iterate
\begin{equation}
\label{e:c25}
\begin{array}{l}
\textup{for}\;\nS=0,1,\ldots\\
\left\lfloor
\begin{array}{l}
\textup{take}\;e_{\nS}\in L^2(\upOmega,\FE,\PP;\HS)\;
\textup{and}\;
\mu_{\nS}\in L^\infty(\upOmega,\FE,\PP;\left]0,1/\upalpha\right[)\\
x_{\nS+1}=x_{\nS}+\mu_{\nS}\brk1{\TS x_{\nS}+e_{\nS}-x_{\nS}}.
\end{array}
\right.\\
\end{array}
\end{equation}
Set $(\forall\nnn)$ $\upPhi_{\nS}=\{x_{0},\dots,x_{\nS}\}$
and $\XX_{\nS}=\upsigma(\upPhi_{\nS})$. Suppose that 
$\sum_{\nnn}\EE\brk{\mu_{\nS}(1-\upalpha\mu_{\nS})}=\pinf$ and, 
for every $\nnn$, that $\mu_{\nS}$ is independent of
$\upsigma(\{e_{\nS}\}\cup\upPhi_{\nS})$. Then the following hold 
for some $\Fix\TS$-valued random variable $x$:
\begin{enumerate}
\item
Suppose that 
$\EC{\norm{e_{\nS}}_{\HS}^2}{\XX_{\nS}}\to 0\;\Pas$ and
$\sum_{\nnn}\EE\mu_{\nS}\sqrt{\EC{\norm{e_{\nS}}_{\HS}^2}
{\XX_{\nS}}}<\pinf\;\Pas$ Then the following hold:
\begin{enumerate}
\item
\label{c:25ia-}
$(\TS x_{\nS}-x_{\nS})_{\nnn}$ converges strongly $\Pas$ to
$0$.
\item
\label{c:25ia}
$(x_{\nS})_{\nnn}$ converges weakly $\Pas$ to $x$.
\item
\label{c:25ib}
Suppose that $\TS-\Id$ is demiregular at every point in $\Fix\TS$.
Then $(x_{\nS})_{\nnn}$ converges strongly $\Pas$ to $x$.
\end{enumerate}
\item
Suppose that $\EE\norm{e_{\nS}}^2_{\HS}\to0$ and
$\sum_{\nnn}\sqrt{\EE\mu_{\nS}^2\EE\norm{e_{\nS}}_{\HS}^2}<\pinf$.
Then the following hold:
\begin{enumerate}
\item
\label{c:25ic-}
$(\TS x_{\nS}-x_{\nS})_{\nnn}$ converges strongly in
$L^1(\upOmega,\FE,\PP;\HS)$ and strongly $\Pas$ to $0$.
\item
\label{c:25ic}
$x\in L^2(\upOmega,\FE,\PP;\Fix\TS)$ and $(x_{\nS})_{\nnn}$
converges weakly in $L^2(\upOmega,\FE,\PP;\HS)$ and weakly $\Pas$
to $x$.
\item
\label{c:25id}
Suppose that $\TS-\Id$ is demiregular at every point in $\Fix\TS$. 
Then $(x_{\nS})_{\nnn}$ converges strongly in
$L^2(\upOmega,\FE,\PP;\HS)$ and strongly $\Pas$ to $x$.
\end{enumerate}
\end{enumerate}
\end{corollary}
\begin{proof}
Apply Theorem~\ref{t:25} to the nonexpansive operator
$\Id+\upalpha^{-1}(\TS-\Id)$ and observe that it has the same fixed
points as $\TS$.
\end{proof}

\begin{remark}
As discussed in \cite{Opti04,Acnu24}, the Krasnosel'ski\u\i--Mann
iterative process for averaged operators is at the core of monotone
operator splitting strategies such as the three operator splitting
scheme of \cite{Davi17}, the Douglas--Rachford algorithm
\cite{Lion79}, and the constant proximal parameter version of the
forward-backward algorithm \cite{Merc80}. Stochastically relaxed
and perturbed extensions of these algorithms can be derived from
Corollary~\ref{c:25} with weaker assumptions than those of 
\cite[Theorem~4.1]{Pafa16}.
\end{remark}

We now consider a stochastic version of the (forward) Euler method
to find a zero of a cocoercive operator. For simplicity, we adopt
deterministic step-sizes $(\upgamma_{\nS})_{\nnn}$. This result
extends those of \cite{Siop15,Pafa16,Rosa16} by establishing,
under weaker assumptions, weak convergence $\PP$-almost surely and,
in addition, proving for the first time weak convergence in
$L^2(\upOmega,\FE,\PP;\HS)$.

\begin{corollary}
\label{c:26}
Let $\upbeta\in\RPP$ and let $\mathsf{B}\colon\HS\to\HS$ be
$\upbeta$-cocoercive, with $\zer\mathsf{B}=\menge{\mathsf{z}\in\HS}
{\mathsf{B}\mathsf{z}=0}\neq\emp$. Let
$(\mathsf{K},\EuScript{K})$ be a measurable space, let
$k\colon(\upOmega,\FE,\PP)\to(\mathsf{K},\EuScript{K})$ be a random
variable, let $\upxi\in\RPP$, and let 
$(\mathsf{B}_{\kS})_{\kS\in\mathsf{K}}$ be operators from $\HS$ to 
$\HS$ such that $\boldsymbol{\mathsf{B}}\colon(\mathsf{K}\times\HS,
\EuScript{K}\otimes\BE_{\HS})\to(\HS,\BE_{\HS})\colon(\kS,\xS)
\mapsto\mathsf{B}_{\kS}\xS$ is measurable and
\begin{equation}
\label{e:sgd1}
(\forall\xS\in\HS)\quad
\EE(\mathsf{B}_{k}\xS)=\mathsf{B}\xS\quad\text{and}\quad
\EE\norm{\mathsf{B}_{k}\xS-\mathsf{B}\xS}^2_{\HS}\leq\upxi.
\end{equation} 
Let $\upnu\in\left]{2}/{3},1\right]$
and $x_{0}\in L^2(\upOmega,\FE,\PP;\HS)$. Iterate
\begin{equation}
\label{e:c26}
\begin{array}{l}
\textup{for}\;\nS=0,1,\ldots\\
\left\lfloor
\begin{array}{l}
\XX_{\nS}=\upsigma(x_{0},\ldots,x_{\nS})\\
k_{\nS}\;\textup{is a copy of}\;k\;\text{and is independent of}
\;\XX_{\nS}\\
\upgamma_{\nS}=\dfrac{2\upbeta}{(\nS+1)^{\upnu}}\\
x_{\nS+1}=x_{\nS}-\upgamma_{\nS}\mathsf{B}_{k_{\nS}}x_{\nS}.
\end{array}
\right.\\
\end{array}
\end{equation}
Then the following hold for some 
$x\in L^2(\upOmega,\FE,\PP;\zer\mathsf{B})$:
\begin{enumerate}
\item
\label{c:26i-}
$(\mathsf{B}x_{\nS})_{\nnn}$ converges strongly in
$L^1(\upOmega,\FE,\PP;\HS)$ and strongly $\Pas$ to $0$.
\item
\label{c:26i}
$(x_{\nS})_{\nnn}$ converges weakly in $L^2(\upOmega,\FE,\PP;\HS)$ 
and weakly $\Pas$ to $x$.
\item
\label{c:26ii}
Suppose that $\mathsf{B}$ is demiregular at every point in
$\zer\mathsf{B}$. Then $(x_{\nS})_{\nnn}$ converges
strongly in $L^1(\upOmega,\FE,\PP;\HS)$ and strongly $\Pas$ to $x$.
\end{enumerate}
\end{corollary}
\begin{proof}
We apply Corollary~\ref{c:25} to the reduction technique of
\cite{Brav24}. Fix $\uptheta\in\left]0,2\upbeta\right[$ and set
\begin{equation}
(\forall\nnn)\quad
e_{\nS}
=\mathsf{B}x_{\nS}-
\mathsf{B}_{k_{\nS}}x_{\nS}
=\mathsf{B}x_{\nS}-
\boldsymbol{\mathsf{B}}\circ(k_{\nS},x_{\nS})\;\;\text{and}\;\;
\mu_{\nS}=\frac{\upgamma_{\nS}}{\uptheta}\in
\left]0,\frac{2\upbeta}{\uptheta}\right[.
\end{equation}
Then, for every $\nnn$, $e_{\nS}$ is measurable with 
$\EC{e_{\nS}}{\XX_{\nS}}=0$ and
$\EE\mu_{\nS}=\upgamma_{\nS}/\uptheta$. Let us also define
$e'_{0}=\upgamma_{0}e_{0}$ and 
\begin{equation}
(\forall\nnn)\quad 
e'_{\nS+1}=(1-\upgamma_{\nS+1})e'_{\nS}+\upgamma_{\nS+1}e_{\nS+1}.
\end{equation}
Set $\TS=\Id-\uptheta\mathsf{B}$. Then
$\Fix\TS=\zer\mathsf{B}$ and $\TS$ is 
$\uptheta/(2\upbeta)$-averaged
\cite[Proposition~4.39]{Livre1}. Finally, 
define, for every $\nnn$, $y_{\nS}=x_{\nS}-e'_{\nS}$ and
$\YY_{\nS}=\upsigma(y_{0},\ldots,y_{\nS})$. Then, we infer
from \eqref{e:c26} that 
\begin{equation}
\label{e:c25new}
(\forall\nnn)\quad
y_{\nS+1}=y_{\nS}+\upmu_{\nS}\brk1{\TS y_{\nS}+e''_{\nS}-y_{\nS}},
\end{equation}
where 
\begin{equation}
\label{e:sgd2}
(\forall\nnn)\quad
\begin{cases}
e''_{\nS}=\TS x_{\nS}-\TS y_{\nS}\in L^2(\upOmega,\FE,\PP;\HS);\\
\norm{e''_{\nS}}_{\HS}\leq\norm{x_{\nS}-y_{\nS}}_{\HS}
=\norm{e'_{\nS}}_{\HS}\;\;\Pas
\end{cases}
\end{equation}
It follows from the choice of
$(\upgamma_{\nS})_{\nnn}$, the uniformly bounded variance in
\eqref{e:sgd1}, and \cite[Example~2.7 and Theorem~2.5]{Brav24} that
\begin{equation}
\sum_{\nnn}\sqrt{\EE\mu_{\nS}^2\EE\norm{e'_{\nS}}_{\HS}^2}
=\sum_{\nnn}\frac{\upgamma_{\nS}}{\uptheta}
\sqrt{\EE\norm{e'_{\nS}}_{\HS}^2}
<\pinf,
\end{equation}
and
\begin{equation}
\sum_{\nnn}\EE\brk3{\mu_{\nS}
\brk2{1-\frac{\uptheta}{2\upbeta}\mu_{\nS}}}
=\sum_{\nnn}{\frac{\upgamma_{\nS}}{\uptheta}
\brk2{1-\frac{\upgamma_{\nS}}{2\upbeta}}}
=\pinf.
\end{equation}
We also deduce from the
proofs of \cite[Lemma~2.4 and Theorem~2.5]{Brav24} that 
$\EE\norm{e'_{\nS}}_{\HS}^2\to 0$ and $\norm{e'_{\nS}}_{\HS}\to
0\;\Pas$
Consequently, by taking \eqref{e:sgd2} into account, we obtain
\begin{equation}
\displaystyle\sum_{\nnn}\sqrt{\EE\mu_{\nS}^2
\EE\norm{e''_{\nS}}_{\HS}^2}<\pinf\quad\text{and}\quad
\EE\norm{e''_{\nS}}_{\HS}^2\to 0.
\end{equation}

\ref{c:26i-}: It follows from Theorem~\ref{t:25}\ref{t:25ic-}
that $(\mathsf{B}y_{\nS})_{\nnn}$ converges strongly in
$L^1(\upOmega,\FE,\PP;\HS)$ and strongly $\Pas$ to $0$.
On the other hand, \eqref{e:sgd2} implies that
\begin{align}
(\forall\nnn)\quad
\uptheta\norm{\mathsf{B}x_{\nS}-\mathsf{B}y_{\nS}}_{\HS}
&=\norm{\uptheta\mathsf{B}x_{\nS}-\uptheta\mathsf{B}y_{\nS}}_{\HS}
\nonumber\\
&\leq\norm{x_{\nS}-\uptheta\mathsf{B}x_{\nS}-
(y_{\nS}-\uptheta\mathsf{B}y_{\nS})}_{\HS}+
\norm{x_{\nS}-y_{\nS}}_{\HS}\nonumber\\
&=\norm{e''_{\nS}}_{\HS}+\norm{e'_{\nS}}_{\HS}\nonumber\\
&\leq 2\norm{e'_{\nS}}_{\HS}\;\;\Pas
\end{align}
Since $\EE\norm{e'_{\nS}}_{\HS}^2\to 0$ and 
$\norm{e'_{\nS}}_{\HS}\to 0\;\Pas$, we deduce that 
$(\mathsf{B}x_{\nS}-\mathsf{B}y_{\nS})_{\nnn}$ converges strongly 
in $L^1(\upOmega,\FE,\PP;\HS)$ and strongly $\Pas$ to $0$
and, therefore, we obtain the convergence results for
$(\mathsf{B}x_{\nS})_{\nnn}$.

\ref{c:26i}: We infer from Theorem~\ref{t:25}\ref{t:25ic} that
$(y_{\nS})_{\nnn}$ converges weakly in $L^2(\upOmega,\FE,\PP;\HS)$
and weakly $\Pas$ to some $x\in
L^2(\upOmega,\FE,\PP;\zer\mathsf{B})$. 
However, for every $\nnn$, $x_{\nS}=y_{\nS}+e'_{\nS}$.
Since $(e'_{\nS})_{\nnn}$
converges strongly $\Pas$ and strongly in 
$L^2(\upOmega,\FE,\PP;\HS)$ to $0$, we conclude that
$(x_{\nS})_{\nnn}$ converges weakly in 
$L^2(\upOmega,\FE,\PP;\HS)$ and weakly $\Pas$ to $x$.

\ref{c:26ii}: This follows from 
Theorem~\ref{t:25}\ref{t:25id} using the same arguments as
in the proofs of \ref{c:26i-} and \ref{c:26i}.
\end{proof}

The following special case of Corollary~\ref{c:26} concerns
stochastic optimization and establishes new results on the
convergence of the iterates generated by the standard stochastic
gradient method, a method that goes back to the classical work of
\cite{Blum54,Ermo66,Robi51}.

\begin{corollary}
\label{c:27}
Let $\upbeta\in\RPP$ and let $\mathsf{f}\colon\HS\to\RR$ be
convex, differentiable, and such that $\nabla\mathsf{f}$ is
$1/\upbeta$-Lipschitzian, with $\Argmin\mathsf{f}\neq\emp$. Let
$(\mathsf{K},\EuScript{K})$ be a measurable space, let
$k\colon(\upOmega,\FE,\PP)\to(\mathsf{K},\EuScript{K})$ be a random
variable, let $\upxi\in\RPP$, 
and, for every $\kS\in\mathsf{K}$, let
$\mathsf{g}_{\kS}\colon\HS\to\RR$ be differentiable and such that
$\boldsymbol{\mathsf{B}}\colon(\mathsf{K}\times\HS,
\EuScript{K}\otimes\BE_{\HS})\to(\HS,\BE_{\HS})\colon(\kS,\xS)
\mapsto\nabla\mathsf{g}_{\kS}(\xS)$ is measurable and
\begin{equation}
\label{e:sgd12}
(\forall\xS\in\HS)\quad
\EE\nabla\mathsf{g}_{k}(\xS)=\nabla\mathsf{f}(\xS)\quad\text{and}
\quad\EE\norm{\nabla\mathsf{g}_{k}(\xS)
-\nabla\mathsf{f}(\xS)}^2_{\HS}\leq\upxi.
\end{equation} 
Let $\upnu\in\left]{2}/{3},1\right]$
and $x_{0}\in L^2(\upOmega,\FE,\PP;\HS)$. Iterate
\begin{equation}
\label{e:c27}
\begin{array}{l}
\textup{for}\;\nS=0,1,\ldots\\
\left\lfloor
\begin{array}{l}
\XX_{\nS}=\upsigma(x_{0},\ldots,x_{\nS})\\
k_{\nS}\;\textup{is a copy of}\;k\;\textup{and is independent of}
\;\XX_{\nS}\\
\upgamma_{\nS}=\dfrac{2\upbeta}{(\nS+1)^{\upnu}}\\
x_{\nS+1}=x_{\nS}
-\upgamma_{\nS}\nabla\mathsf{g}_{k_{\nS}}(x_{\nS}).
\end{array}
\right.\\
\end{array}
\end{equation}
Then the following hold for some 
$x\in L^2(\upOmega,\FE,\PP;\Argmin\mathsf{f})$:
\begin{enumerate}
\item
\label{c:27i-}
$(\nabla\mathsf{f}(x_{\nS}))_{\nnn}$ converges strongly in
$L^1(\upOmega,\FE,\PP;\HS)$ and strongly $\Pas$ to $0$.
\item
\label{c:27i}
$(x_{\nS})_{\nnn}$ converges weakly in $L^2(\upOmega,\FE,\PP;\HS)$ 
and weakly $\Pas$ to $x$.
\item
\label{c:27ii}
Suppose that $\nabla\mathsf{f}$ is demiregular at every point in 
$\Argmin\mathsf{f}$.
Then $(x_{\nS})_{\nnn}$ converges
strongly in $L^1(\upOmega,\FE,\PP;\HS)$ and strongly $\Pas$ to $x$.
\end{enumerate}
\end{corollary}
\begin{proof}
Apply Corollary~\ref{c:26} to $\mathsf{B}=\nabla\mathsf{f}$, which
is $\upbeta$-cocoercive \cite[Corollary~18.17]{Livre1}, and, for
every $\kS\in\mathsf{K}$,
$\mathsf{B}_{\kS}=\nabla\mathsf{g}_{\kS}$.
\end{proof}

\begin{remark}
In Corollary~\ref{c:27}\ref{c:27i}, the weak convergence $\Pas$
and in $L^2$ results are new. In a finite-dimensional setting, we
recover the $\Pas$ convergence of \cite[Corollary~4.5]{Brav24} 
with the novelty of the $L^1$ convergence. In the
infinite-dimensional setting, we extend the result of \cite{Rosa20}
where the $\Pas$ weak convergence is stated only for a subsequence
of the iterates.
\end{remark}

\begin{remark}
Variants of Corollary~\ref{c:26} can be explored by modifying
the probabilistic assumptions in \eqref{e:sgd1}. In the context of
Corollary~\ref{c:27}, see for instance \cite{Chen24,Jofr19,Nguy19}
and their bibliographies for possible candidates.
\end{remark}

\section{Application to common fixed point problems}
\label{sec:5}

The problem under consideration is a common fixed point problem
involving an arbitrary family of firmly quasinonexpansive
operators. Recall that $\TS\colon\HS\to\HS$ is firmly
quasinonexpansive \cite[Definition~4.1(iv)]{Livre1} if
\begin{equation}
(\forall\xS\in\HS)(\forall\mathsf{y}\in\Fix\TS)\quad
\norm{\TS\xS-\mathsf{y}}_{\HS}^2+\norm{\TS\xS-\xS}_{\HS}^2
\leq\norm{\xS-\mathsf{y}}_{\HS}^2.
\end{equation}

\begin{example}[{\cite[Proposition~2.3]{Mor1}}]
\label{ex:51}
Let $\TS\colon\HS\to\HS$. Then $\TS$ is firmly quasinonexpansive
if one of the following holds:
\begin{enumerate}
\item
\label{ex:51i}
$\CS$ is a nonempty closed convex subset of $\HS$ and
$\TS=\proj_{\CS}$ is the projector onto $\CS$. Here, $\Fix\TS=\CS$.
\item
\label{ex:51ii}
$\mathsf{f}\colon\HS\to\left]\minf,\pinf\right]$ is a proper
lower semicontinuous convex function and
\begin{equation}
\TS=\prox_{\mathsf{f}}\colon\HS\to\HS\colon\xS\mapsto
\underset{\mathsf{y}\in\HS}{\argmin}
\brk2{\mathsf{f}(\mathsf{y})+
\frac{1}{2}\norm{\xS-\mathsf{y}}_{\HS}^2}.
\end{equation}
Here, $\Fix\TS=\Argmin\mathsf{f}$.
\item
\label{ex:51iii}
$\mathsf{A}\colon\HS\to 2^{\HS}$ is maximally monotone and
$\TS=\mathsf{J}_{\mathsf{A}}=(\Id+\mathsf{A})^{-1}$.
Here, $\Fix\TS=\menge{\zS\in\HS}{0\in\mathsf{A}\zS}$.
\item
\label{ex:51iv}
$\mathsf{f}\colon\HS\to\RR$ is a continuous convex function, 
$\mathsf{s}\colon\HS\to\HS\colon\xS\mapsto\mathsf{s}(\xS)\in
\partial\mathsf{f}(\xS)$ is a selection of $\partial\mathsf{f}$, 
and
\begin{equation}
\label{e:sgp}
\TS=\GS_{\mathsf{f}}\colon\HS\to\HS\colon\mathsf{x}\mapsto
\begin{cases}
\xS-\dfrac{\mathsf{f}(\xS)}{\norm{\mathsf{s}(\xS)}_{\HS}^2}
\mathsf{s}(\xS),&\text{if}\;\;\mathsf{f}(\xS)>0;\\
\xS,&\text{if}\;\;\mathsf{f}(\xS)\leq 0,
\end{cases}
\end{equation}
is the subgradient projector onto
$\Fix\TS=\menge{\xS\in\HS}{\mathsf{f}(\xS)\leq0}$.
\end{enumerate} 
\end{example}

The following formulation covers a wide range of problems in
mathematics and its applications \cite{Byrn14,Aiep96,Else01}.

\begin{problem}
\label{prob:5}
Let $(\mathsf{K},\EuScript{K})$ be a measurable space and
$(\TS_{\kS})_{\kS\in\mathsf{K}}$ a family of firmly
quasinonexpansive operators such that 
$\boldsymbol{\TS}\colon(\mathsf{K}\times\HS,
\EuScript{K}\otimes\BE_{\HS})\to(\HS,\BE_{\HS})\colon(\kS,\xS)
\mapsto\TS_{\kS}\xS$ is measurable and, for every 
$\kS\in\mathsf{K}$, $\Id-\TS_{\kS}$ is demiclosed at
$0$. Let
$k\colon(\upOmega,\FE,\PP)\to(\mathsf{K},\EuScript{K})$ be a random
variable. The task is to 
\begin{equation}
\text{find}\;\;\xS\in\ZS=
\menge{\mathsf{z}\in\HS}{\mathsf{z}\in\Fix\TS_{k}\;\Pas},
\end{equation}
under the assumption that $\ZS\neq\emp$.
\end{problem}

\begin{remark}
$\ZS$ is a closed convex subset of $\HS$. Indeed, let 
$(\mathsf{z}_{\nS})_{\nnn}$ be a sequence in $\ZS$ that converges
to $\mathsf{z}\in\HS$. For every $\nnn$, let
$\upOmega_{\nS}\in\FE$ be such that $\PP(\upOmega_{\nS})=1$ and,
for every $\upomega\in\upOmega_{\nS}$, 
$\mathsf{z}_{\nS}\in\Fix\TS_{k(\upomega)}$. Set 
$\upOmega'=\bigcap_{\nnn}\upOmega_{\nS}$. Then
$\PP(\upOmega')=1$ and 
\begin{equation}
(\forall\upomega\in\upOmega')(\forall\nnn)\quad
\mathsf{z}_{\nS}\in\Fix\TS_{k(\upomega)}.
\end{equation}
Since each set of fixed points is closed 
\cite[Proposition~2.3(v)]{Else01}, we deduce that, for every 
$\upomega\in\upOmega'$, 
$\mathsf{z}\in\Fix\TS_{k(\upomega)}$, i.e., $\mathsf{z}\in\ZS$. 
So $\ZS$ is closed. 
Likewise, let $\mathsf{z}_1\in\ZS$,
$\mathsf{z}_2\in\ZS$, and $\upalpha\in\zeroun$. Define almost sure
events $\upOmega_1\in\FE$ and $\upOmega_2\in\FE$ as above. Then, it
follows from the convexity of each set of fixed points
\cite[Proposition~2.3(v)]{Else01} that
\begin{equation}
(\forall\upomega\in\upOmega_1\cap\upOmega_2)\quad
\upalpha\mathsf{z}_1+(1-\upalpha)\mathsf{z}_2\in
\Fix\TS_{k(\upomega)}.
\end{equation}
Since $\PP(\upOmega_1\cap\upOmega_2)=1$, we get
$\upalpha\mathsf{z}_1+(1-\upalpha)\mathsf{z}_2\in\ZS$, which shows
that $\ZS$ is convex.
\end{remark}

We propose the following stochastic variant of the extrapolated
parallel block-iterative fixed point algorithm of \cite{Else01}. It
introduces stochasticity at four levels: 
\begin{itemize}
\item
The operators are indexed on a general measurable space rather than
a countable set.
\item
The block of activated operators is randomly selected at each
iteration.
\item
The evaluations of the operators at iteration $\nS$ are averaged
and extrapolated with random weights 
$(\beta_{\iS,\nS})_{1\leq\iS\leq\mathsf{M}}$.
\item
The relaxation parameter $\lambda_{\nS}$ at iteration $\nS$
is random and not confined to the interval $\left]0,2\right[$ as in
traditional fixed point methods \cite{Livre1,Else01,Dong22}. 
\end{itemize}

\begin{theorem}
\label{t:21}
In the setting of Problem~\ref{prob:5}, let
$x_{0}\in L^2(\upOmega,\FE,\PP;\HS)$, 
$0<\mathsf{M}\in\NN$,
$\updelta\in\left]0,1/\mathsf{M}\right[$, and
$\uprho\in\left[2,\pinf\right[$. Iterate
\begin{equation}
\label{e:p21}
\begin{array}{l}
\textup{for}\;\nS=0,1,\ldots\\
\left\lfloor
\begin{array}{l}
\XX_{\nS}=\upsigma(x_{0},\dots,x_{\nS})\\
\begin{array}{l}
\textup{for}\;\iS=1,\dots,\mathsf{M}\\
\left\lfloor
\begin{array}{l}
k_{\iS,\nS}\;\textup{is a copy of}\;k\;
\textup{and is independent of}\;\XX_{\nS}\\
p_{\iS,\nS}=\TS_{k_{\iS,\nS}}x_{\nS}\\[1mm]
\end{array}
\right.\\
\end{array}\\[5mm]
(\beta_{\iS,\nS})_{1\leq\iS\leq\mathsf{M}}\;\text{are}\;
\left[0,1\right]\textup{-valued random variables such that}\\
\qquad\sum_{\iS=1}^\mathsf{M}\beta_{\iS,\nS}=1\;\Pas
\;\textup{and}\;\;
(\forall\iS\in\{1,\dots,\mathsf{M}\})\;
\beta_{\iS,\nS}\geq\updelta
\mathsf{1}_{\bigl[\norm{p_{\iS,\nS}-x_{\nS}}_{\HS}=
\max\limits_{1\leq\jS\leq\mathsf{M}}
\norm{p_{\jS,\nS}-x_{\nS}}_{\HS}\bigr]}\;\\
p_{\nS}=\sum_{\iS=1}^\mathsf{M}\beta_{\iS,\nS}p_{\iS,\nS}\\[1mm]
L_{\nS}=\dfrac{{\sum_{\iS=1}^{\mathsf{M}}\beta_{\iS,\nS}
\|p_{\iS,\nS}-x_{\nS}\|_{\HS}^2}+
\mathsf{1}_{\bigl[p_{\nS}=x_{\nS}\bigr]}}
{\|p_{\nS}-x_{\nS}\|_{\HS}^2+
\mathsf{1}_{\bigl[p_{\nS}=x_{\nS}\bigr]}}\\[5mm]
a_{\nS}=x_{\nS}+L_{\nS}(p_{\nS}-x_{\nS})\\
\textup{take}\;
\lambda_{\nS}\in L^\infty(\upOmega,\FE,\PP;\left]0,\uprho\right])\\
x_{\nS+1}=x_{\nS}+\lambda_{\nS}\brk{a_{\nS}-x_{\nS}}.
\end{array}
\right.\\
\end{array}
\end{equation}
Suppose that there exists $\upmu\in\zeroun$ such that
$\inf_{\nnn}\EE\brk{\lambda_{\nS}(2-\lambda_{\nS})}\geq\upmu$ and
that, for every $\nnn$, $\lambda_{\nS}$ is independent of 
$\upsigma(p_{1,\nS},\ldots,p_{\mathsf{M},\nS},\beta_{1,\nS},\ldots,
\beta_{\mathsf{M},\nS},x_{0},\ldots,x_{\nS})$. Then the
following hold for some $x\in L^2(\upOmega,\FE,\PP;\ZS)$:
\begin{enumerate} 
\item
\label{t:21i}
$(x_{\nS})_{\nnn}$ converges weakly in 
$L^2(\upOmega,\FE,\PP;\HS)$ and weakly $\Pas$ to $x$.
\item
\label{t:21ii}
Suppose that there exists $\mathsf{S}\in\FE$ such that
\begin{equation}
\label{e:45}
\mathsf{S}\subset\menge{\upomega\in\upOmega}
{\TS_{k(\upomega)}-\Id\;\text{is demiregular at every point in}\;
\Fix\TS_{k(\upomega)}}\;\;\text{and}\;\;\PP(\mathsf{S})>0.
\end{equation}
Then $(x_{\nS})_{\nnn}$ converges strongly in 
$L^1(\upOmega,\FE,\PP;\HS)$ and strongly $\Pas$ to $x$.
\item
\label{t:21iii}
Suppose that one of the following is satisfied:
\begin{enumerate}[label=\normalfont{[\Alph*]}]
\item
\label{t:21iiiA}
There exists $\upchi\in\zeroun$ such that
\begin{equation}
(\forall\nnn)\quad\EC1{\dS_{\ZS}^2(x_{\nS+1})}{\XX_{\nS}}\leq\upchi
\dS_{\ZS}^2(x_{\nS})\;\;\Pas
\end{equation}
\item
\label{t:21iiiB}
$\boldsymbol{\TS}$ is linearly regular in the sense that there
exists $\upnu\in\left[1,\pinf\right[$ such that
\begin{equation}
\label{e:lr}
(\forall\xS\in\HS)\quad\dS_{\ZS}^2(\xS)
\leq\upnu\EE\norm{\TS_{k}\xS-\xS}_{\HS}^2
=\upnu\int_{\upOmega}\norm1{\TS_{k(\upomega)}\xS-\xS}_{\HS}^2
\PP(d\upomega),
\end{equation}
in which case we set 
$\upzeta=\inf_{\jjj}\EE\lambda_{\jS}^2$ and 
$\upchi=1-\upmu\updelta\upzeta
/(\uprho^2\upnu)$.
\end{enumerate}
Then $(x_{\nS})_{\nnn}$ converges strongly in 
$L^2(\upOmega,\FE,\PP;\HS)$ and strongly $\Pas$ to $x$, and
\begin{equation}
\label{e:21.2}
(\forall\nnn)\quad \EE\|x_{\nS}-x\|_{\HS}^2\leq 4\upchi^\nS\EE 
\dS^2_{\ZS}(x_{0}).
\end{equation}
\end{enumerate}
\end{theorem}
\begin{proof}
We define 
\begin{equation}
\label{e:403}
(\forall\nnn)\quad
\begin{cases}
t_{\nS}^*=x_{\nS}-p_{\nS};\\
\eta_{\nS}=\Sum_{\iS=1}^{\mathsf{M}}
\beta_{\iS,\nS}\scal{p_{\iS,\nS}}{x_{\nS}-p_{\iS,\nS}}_{\HS};\\
\alpha_{\nS}=\dfrac{\mathsf{1}_{[t_{\nS}^*\neq0]}\mathsf{1}_{
\left[\scal{x_{\nS}}{t_{\nS}^*}_{\HS}>\eta_{\nS}\right]}
\bigl(\scal{x_{\nS}}{t_{\nS}^*}_{\HS}-\eta_{\nS}\bigr)}
{\|t_{\nS}^*\|_{\HS}^2+\mathsf{1}_{[t_{\nS}^*=0]}};\\
\varepsilon_{\nS}=0\;\:\Pas;\\
d_{\nS}=x_{\nS}-a_{\nS}
\end{cases}
\end{equation}
and shall show that, in this setting, the sequence
$(x_{\nS})_{\nnn}$ constructed by \eqref{e:p21} corresponds to 
one generated by Algorithm~\ref{algo:4}. Let $\nnn$. We first infer
from \eqref{e:403} and \eqref{e:p21} that
\begin{align}
d_{\nS}
&=x_{\nS}-a_{\nS}\nonumber\\
&=L_{\nS}\brk1{x_{\nS}-p_{\nS}}\nonumber\\
&=\dfrac{{\sum_{\iS=1}^{\mathsf{M}}\beta_{\iS,\nS}
\|p_{\iS,\nS}-x_{\nS}\|_{\HS}^2}+
\mathsf{1}_{\bigl[p_{\nS}=x_{\nS}\bigr]}}
{\|p_{\nS}-x_{\nS}\|_{\HS}^2+
\mathsf{1}_{\bigl[p_{\nS}=x_{\nS}\bigr]}}
\brk1{x_{\nS}-p_{\nS}}
\nonumber\\
&=\dfrac{{\sum_{\iS=1}^{\mathsf{M}}\beta_{\iS,\nS}
\|x_{\nS}-p_{\iS,\nS}\|_{\HS}^2}+
\mathsf{1}_{\bigl[t^*_{\nS}=0\bigr]}}
{\|t_{\nS}^*\|_{\HS}^2+
\mathsf{1}_{\bigl[t^*_{\nS}=0\bigr]}}t_{\nS}^*
\nonumber\\
&=\dfrac{{\sum_{\iS=1}^{\mathsf{M}}\beta_{\iS,\nS}
\brk1{\scal{x_{\nS}}{x_{\nS}-p_{\iS,\nS}}_{\HS}
-\scal{p_{\iS,\nS}}{x_{\nS}-p_{\iS,\nS}}_{\HS}}}}
{\|t_{\nS}^*\|_{\HS}^2+
\mathsf{1}_{\bigl[t^*_{\nS}=0\bigr]}}
t_{\nS}^*
\nonumber\\
&=\dfrac{\scal{x_{\nS}}{t_{\nS}^*}_{\HS}
-\sum_{\iS=1}^{\mathsf{M}}\beta_{\iS,\nS}
\scal{p_{\iS,\nS}}{x_{\nS}-p_{\iS,\nS}}_{\HS}}
{\|t_{\nS}^*\|_{\HS}^2+
\mathsf{1}_{\bigl[t^*_{\nS}=0\bigr]}}
t_{\nS}^*\nonumber\\
&=\alpha_{\nS}t_{\nS}^*\;\Pas
\label{e:998}
\end{align}
Next, let us show that
\begin{equation}
\label{e:405}
L_{\nS}\geq1\;\:\Pas
\end{equation}
Fix $\zS\in\ZS$ and, for every 
$\iS\in\{1,\ldots,\mathsf{M}\}$, let
$\upOmega_{\iS,\nS}\in\FE$ be such that 
\begin{equation}
\label{e:s7}
\PP(\upOmega_{\iS,\nS})=1\quad\text{and}\quad
(\forall\upomega\in\upOmega_{\iS,\nS})\quad
\zS\in\Fix\TS_{k_{\iS,\nS}(\upomega)}. 
\end{equation}
Thanks to \eqref{e:p21}, we then choose
$\upOmega_{\nS}\in\FE$ such that 
\begin{equation}
\PP(\upOmega_{\nS})=1\quad\text{and}\quad
(\forall\upomega\in\upOmega_{\nS})\quad
\bigcap_{1\leq\iS\leq\mathsf{M}}\Fix\TS_{k_{\iS,\nS}(\upomega)}
\neq\emp\quad\text{and}\quad
\sum_{\iS=1}^\mathsf{M}\beta_{\iS,\nS}(\upomega)=1.
\end{equation}
Given $\upomega\in\upOmega_{\nS}$, we consider the following two
cases:
\begin{itemize}
\item
Suppose that $p_{\nS}(\upomega)=x_{\nS}(\upomega)$. Then 
\cite[Proposition~2.4]{Else01} yields
$x_{\nS}(\upomega)\in\Fix(\sum_{\iS=1}^{\mathsf{M}}
\beta_{\iS,\nS}(\upomega)\TS_{k_{\iS,\nS}(\upomega)})=
\bigcap_{1\leq\iS\leq\mathsf{M}}\Fix\TS_{k_{\iS,\nS}(\upomega)}$,
hence, $(\forall\iS\in\{1,\ldots,\mathsf{M}\})$ 
$x_{\nS}(\upomega)=p_{\iS,\nS}(\upomega)$. Thus,
\begin{equation}
L_{\nS}(\upomega)=
\dfrac{{\Sum_{\iS=1}^{\mathsf{M}}\beta_{\iS,\nS}(\upomega)
\|p_{\iS,\nS}(\upomega)-x_{\nS}(\upomega)\|_{\HS}^2}+
\mathsf{1}_{\bigl[p_{\nS}=x_{\nS}\bigr]}(\upomega)}
{\|p_{\nS}(\upomega)-x_{\nS}(\upomega)\|_{\HS}^2+
\mathsf{1}_{\bigl[p_{\nS}=x_{\nS}\bigr]}(\upomega)}
=\dfrac{\mathsf{1}_{\bigl[p_{\nS}=x_{\nS}\bigr]}(\upomega)}
{\mathsf{1}_{\bigl[p_{\nS}=x_{\nS}\bigr]}(\upomega)}=1.
\end{equation}
\item
Suppose that
$p_{\nS}(\upomega)\neq x_{\nS}(\upomega)$. Then it
follows from the convexity of $\|\cdot\|_{\HS}^2$ that
\begin{equation}
0<\norm{p_{\nS}(\upomega)-x_{\nS}(\upomega)}_{\HS}^2
=\norm3{\sum_{\iS=1}^{\mathsf{M}}
\beta_{\iS,\nS}(\upomega)
\brk1{p_{\iS,\nS}(\upomega)-x_{\nS}(\upomega)}}_{\HS}^2
\leq\sum_{\iS=1}^{\mathsf{M}}\beta_{\iS,\nS}(\upomega)
\norm1{p_{\iS,\nS}(\upomega)-x_{\nS}(\upomega)}_{\HS}^2,
\end{equation}
which implies that $L_{\nS}(\upomega)\geq1$. 
\end{itemize}
In view of \eqref{e:a2}, our next task is to show by induction that
$(x_{\nS})_{\nnn}$ and $(t_{\nS}^*)_{\nnn}$ are in 
$L^2(\upOmega,\FE,\PP;\HS)$, and that $(\eta_{\nS})_{\nnn}$ is in
$L^1(\upOmega,\FE,\PP;\RR)$. To this end, let $\nS\in\NN$ and 
$\iS\in\{1,\ldots,\mathsf{M}\}$, and suppose that 
$x_{\nS}\in L^2(\upOmega,\FE,\PP;\HS)$. Then
$\TS_{k_{\iS,\nS}}x_{\nS}=
\boldsymbol{\TS}\circ(k_{\iS,\nS},x_{\nS})$ is measurable. On the
other hand, for every $\upomega\in\upOmega_{\iS,\nS}$, 
$2\TS_{k_{\iS,\nS}(\upomega)}-\Id$ is quasinonexpansive 
with $\Fix(2\TS_{k_{\iS,\nS}(\upomega)}-\Id)=
\Fix\TS_{k_{\iS,\nS}(\upomega)}$
\cite[Proposition~2.2(v)]{Else01} and hence
\begin{align}
\label{e:p21.1}
2\norm1{p_{\iS,\nS}(\upomega)}^2_{\HS}
&=\dfrac{1}{2}
\norm1{2\TS_{k_{\iS,\nS}(\upomega)}x_{\nS}(\upomega)}^2_{\HS}
\nonumber\\
&\leq\norm1{\brk1{2\TS_{k_{\iS,\nS}(\upomega)}
-\Id}x_{\nS}(\upomega)-\zS}^2_{\HS}
+\norm1{x_{\nS}(\upomega)+\zS}^2_{\HS}
\nonumber\\
&\leq\norm1{x_{\nS}(\upomega)-\zS}^2_{\HS}
+\norm1{x_{\nS}(\upomega)+\zS}^2_{\HS}.
\end{align}
Consequently, since 
$x_{\nS}\in L^2(\upOmega,\FE,\PP;\HS)$ and $\zS\in\HS$, we have
$p_{\iS,\nS}\in L^2(\upOmega,\FE,\PP;\HS)$ and \eqref{e:p21}
therefore yields
$p_{\nS}\in L^2(\upOmega,\FE,\PP;\HS)$. Thus, 
$t_{\nS}^*=x_{\nS}-p_{\nS}\in L^2(\upOmega,\FE,\PP;\HS)$.
On the other hand, it follows from the Cauchy--Schwarz inequalities
in $\HS$ as well in $L^2(\upOmega,\FE,\PP;\RR)$ that
\begin{align}
\EE\Bigl|\scal{p_{\iS,\nS}}{x_{\nS}-p_{\iS,\nS}}_{\HS}\Bigr|
\leq\EE\brk2{\bigl\|p_{\iS,\nS}\bigr\|_{\HS}
\bigl\|x_{\nS}-p_{\iS,\nS}\bigr\|_{\HS}}
\leq\sqrt{\EE\bigl\|p_{\iS,\nS}\bigr\|_{\HS}^2
\EE\bigl\|x_{\nS}-p_{\iS,\nS}\bigr\|_{\HS}^2}
<\pinf,
\end{align}
which shows that 
$\scal1{p_{\iS,\nS}}{x_{\nS}-p_{\iS,\nS}}_{\HS}
\in L^1(\upOmega,\FE,\PP;\RR)$. 
Since this is true for every $\iS\in\{1,\ldots,\mathsf{M}\}$, we 
obtain $\eta_{\nS}\in L^1(\upOmega,\FE,\PP;\RR)$. Further, it
follows from \cite[Proposition~4.2(iv)]{Livre1} that
\begin{equation}
\label{e:c1}
(\forall\iS\in\{1,\dots,\mathsf{M}\})\quad
\scal1{p_{\iS,\nS}-\zS}{x_{\nS}-p_{\iS,\nS}}_{\HS}
=\scal1{\TS_{k_{\iS,\nS}}x_{\nS}-\zS}
{x_{\nS}-\TS_{k_{\iS,\nS}}x_{\nS}}_{\HS}\geq0\;\;\Pas
\end{equation}
In turn, the concavity of
$\mathsf{y}\mapsto\scal{\mathsf{y}-\zS}
{x_{\nS}(\upomega)-\mathsf{y}}_{\HS}$ yields
\begin{equation}
\label{e:c2}
0\leq\sum_{\iS=1}^{\mathsf{M}}\beta_{\iS,\nS}
\scal1{p_{\iS,\nS}-\zS}{x_{\nS}-p_{\iS,\nS}}_{\HS}
\leq \scal1{p_{\nS}-\zS}{x_{\nS}-p_{\nS}}_{\HS}
=\scal1{x_{\nS}-t_{\nS}^*-\zS}{t_{\nS}^*}_{\HS}\;\;\Pas
\end{equation}
and therefore
\begin{align}
\dfrac{1}{2}
\EE\abs3{\dfrac{\mathsf{1}_{[t_{\nS}^*\neq0]}\mathsf{1}_{
\left[\scal{x_{\nS}}{t_{\nS}^*}_{\HS}>\eta_{\nS}\right]}\eta_{\nS}}
{\|t_{\nS}^*\|_{\HS}+\mathsf{1}_{[t_{\nS}^*=0]}}}^2
&\leq\EE\abs3{\dfrac{\eta_{\nS}
-\sum_{\iS=1}^{\mathsf{M}}\beta_{\iS,\nS}
\scal{\zS}{x_{\nS}-p_{\iS,\nS}}_{\HS}}
{\|t_{\nS}^*\|_{\HS}+\mathsf{1}_{[t_{\nS}^*=0]}}}^2
+\EE\abs3{\dfrac{\sum_{\iS=1}^{\mathsf{M}}\beta_{\iS,\nS}
\scal{\zS}{x_{\nS}-p_{\iS,\nS}}_{\HS}}
{\|t_{\nS}^*\|_{\HS}+\mathsf{1}_{[t_{\nS}^*=0]}}}^2
\nonumber\\
&=\EE\abs3{\dfrac{
\sum_{\iS=1}^\mathsf{M}\beta_{\iS,\nS}\scal{p_{\iS,\nS}
-\zS}{x_{\nS}-p_{\iS,\nS}}_{\HS}}
{\|t_{\nS}^*\|_{\HS}+\mathsf{1}_{[t_{\nS}^*=0]}}}^2
+\EE\abs3{\dfrac{\scal{\zS}{t_{\nS}^*}_{\HS}}
{\|t_{\nS}^*\|_{\HS}+\mathsf{1}_{[t_{\nS}^*=0]}}}^2
\nonumber\\
&\leq\EE\abs3{
\dfrac{\scal1{x_{\nS}-t_{\nS}^*-\zS}{t_{\nS}^*}_{\HS}}
{\|t_{\nS}^*\|_{\HS}+\mathsf{1}_{[t_{\nS}^*=0]}}}^2
+\EE\abs3{
\dfrac{\|\zS\|_{\HS}\|t_{\nS}^*\|_{\HS}}
{\|t_{\nS}^*\|_{\HS}+\mathsf{1}_{[t_{\nS}^*=0]}}}^2
\nonumber\\
&\leq\EE\abs3{
\dfrac{\|x_{\nS}-t_{\nS}^*-\zS\|_{\HS}\|t_{\nS}^*\|_{\HS}}
{\|t_{\nS}^*\|_{\HS}+\mathsf{1}_{[t_{\nS}^*=0]}}}^2
+\EE\|\zS\|_{\HS}^2
\nonumber\\
&\leq\EE\|x_{\nS}-t_{\nS}^*-\zS\|_{\HS}^2
+\|\zS\|_{\HS}^2.
\label{e:911}
\end{align}
Since $\{x_{\nS},t_{\nS}^*\}\subset L^2(\upOmega,\FE,\PP;\HS)$
and $\zS\in\HS$, we thus obtain
$\mathsf{1}_{[t_{\nS}^*\neq0]}\eta_{\nS}/
(\|t_{\nS}^*\|_{\HS}+\mathsf{1}_{[t_{\nS}^*=0]})\in
L^2(\upOmega,\FE,\PP;\RR)$. Hence, arguing as in 
\eqref{e:999}, we deduce from \eqref{e:998} that 
\begin{equation}
\label{e:404}
x_{\nS}-a_{\nS}=\alpha_{\nS}{t}_{\nS}^*
\in L^2(\upOmega,\FE,\PP;\HS). 
\end{equation}
It therefore results from \eqref{e:p21} that 
$x_{\nS+1}\in L^2(\upOmega,\FE,\PP;\HS)$, which completes the
induction argument. On the other hand, since
$\alpha_{\nS}\in\RP\;\Pas$, \eqref{e:c2} yields
\begin{equation}
\scal1{\mathsf{z}}{\alpha_{\nS}t_{\nS}^*}_{\HS}
=\alpha_{\nS}\sum_{\iS=1}^{\mathsf{M}}\beta_{\iS,\nS}
\scal1{\mathsf{z}}{x_{\nS}-p_{\iS,\nS}}_{\HS}
\leq\alpha_{\nS}\sum_{\iS=1}^{\mathsf{M}}\beta_{\iS,\nS}
\scal1{p_{\iS,\nS}}{x_{\nS}-p_{\iS,\nS}}_{\HS}
=\alpha_{\nS}\eta_{\nS}\;\;\Pas
\label{e:FejPas}
\end{equation}
Thus, appealing to Lemma~\ref{l:2} and \eqref{e:403}, we obtain
\begin{equation}
\scal2{\mathsf{z}}{\EC1{\alpha_{\nS}t_{\nS}^*}{\XX_{\nS}}}_{\HS}
=\EC2{\scal1{\mathsf{z}}{\alpha_{\nS}t_{\nS}^*}_{\HS}}{\XX_{\nS}}
\leq
\EC1{\alpha_{\nS}\eta_{\nS}}{\XX_{\nS}}+\varepsilon_{\nS}\;\;\Pas
\end{equation}
Altogether, the sequence $(x_{\nS})_{\nnn}$ constructed by 
\eqref{e:p21} corresponds to one generated by 
Algorithm~\ref{algo:4}.
Now set $\upzeta=\inf_{\jjj}\EE\lambda_{\nS}^2$ and note that
$\upzeta\geq\inf_{\jjj}\EE^2\lambda_{\jS}\geq\upmu^2/4>0$.
Hence, we infer from \eqref{e:p21} and
Lemma~\ref{l:6} that
\begin{align}
\label{e:21.32}
\EC1{\|x_{\nS+1}-x_{\nS}\|_{\HS}^2}{\XX_{\nS}} 
&=\EC2{\norm1{\lambda_{\nS}\brk1{a_{\nS}-x_{\nS}}}_{\HS}^2}
{\XX_{\nS}}\nonumber\\
&=\EC2{\norm1{\lambda_{\nS}L_{\nS}\brk1{p_{\nS}-x_{\nS}}}_{\HS}^2}
{\XX_{\nS}}\nonumber\\
&=\EC2{\abs1{\lambda_{\nS}L_{\nS}}^2
\norm1{p_{\nS}-x_{\nS}}_{\HS}^2}{\XX_{\nS}}\nonumber\\
&=\EC3{\lambda_{\nS}^2L_{\nS}
\sum_{\iS=1}^{\mathsf{M}}\beta_{\iS,\nS}
\norm1{p_{\iS,\nS}-x_{\nS}}_{\HS}^2}{\XX_{\nS}}\nonumber\\
&\geq\EC3{\lambda_{\nS}^2
\sum_{\iS=1}^{\mathsf{M}}\beta_{\iS,\nS}
\norm1{p_{\iS,\nS}-x_{\nS}}_{\HS}^2}{\XX_{\nS}}\nonumber\\
&=\brk1{\EE\lambda_{\nS}^2}\EC3{
\sum_{\iS=1}^{\mathsf{M}}\beta_{\iS,\nS}
\norm1{p_{\iS,\nS}-x_{\nS}}_{\HS}^2}{\XX_{\nS}}\nonumber\\
&\geq\upzeta\EC3{
\sum_{\iS=1}^{\mathsf{M}}\beta_{\iS,\nS}
\norm1{p_{\iS,\nS}-x_{\nS}}_{\HS}^2}{\XX_{\nS}}\nonumber\\
&\geq\upzeta
\EC2{\updelta\max_{1\leq\jS\leq\mathsf{M}}
\norm1{p_{\jS,\nS}-x_{\nS}}_{\HS}^2}{\XX_{\nS}}\nonumber\\
&\geq\updelta\upzeta\EC2{
\norm1{p_{1,\nS}-x_{\nS}}_{\HS}^2}{\XX_{\nS}}\nonumber\\
&=\updelta\upzeta\EC2{
\norm1{\TS_{k_{1,\nS}}x_{\nS}-x_{\nS}}_{\HS}^2}{\XX_{\nS}}.
\end{align}
However, since $k_{1,\nS}$ is independent of $\XX_{\nS}$, 
Lemma~\ref{l:7}
implies that, for $\PP$-almost every $\upomega'\in\upOmega$, 
\begin{align}
\EC2{\norm1{\TS_{k_{1,\nS}}x_{\nS}-x_{\nS}}_{\HS}^2}{\XX_{\nS}}
(\upomega')
&=\int_{\upOmega}\norm1{\TS_{k_{1,\nS}(\upomega)}
x_{\nS}(\upomega')-x_{\nS}(\upomega')}_{\HS}^2\PP(d\upomega)
\nonumber\\
&=\int_{\upOmega}\norm1{\TS_{k(\upomega)}
x_{\nS}(\upomega')-x_{\nS}(\upomega')}_{\HS}^2\PP(d\upomega).
\end{align}
Therefore, for $\PP$-almost every $\upomega'\in\upOmega$, 
\eqref{e:21.32} implies that 
\begin{equation}
\EC1{\|x_{\nS+1}-x_{\nS}\|_{\HS}^2}{\XX_{\nS}}(\upomega')
\geq\updelta\upzeta\int_{\upOmega}\norm1{\TS_{k(\upomega)}
x_{\nS}(\upomega')-x_{\nS}(\upomega')}_{\HS}^2\PP(d\upomega)\;\;
\Pas
\end{equation}
Upon taking the expected value in \eqref{e:21.32}, summing over
$\nnn$, and invoking Theorem~\ref{t:2}\ref{t:2viie}, we obtain 
\begin{equation}
\label{e:condb}
\EE\brk3{\sum_{\nnn}\int_{\upOmega}
\norm1{\TS_{k(\upomega)}x_{\nS}-x_{\nS}}_{\HS}^2\PP(d\upomega)}
=\sum_{\nnn}\EE\brk3{\int_{\upOmega}
\norm1{\TS_{k(\upomega)}x_{\nS}-x_{\nS}}_{\HS}^2\PP(d\upomega)}
<\pinf.
\end{equation}
Hence, 
\begin{equation}
\label{e:21.4}
\sum_{\nnn}\int_{\upOmega}\norm1{
\TS_{k(\upomega)}x_{\nS}-x_{\nS}}_{\HS}^2
\PP(d\upomega)<\pinf\;\;\Pas
\end{equation} 
Let $\upOmega'\in\FE$ such that 
$\PP(\upOmega')=1$ and
\begin{equation}
\brk1{\forall\upomega'\in\upOmega'}\quad\sum_{\nnn}
\int_{\upOmega}
\norm1{\TS_{k(\upomega)}
x_{\nS}(\upomega')-x_{\nS}(\upomega')}_{\HS}^2\PP(d\upomega)<\pinf
\quad\text{and}\quad\WC\brk1{x_{\nS}(\upomega')}_{\nnn}\neq
\emp.
\end{equation}
The existence of such a set $\upOmega'$ follows from 
\eqref{e:21.4} as well as Theorem~\ref{t:3}\ref{t:3iva}.
Fix $\upomega'\in\upOmega'$ and let $x(\upomega')\in
\WC(x_{\nS}(\upomega'))_{\nnn}$, say 
$x_{\jS_{\nS}}(\upomega')\weakly x(\upomega')$. On the other hand,
it follows from the monotone convergence theorem that
\begin{equation}
\int_{\upOmega}\sum_{\nnn}
\norm1{\TS_{k(\upomega)}x_{\nS}(\upomega')
-x_{\nS}(\upomega')}_{\HS}^2\PP(d\upomega)
=\sum_{\nnn}\int_{\upOmega}
\norm1{\TS_{k(\upomega)}x_{\nS}(\upomega')
-x_{\nS}(\upomega')}_{\HS}^2\PP(d\upomega)<\pinf.
\end{equation}
Hence, for $\PP$-almost every $\upomega\in\upOmega$, 
$\sum_{\nnn}\norm{\TS_{k(\upomega)}x_{\nS}(\upomega')
-x_{\nS}(\upomega')}_{\HS}^2<\pinf$.
Therefore, there exists 
$\upOmega''\in\FE$ such that $\PP(\upOmega'')=1$ and 
\begin{equation}
\label{e:demicl}
\brk1{\forall\upomega\in\upOmega''}\quad 
\TS_{k(\upomega)}x_{\nS}(\upomega')-
x_{\nS}(\upomega')\to 0.
\end{equation}
It then follows from the demiclosedness of the operators 
$(\Id-\TS_{\kS})_{\kS\in\mathsf{K}}$ at $0$ that
\begin{equation}
\label{e:mi}
\brk1{\forall\upomega\in\upOmega''}\quad
\TS_{k(\upomega)}x(\upomega')=x(\upomega').
\end{equation}
Therefore $x(\upomega')\in\menge{\mathsf{z}\in\HS}{\mathsf{z}\in
\Fix\TS_{k}\;\Pas}=\ZS$. Since $\upomega'$ is arbitrarily taken in 
$\upOmega'$, we conclude that 
\begin{equation}
\label{e:ii}
\WC(x_{\nS})_{\nnn}\subset\ZS\;\Pas
\end{equation}

\ref{t:21i}: This follows from \eqref{e:ii} and
Theorems~\ref{t:2}\ref{t:2ivf} and \ref{t:2}\ref{t:2weak}.

\ref{t:21ii}: 
Let $\upomega'\in\upOmega'$. 
In view of \eqref{e:45}, \eqref{e:mi}, and \eqref{e:demicl}, there
exists $\FE\ni\upOmega'''\subset\upOmega''$ such that 
$\PP(\upOmega''')>0$ and
\begin{equation}
\label{e:46}
\brk1{\forall\upomega\in\upOmega'''}\quad
\begin{cases}
\TS_{k(\upomega)}-\Id\;\text{is demiregular at}\;x(\upomega');\\
\TS_{k(\upomega)}x_{\nS}(\upomega')-x_{\nS}(\upomega')\to 0.
\end{cases}
\end{equation}
However, \ref{t:21i} implies that, for $\PP$-almost every
$\upomega'\in\upOmega$, 
$x_{\nS}(\upomega')\weakly x(\upomega')$. Therefore,
by demiregularity, for $\PP$-almost every
$\upomega'\in\upOmega$, we deduce from \eqref{e:46} that
$x_{\nS}(\upomega')\to x(\upomega')$. Thus, $(x_{\nS})_{\nnn}$
converges strongly $\Pas$ to $x$. Finally, the strong convergence
in $L^1(\upOmega,\FE,\PP;\HS)$ follows from
Theorem~\ref{t:2}\ref{t:2viig}.

\ref{t:21iii}: This follows from Theorem~\ref{t:2}\ref{t:2viiI}
when \ref{t:21iiiA} holds. It remains to show that \ref{t:21iiiB}
implies \ref{t:21iiiA}. Let us first show that 
\begin{equation}
\label{e:let.0}
\upchi\in\zeroun.
\end{equation}
First, the concavity of $\upxi\mapsto\upxi(2-\upxi)$ and Jensen's
inequality yield
\begin{equation}
\label{e:let.1}
0<\upmu\leq\inf_{\nS\in\NN}
\mathsf{E}\brk1{\lambda_{\nS}(2-\lambda_{\nS})}
\leq\inf_{\nS\in\NN}
\mathsf{E}\lambda_{\nS}(2-\mathsf{E}\lambda_{\nS}).
\end{equation}
This quadratic inequality forces
\begin{equation}
\label{e:let.2}
0<1-\sqrt{1-\upmu}\leq\inf_{\nS\in\NN}\mathsf{E}\lambda_{\nS},
\end{equation}
and Jensen's inequality guarantees that
$0<\inf_{\nS\in\NN}\mathsf{E}\lambda_{\nS}^2=\upzeta$. Next, since
$\upmu\in\left]0,1\right[$,
$\updelta\in\left]0,1\right[$, $\upnu\in\left[1,\pinf\right[$,
$\uprho\in\left[2,\pinf\right[$, and
$\lambda_{\nS}\in\left]0,\uprho\right]\;\mathsf{P}$-a.s., we have
$\lambda_{\nS}^2/\uprho^2\in\left]0,1\right]\;\mathsf{P}$-a.s. and
\begin{equation}
\label{e:let.3}
\dfrac{\upzeta}{\uprho^2}
=\dfrac{\inf_{\nS\in\NN}\mathsf{E}\lambda_{\nS}^2}{\uprho^2}
\in\left]0,1\right].
\end{equation}
It follows then that
$\upmu\updelta\upzeta/(\uprho^2\upnu)\in\left]0,1\right[$ and
therefore that
$\upchi=1-\upmu\updelta\upzeta/(\uprho^2\upnu)\in\left]0,1\right[$.
Next, let $\nnn$ and $z\in L^2(\upOmega,\XX_{\nS},\PP;\ZS)$. 
Theorem~\ref{t:3}\ref{t:3ii}, the independence assumption 
for $\lambda_{\nS}$, and \eqref{e:a2} imply that
\begin{align}
\label{e:nj1}
\EC1{\norm{x_{\nS+1}-z}_{\HS}^2}{\XX_{\nS}}
&\leq\norm{x_{\nS}-z}_{\HS}^2
-\EE\brk1{\lambda_{\nS}(2-\lambda_{\nS})}
\EC1{\norm{d_{\nS}}_{\HS}^2}{\XX_{\nS}}\nonumber\\
&=\norm{x_{\nS}-z}_{\HS}^2
-\EE\brk1{\lambda_{\nS}(2-\lambda_{\nS})}
\EC3{\frac{1}{\lambda_{\nS}^2}
\norm{x_{\nS+1}-x_{\nS}}_{\HS}^2}{\XX_{\nS}}\;\;\Pas
\end{align}
Upon taking $z=\proj_{\ZS}x_{\nS}$ in \eqref{e:nj1}, 
\begin{align}
\EC1{\dS_{\ZS}^2(x_{\nS+1})}{\XX_{\nS}}
&\leq\dS_{\ZS}^2(x_{\nS})
-\EE\brk1{\lambda_{\nS}(2-\lambda_{\nS})}
\EC3{\frac{1}{\lambda_{\nS}^2}
\norm{x_{\nS+1}-x_{\nS}}_{\HS}^2}{\XX_{\nS}}\nonumber\\
&\leq\dS_{\ZS}^2(x_{\nS})
-\upmu
\EC3{\frac{1}{\lambda_{\nS}^2}
\norm{x_{\nS+1}-x_{\nS}}_{\HS}^2}{\XX_{\nS}}\nonumber\\
&\leq\dS_{\ZS}^2(x_{\nS})
-\frac{\upmu}{\uprho^2}
\EC1{\norm{x_{\nS+1}-x_{\nS}}_{\HS}^2}{\XX_{\nS}}.
\end{align}
Thus, for $\PP$-almost every $\upomega'\in\upOmega$, 
we derive from \eqref{e:21.32} that
\begin{align}
\EC1{\dS_{\ZS}^2(x_{\nS+1})}{\XX_{\nS}}(\upomega')
&\leq\dS_{\ZS}^2(x_{\nS})(\upomega')
-\frac{\upmu\updelta\upzeta}{\uprho^2}
\int_{\upOmega}\norm1{\TS_{k(\upomega)}
x_{\nS}(\upomega')-x_{\nS}(\upomega')}_{\HS}^2
\PP(d\upomega)\nonumber\\
&\leq\upchi\dS_{\ZS}^2(x_{\nS})(\upomega').
\end{align}
Hence, $\EC{\dS_{\ZS}^2(x_{\nS+1})}{\XX_{\nS}}
\leq\upchi\dS_{\ZS}^2(x_{\nS})\;\Pas$ and, in view of 
\eqref{e:let.0}, \ref{t:21iiiA} holds. The conclusion follows from 
Theorem~\ref{t:2}\ref{t:2viiI}.
\end{proof}

\begin{remark}\
\label{r:30}
\begin{enumerate}
\item
In Algorithm~\eqref{e:p21}, $\mathsf{M}$ is the batch size, i.e.,
the number of activated sets, $p_{\nS}$ is the standard average of
the selected operators, $L_{\nS}\geq 1$ is the
extrapolation parameter, $a_{\nS}$ is the extrapolated average, and
$\lambda_{\nS}$ is the relaxation parameter, which can exceed
the standard bound $2$ imposed by deterministic methods
\cite{Else01}.
\item
Problem~\ref{prob:5} is studied in \cite{Flam95} for firmly
nonexpansive operators with errors. A deterministic algorithm
which activates all the operators at each iteration via a Bochner
integral average is proposed.
The weak convergence to a solution is established; see also
\cite{Butn95} for a version in the context of projectors of
Example~\ref{ex:51}\ref{ex:51i}. This result contrasts with
Theorem~\ref{t:21} in which the convergence is guaranteed even
when a finite number of operators are activated at each iteration. 
\item
\label{r:30i}
In \eqref{e:p21}, we need not impose a lower bound on the weights
$(\beta_{\iS,\nS})_{1\leq\iS\leq\mathsf{M}}$ if we assume that, for
every $\iS\in\{1,\ldots,\mathsf{M}\}$, $\beta_{\iS,\nS}$ is 
independent of
$\upsigma(p_{\iS,\nS},x_{0},\ldots,x_{\nS})$. Indeed,
in such a case, Lemma~\ref{l:6} asserts that
\begin{align}
\EC3{\sum_{\iS=1}^{\mathsf{M}}
\beta_{\iS,\nS}\norm{p_{\iS,\nS}-x_{\nS}}_{\HS}^2}{\XX_{\nS}}
&=\sum_{\iS=1}^{\mathsf{M}}\EC1{
\beta_{\iS,\nS}\norm{p_{\iS,\nS}-x_{\nS}}_{\HS}^2}{\XX_{\nS}}
\nonumber\\
&=\sum_{\iS=1}^{\mathsf{M}}\brk1{\EE\beta_{\iS,\nS}}
\EC1{\norm{p_{1,\nS}-x_{\nS}}_{\HS}^2}{\XX_{\nS}}
\nonumber\\
&=\EC1{\norm{p_{1,\nS}-x_{\nS}}_{\HS}^2}{\XX_{\nS}}.
\end{align}
\item
\label{r:30ii}
Suppose that, for every $\kS\in\mathsf{K}$, 
$\TS_{\kS}\colon\HS\to\HS$ is continuous. Then, to obtain the 
joint measurability of $\boldsymbol{\TS}$, it is enough to 
suppose that, for every $\xS\in\HS$,
$\boldsymbol{\TS}(\cdot,\xS)\colon\kS\mapsto\TS_{\kS}\xS$ is
measurable \cite[Lemma~4.51]{Alip06}.
\end{enumerate}
\end{remark}

\begin{remark}
In the literature, convergence to solutions has been established 
in specific instances of Problem~\ref{prob:5} and algorithm 
\eqref{e:p21}.
\begin{enumerate}
\item
Several works have focused on the sequential unrelaxed case, that
is, the scenario in which 
\begin{equation}
\label{e:rmk1}
\mathsf{M}=1,\;\lambda_{\nS}=1,\;\text{and therefore}\;
x_{\nS+1}=a_{\nS}=p_{\nS}=p_{1,\nS}=\mathsf{T}_{k_{1,\nS}}x_{\nS}.
\end{equation}
In the context of the projectors of
Example~\ref{ex:51}\ref{ex:51i},
\cite{Nedi10} guarantees almost sure convergence to a solution
when $\HS=\RR^{\mathsf{N}}$ and $\mathsf{K}$ is finite. This result
is also found in \cite{Bric22} and in \cite{Kost23}. The setting of
\cite{Kost23} involves a Euclidean space $\HS$ and a general 
measurable space $(\mathsf{K},\EuScript{K})$, and it also shows
convergence in $L^2(\upOmega,\FE,\PP;\HS)$. When the subsets are 
half-spaces or 
when the interior of $\ZS$ is nonempty, \cite{Nedi10} provides a 
rate for convergence in $L^2(\upOmega,\FE,\PP;\HS)$.
For general separable Hilbert spaces and under the assumption that
the operators are averaged mappings, \cite{Herm19} shows weak
almost sure convergence. In addition, a convergence rate is
established in $L^1(\upOmega,\FE,\PP;\HS)$ when \eqref{e:lr} is
satisfied. The paper \cite{Nedi11} involves deterministic
relaxations $\uplambda_{\nS}\in\left]0,2\right[$ in the context of
subgradient projectors of Example~\ref{ex:51}\ref{ex:51iv} in
$\HS=\RR^{\mathsf{N}}$. Assuming that \eqref{e:lr} holds and,
additionally, that the subgradients are uniform bounded, almost
sure convergence to a solution is established.
\item
We now discuss works that have studied algorithms for
$\mathsf{M}>1$. In \cite{Kolo22}, $\mathsf{K}$ 
is countable, extrapolations are not allowed (hence
$a_{\nS}=p_{\nS}$), 
$\lambda_{\nS}$ is a deterministic parameter in 
$\left]0,2\right]$, and the condition 
$\inte\ZS\neq\emp$ is imposed. Finite convergence is established.
In the context of projectors in
$\HS=\RR^{\mathsf{N}}$, a similar approach to
Algorithm~\ref{algo:1} is studied in \cite{Neco22} and
\cite{Neco19} with the following restrictions:
deterministic relaxations $(\uplambda_{\nS})_{\nnn}$ in 
$\left]0,2\right[$ and iteration-independent fixed deterministic
weights $\upbeta_{\iS,\nS}\equiv 1/\mathsf{M}$. Mean-square rates
of convergence are established by assuming that \eqref{e:lr} holds,
as well as ergodic convergence results. However, almost sure 
convergence is not proved.
Similarly, \cite{Neco21} and \cite{Nedi19} use a deterministic
relaxation sequence $(\uplambda_{\nS})_{\nnn}$ in
$\left]0,2\right[$ and iteration-independent fixed deterministic
weights $\upbeta_{\iS,\nS}\equiv 1/\mathsf{M}$ to solve
Problem~\ref{prob:5} in the context of subgradient projectors in
$\HS=\RR^{\mathsf{N}}$. Under linear regularity assumptions and,
additionally, uniform boundedness of the subgradients, rates of
convergence in mean-square are provided. Nevertheless, almost sure
convergence of the sequence of iterates is not guaranteed.
\end{enumerate}
\end{remark}

\begin{remark}
\label{r:57}
By combining the proofs of Theorem~\ref{t:25} and
Theorem~\ref{t:21}, it is possible to establish convergence results
for the following error-tolerant algorithm for solving
Problem~\ref{prob:5}: Let 
$x_{0}\in L^2(\upOmega,\FE,\PP;\HS)$, 
$0<\mathsf{M}\in\NN$, and
$\updelta\in\left]0,1/\mathsf{M}\right[$. Iterate
\begin{equation}
\label{e:p212}
\begin{array}{l}
\text{for}\;\nS=0,1,\ldots\\
\left\lfloor
\begin{array}{l}
\XX_{\nS}=\upsigma(x_{0},\dots,x_{\nS})\\
\begin{array}{l}
\text{for}\;\iS=1,\dots,\mathsf{M}\\
\left\lfloor
\begin{array}{l}
k_{\iS,\nS}\;\text{is a copy of}\;k\;\text{and is independent of}
\;\XX_{\nS}\\
\text{take}\;e_{\iS,\nS}\in L^2(\upOmega,\FE,\PP;\HS)\\
p_{\iS,\nS}=\TS_{k_{\iS,\nS}}x_{\nS}+e_{\iS,\nS}\\[1mm]
\end{array}
\right.\\
\end{array}\\[5mm]
(\beta_{\iS,\nS})_{1\leq\iS\leq\mathsf{M}}\;\text{are}\;
\left[0,1\right]\text{-valued random variables such that}\\
\qquad\sum_{\iS=1}^\mathsf{M}\beta_{\iS,\nS}=1\;\Pas
\;\text{and}\;\;
(\forall\iS\in\{1,\dots,\mathsf{M}\})\;
\beta_{\iS,\nS}\geq\updelta
\mathsf{1}_{\bigl[\norm{p_{\iS,\nS}-x_{\nS}}_{\HS}=
\max\limits_{1\leq\jS\leq\mathsf{M}}
\norm{p_{\jS,\nS}-x_{\nS}}_{\HS}\bigr]}\;\\
p_{\nS}=\sum_{\iS=1}^\mathsf{M}\beta_{\iS,\nS}p_{\iS,\nS}\\[1mm]
\text{take}\;
\lambda_{\nS}\in L^\infty(\upOmega,\FE,\PP;\left]0,2\right[)\\
x_{\nS+1}=x_{\nS}+\lambda_{\nS}\brk{p_{\nS}-x_{\nS}}.
\end{array}
\right.\\
\end{array}
\end{equation}
Suppose that
$\inf_{\nnn}\EE\brk{\lambda_{\nS}(2-\lambda_{\nS})}>0$,
$\max_{1\leq\iS\leq\mathsf{M}}\sum_{\nnn}
\sqrt{\EE\norm{e_{\iS,\nS}}_{\HS}^2}<\pinf$,
and that, for every $\nnn$, $\lambda_{\nS}$ is independent of 
$\upsigma(k_{1,\nS},\ldots,k_{\mathsf{M},\nS},e_{1,\nS},\ldots,
e_{\mathsf{M},\nS},\beta_{1,\nS},\ldots,
\beta_{\mathsf{M},\nS},x_{0},\ldots,x_{\nS})$.
Then there exists $x\in L^2(\upOmega,\FE,\PP;\ZS)$ such that
$(x_{\nS})_{\nnn}$ converges weakly in 
$L^2(\upOmega,\FE,\PP;\HS)$ and weakly $\Pas$ to $x$.
\end{remark}

\section{Numerical experiments}
\label{sec:6}

We illustrate numerically our results in the context of
Problem~\ref{prob:5} with applications of the stochastic
algorithm~\eqref{e:p21} with the deterministic relaxation
strategies
\begin{equation}
\label{e:ex-1}
(\forall\nnn)\quad\lambda_{\nS}=1.0
\end{equation}
and
\begin{equation}
\label{e:ex0}
(\forall\nnn)\quad\lambda_{\nS}=1.9.
\end{equation}
We also consider the random relaxation strategies
\begin{equation}
\label{e:ex1}
(\forall\nnn)\quad\PP([\lambda_{\nS}=2.3])=\dfrac{1}{2}\;
\text{and}\;\PP([\lambda_{\nS}=1.5])=\dfrac{1}{2}
\end{equation}
and
\begin{equation}
\label{e:ex2}
(\forall\nnn)\quad\lambda_{\nS}\sim\unif([1.5,2.3]).
\end{equation}
Note that \eqref{e:ex1} and \eqref{e:ex2} are super relaxation
strategies that satisfy, for every $\nnn$,
$\EE(\lambda_{\nS}(2-\lambda_{\nS}))>0$,
$\PP([\lambda_{\nS}>2])>0$, and $\EE\lambda_{\nS}=1.9$. 
Problem~\ref{prob:5} is specialized to the standard Euclidean space 
$\HS=\RR^\mathsf{N}$ with $\norm{\cdot}_{\HS}=\norm{\cdot}$,
$\mathsf{K}=\{1,\ldots,\mathsf{p}\}$, and
$k\sim\textrm{uniform}(\mathsf{K})$. 

\begin{problem}
\label{prob:6}
For every $\kS\in\{1,\ldots,\mathsf{p}\}$,
$\mathsf{f}_{\kS}\colon\RR^{\mathsf{N}}\to\RR$ is a convex 
function and $\mathsf{C}_{\kS}=\menge{\xS\in\RR^{\mathsf{N}}}
{\mathsf{f}_{\kS}(\xS)\leq 0}$. It is assumed that
$\ZS=\bigcap_{1\leq\kS\leq\mathsf{p}}\CS_{\kS}\neq\emp$. 
The task is to
\begin{equation}
\label{e:p1}
\text{find}\;\xS\in\RR^{\mathsf{N}}\;\text{such that}\;\xS\in\ZS.
\end{equation}
\end{problem}

Consider the setting of Problem~\ref{prob:6}. For every
$\kS\in\{1,\ldots,\mathsf{p}\}$, let
$\TS_{\kS}\colon\RR^{\mathsf{N}}\to\RR^{\mathsf{N}}$ be the
subgradient projector onto $\mathsf{C}_{\kS}$ of
Example~\ref{ex:51}\ref{ex:51iv}, so that,
by \cite[Propositions~16.20 and 29.41]{Livre1}
\begin{equation}
\TS_{\kS}\;\text{is firmly quasinonexpansive},\;
\Fix\TS_{\kS}=\CS_{\kS},\;\text{and}\;
\Id-\TS_{\kS}\;\text{is demiclosed at $0$}.
\end{equation}
Subgradient projectors extend the classical projection operators in
the following sense. Let $\CS$ be a nonempty closed and convex
subset of $\RR^{\mathsf{N}}$ and suppose that 
$\mathsf{f}_{\kS}=\dS_{\CS}$. Then $\CS_{\kS}=\CS$
and $\GS_{\kS}=\proj_{\CS}$ \cite[Example~29.44]{Livre1}.
Their importance in solving Problem~\ref{prob:6} stems from the
fact that they are generally much easier to
implement than exact ones.

\subsection{Signal restoration}
\label{sec:61}

\begin{figure}[b]
\centering
\begin{tikzpicture}[scale=0.545]
\definecolor{darkgray176}{RGB}{176,176,176}
\begin{axis}[height=7cm, width=29cm, legend columns=1 
columns=1 
cell align={left}, xmin=0, xmax=1024, ymin=-1.4, ymax=1.1,
xtick distance=100,
ytick distance=0.4,
x grid style={darkgray176},xmajorgrids,
y grid style={darkgray176},ymajorgrids,
tick label style={font=\Large}, 
legend cell align={left},
legend style={at={(0.02,0.15)},anchor=west},
axis line style=thick,]
\addplot [ultra thick, dred]
table {figures/ex1/original_signal.txt};
\end{axis}
\end{tikzpicture}\\
(a)\\[0.3em]
\begin{tikzpicture}[scale=0.545]
\definecolor{darkgray176}{RGB}{176,176,176}
\begin{axis}[height=7cm, width=29.0cm, legend columns=1 
columns=1 
cell align={left}, xmin=0, xmax=1024, ymin=-1.4, ymax=1.1,
xtick distance=100,
ytick distance=0.4,
x grid style={darkgray176},xmajorgrids,
y grid style={darkgray176},ymajorgrids,
tick label style={font=\Large}, 
legend cell align={left},
legend style={at={(0.02,0.15)},anchor=west},
axis line style=thick,]
\addplot [ultra thick, dblue]
table {figures/ex1/err1.txt};
\end{axis}
\end{tikzpicture}\\
(b)\\[0.3em]
\begin{tikzpicture}[scale=0.545]
\definecolor{darkgray176}{RGB}{176,176,176}
\begin{axis}[height=7cm, width=29.0cm, legend columns=1 
legend columns=1 
cell align={left}, xmin=0, xmax=1024, ymin=-1.4, ymax=1.1,
xtick distance=100,
ytick distance=0.4,
x grid style={darkgray176},xmajorgrids,
y grid style={darkgray176},ymajorgrids,
tick label style={font=\Large}, 
legend cell align={left},
legend style={at={(0.02,0.15)},anchor=west},
axis line style=thick,]
\addplot [ultra thick, teal]
table {figures/ex1/extrap_1.txt};
\end{axis}
\end{tikzpicture}\\
(c)\\[0.3em]
\caption{Experiment of Section~\ref{sec:61}.
(a): Original signal $\overline{\mathsf{x}}$.
(b): Noisy observation $r_1$.
(c): Solution produced by algorithm~\eqref{e:p21}.}
\label{fig:sign}
\end{figure}
The goal is to recover the original signal 
$\overline{\mathsf{x}}\in\RR^\mathsf{N}$ ($\mathsf{N}=1024$) 
shown in Fig.~\ref{fig:sign}(a) from $20$ noisy
observations 
$(r_{\kS})_{1\leq\kS\leq 20}$ given by
\begin{equation}
\label{e:exp1}
(\forall\kS\in\{1,\dots,20\})\quad
r_{\kS}=\LS_{\kS}\overline{\mathsf{x}}+w_{\kS}
\end{equation}
where $\LS_{\kS}\colon\RR^\mathsf{N}\to\RR^\mathsf{N}$ is a known
linear operator, $\upeta_{\kS}\in\RPP$, and 
$w_{\kS}\in\left[-\upeta_{\kS},\upeta_{\kS}\right]^\mathsf{N}$ is
a bounded random noise vector. The parameters
$(\upeta_{\kS})_{1\leq\kS\leq20}\in\RPP^{20}$ are
known. The operators $(\LS_{\kS})_{1\leq\kS\leq 20}$ 
are Gaussian convolution filters with zero mean and standard
deviation taken uniformly in $[10,30]$, $\upeta_{\kS}=0.15$, and
$w_{\kS}$ is taken uniformly
in $\left[-\upeta_{\kS},\upeta_{\kS}\right]^\mathsf{N}$.
Set, for every $\kS\in\{1,\dots,20\}$ and every
$\jS\in\{1,\dots,\mathsf{N}\}$, 
\begin{equation}
\CS_{\kS,\jS}=\menge{\xS\in\RR^{\mathsf{N}}}
{-\upeta_{\kS}\leq\scal{\LS_{\kS}\xS-r_{\kS}}{\mathsf{e}_{\jS}}
\leq\upeta_{\kS}}. 
\end{equation}
Since the intersection of these sets is nonempty and 
their projectors are computable explicitly 
\cite[Example~29.21]{Livre1}, we solve the feasibility problem
\begin{equation}
\label{e:exp3.3}
\text{find}\;\xS\in\RR^\mathsf{N}\;\text{such that}\quad
(\forall\kS\in\{1,\ldots,20\})
(\forall\jS\in\{1,\ldots,\mathsf{N}\})\;\;
\xS\in\CS_{\kS,\jS}
\end{equation}
by algorithm~\eqref{e:p21} implemented with exact projectors.
We run two instances with $x_{0}=0$.
In the first one, $\mathsf{M}=1$. Note that the relaxation 
scheme of \eqref{e:ex-1} leads to the
almost sure convergence result of \cite{Nedi10} (see also 
\cite{Kost23}), while the relaxation schemes
\eqref{e:ex0}--\eqref{e:ex2} are new even in this specialized
context of randomly activated projection method. In the second
instance $\mathsf{M}=16$. Fig.~\ref{fig:ex3.2} displays the
normalized error versus execution time.

Fig.~\ref{fig:ex3.2} (top) shows the benefits of large relaxations
when $\mathsf{M}=1$. Fig.~\ref{fig:ex3.2} (bottom) shows the
advantage of using $\mathsf{M}>1$ random blocks, in which case the
extrapolation parameter $L_{\nS}$
is not equal to $1$ and can attain large values. This behavior was
previously observed for deterministic algorithms
\cite{Nume06,Coap12,Imag97,Pier84}. Fig.~\ref{fig:ex3.2} also
suggests that, on a single run, the use of the proposed random
super relaxation scheme can further improve the speed of
convergence. It is worth noting that the execution time can
naturally be reduced if Algorithm~\ref{algo:1} is
implemented on a multi-core architecture where, at each iteration,
each (subgradient) projector is assigned to a dedicated core and
all the cores work in parallel.

\begin{figure}[hb]
\centering
\begin{tikzpicture}[scale=0.545]
\definecolor{darkgray176}{RGB}{176,176,176}
\begin{axis}[height=8cm, width=14cm, legend columns=1 
cell align={left}, xmin=0, xmax=10, ymin=-90, ymax=-20,
xtick distance=2,
ytick distance=20,
x grid style={darkgray176},xmajorgrids,
y grid style={darkgray176},ymajorgrids,
tick label style={font=\Large}, 
legend cell align={left},
legend style={at={(0.02,0.15)},anchor=west},
axis line style=thick,]
\addplot [ultra thick, ngreen]
table {figures/ex1/M1Lam1a.txt};
\addplot [ultra thick, hotmagenta]
table {figures/ex1/M1Lam19a.txt};
\addplot [ultra thick, azure]
table {figures/ex1/M1LamBIa.txt};
\addplot [ultra thick, Brown]
table {figures/ex1/M1LamUNIa.txt};
\end{axis}
\end{tikzpicture}
\begin{tikzpicture}[scale=0.545]
\definecolor{darkgray176}{RGB}{176,176,176}
\begin{axis}[height=8cm, width=14cm, legend columns=1 
cell align={left}, xmin=0, xmax=10, ymin=-90, ymax=-20,
xtick distance=2,
ytick distance=20,
x grid style={darkgray176},xmajorgrids,
y grid style={darkgray176},ymajorgrids,
tick label style={font=\Large}, 
legend cell align={left},
legend style={at={(0.02,0.15)},anchor=west},
axis line style=thick,]
\addplot [ultra thick, ngreen]
table {figures/ex1/M1Lam1o.txt};
\addplot [ultra thick, hotmagenta]
table {figures/ex1/M1Lam19o.txt};
\addplot [ultra thick, azure]
table {figures/ex1/M1LamBIo.txt};
\addplot [ultra thick, Brown]
table {figures/ex1/M1LamUNIo.txt};
\end{axis}
\end{tikzpicture}\\
\begin{tikzpicture}[scale=0.545]
\definecolor{darkgray176}{RGB}{176,176,176}
\hspace{-1.6mm}
\begin{axis}[height=8cm, width=14cm, legend columns=1 
cell align={left}, xmin=0, xmax=4, ymin=-90, ymax=-20,
xtick distance=1,
ytick distance=20,
x grid style={darkgray176},xmajorgrids,
y grid style={darkgray176},ymajorgrids,
tick label style={font=\Large}, 
legend cell align={left},
legend style={at={(0.02,0.15)},anchor=west},
axis line style=thick,]
\addplot [ultra thick, ngreen]
table {figures/ex1/M16Lam1a.txt};
\addplot [ultra thick, hotmagenta]
table {figures/ex1/M16Lam19a.txt};
\addplot [ultra thick, azure]
table {figures/ex1/M16LamBIa.txt};
\addplot [ultra thick, Brown]
table {figures/ex1/M16LamUNIa.txt};
\end{axis}
\end{tikzpicture}
\begin{tikzpicture}[scale=0.545]
\definecolor{darkgray176}{RGB}{176,176,176}
\begin{axis}[height=8cm, width=14cm, legend columns=1 
cell align={left}, xmin=0, xmax=4, ymin=-90, ymax=-20,
xtick distance=1,
ytick distance=20,
x grid style={darkgray176},xmajorgrids,
y grid style={darkgray176},ymajorgrids,
tick label style={font=\Large}, 
legend cell align={left},
legend style={at={(0.02,0.15)},anchor=west},
axis line style=thick,]
\addplot [ultra thick, ngreen]
table {figures/ex1/M16Lam1o.txt};
\addplot [ultra thick, hotmagenta]
table {figures/ex1/M16Lam19o.txt};
\addplot [ultra thick, azure]
table {figures/ex1/M16LamBIo.txt};
\addplot [ultra thick, Brown]
table {figures/ex1/M16LamUNIo.txt};
\end{axis}
\end{tikzpicture}\\
\caption{Experiment of Section~\ref{sec:61}. 
Normalized error $20\log(\|x_{\nS}-x_{\infty}\|/
\|x_{0}-x_{\infty}\|)$ (dB) versus execution time (s) on
single processor machine for various relaxation strategies.
{\color{ngreen} Green}: \eqref{e:ex-1}.
{\color{hotmagenta} Magenta}: \eqref{e:ex0}. 
{\color{azure} Blue}: \eqref{e:ex1}. 
{\color{Brown} Brown}: \eqref{e:ex2}. 
Left: Average over ten runs.
Right: A single run.
Top: $\mathsf{M}=1$. 
Bottom: $\mathsf{M}=16$.
}
\label{fig:ex3.2}
\end{figure}

\subsection{Image restoration}
\label{sec:62}

The goal is to recover the original image
$\bar{\xS}\in\RR^{\mathsf{N}\times\mathsf{N}}$ $(\mathsf{N}=256)$ 
shown in Fig.~\ref{fig:ex2.1}(a) from four observations
$\{r_1,\ldots,r_4\}$ which are given by the
degradation of $\bar{\xS}$ via a Gaussian kernel with a standard
deviation of $8$ and the
addition of random noise. The noise distribution is 
$\unif([0,5]^{\mathsf{N}\times\mathsf{N}})$. Let $\LS$ be the 
block-Toeplitz matrix associated with the convolutional blur. Then 
\begin{equation}
(\forall\kS\in\{1,2,3,4\})\quad
r_{\kS}=\LS\bar{\xS}+w_{\kS},\quad
\text{where}\quad
w_{\kS}\sim\unif\brk1{[0,5]^{\mathsf{N}\times\mathsf{N}}}.
\end{equation}
\begin{figure}[t]
\centering
\begin{tabular}{c@{}c@{}c@{}}
\includegraphics[width=5.0cm,height=5.0cm]{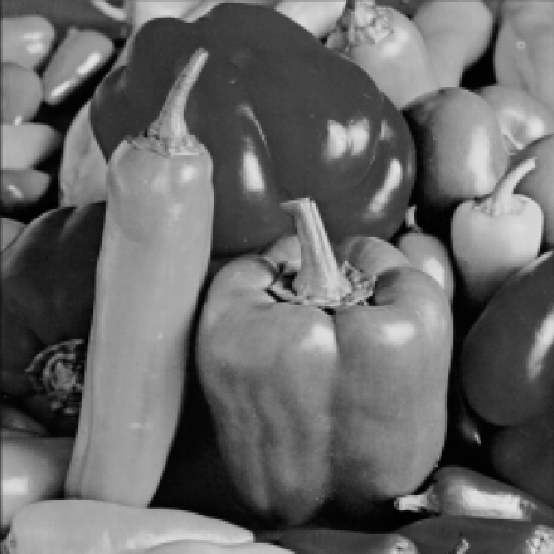}&
\hspace{0.1cm}
\includegraphics[width=5.0cm,height=5.0cm]{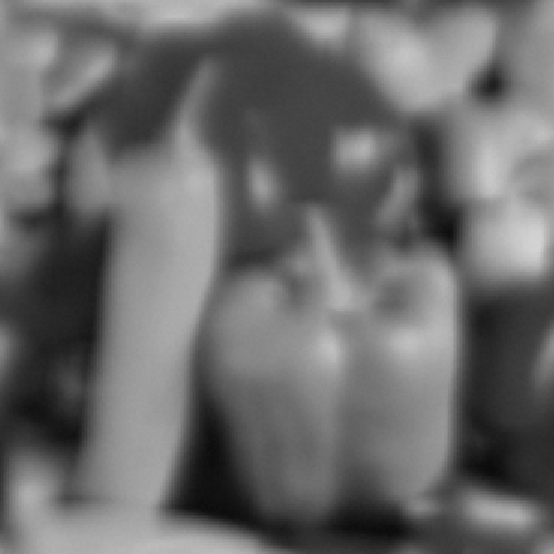}&
\hspace{0.1cm}
\includegraphics[width=5.0cm,height=5.0cm]{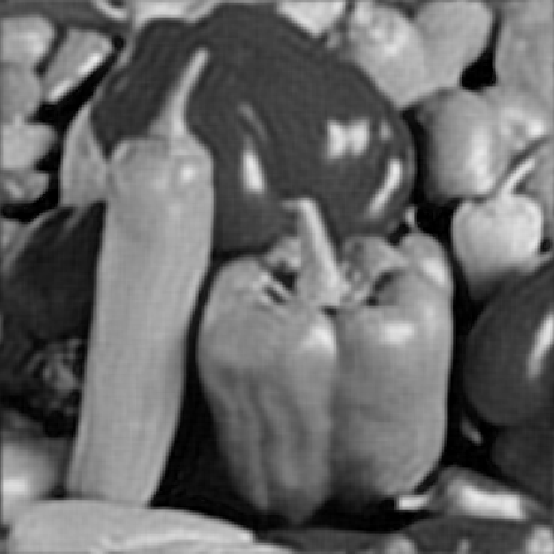}\\
\small{(a)}&\small{(b)}&\small{(c)}
\end{tabular} 
\caption{
Experiment of Section~\ref{sec:62}. 
(a) Original image $\bar{\xS}$. 
(b) Noisy observation $r_1$.
(c) Solution produced by algorithm~\eqref{e:p21}.}
\label{fig:ex2.1}
\end{figure}%
The entries of the random vectors $(w_{\kS})_{1\leq\kS\leq4}$ are 
i.i.d. Therefore, as shown in \cite{Sign91}, for every
$\kS\in\{1,2,3,4\}$, 
with a $95\%$ confidence coefficient 
\begin{equation}
\bar{\xS}\in\mathsf{C}_{\kS}
=\menge{\xS\in\RR^{\mathsf{N}\times\mathsf{N}}}
{\norm{r_{\kS}-\LS\xS}^2\leq\upxi}, 
\end{equation}
where $\upxi=\mathsf{N}^2\EE\abs{u}^2+
1.96\mathsf{N}\sqrt{\EE\abs{u}^4-\EE^2\abs{u}^2}$ 
with $u\sim\unif([0,5])$.
For every $\kS\in\{1,2,3,4\}$, we compute the subgradient projector
onto $\CS_{\kS}$ in \eqref{e:sgp} via the function
$\mathsf{f}_{\kS}\colon\xS\mapsto\norm{r_{\kS}-\LS\xS}^2-
\upxi$. 
In addition, the boundedness on pixel values is incorporated as the
property set $\CS_5=[0,255]^{\mathsf{N}\times\mathsf{N}}$. Finally,
it is assumed that the discrete Fourier transform 
$\mathfrak{F}(\bar{\xS})$ of $\bar{\xS}$ is
known on a portion of its support for low frequencies in both
directions. Specifically, let $\mathsf{S}$ be the set of frequency
pairs $\{0,\ldots,\mathsf{N}/8-1\}^2$ as well as those resulting
from the symmetry properties of the 2D discrete Fourier transform
of real images. The associated set is 
$\CS_6=\menge{\xS\in\RR^{\mathsf{N}\times\mathsf{N}}}
{\mathfrak{F}(\xS)\mathsf{1}_{\mathsf{S}}
=\mathfrak{F}(\bar{\xS})\mathsf{1}_{\mathsf{S}}}$ and its
projection is given by 
$\proj_{\CS_6}\colon\xS\mapsto
\mathfrak{F}^{-1}(\mathfrak{F}(\bar{\xS})\mathsf{1}_{\mathsf{S}}+
\mathfrak{F}(\xS)\mathsf{1}_{\complement\mathsf{S}})$.
We run algorithm~\eqref{e:p21} with $x_{0}=0$ and
$\mathsf{M}=2$. 
Fig.~\ref{fig:ex2.2} displays the normalized error versus 
execution time. These results confirm the conclusions of 
Section~\ref{sec:61}.

\begin{figure}[h]
\centering
\begin{tikzpicture}[scale=0.545]
\definecolor{darkgray176}{RGB}{176,176,176}
\begin{axis}[height=8cm, width=14cm, legend columns=1 
cell align={left}, xmin=0, xmax=325, ymin=-90, ymax=-40,
xtick distance=50,
ytick distance=10,
x grid style={darkgray176},xmajorgrids,
y grid style={darkgray176},ymajorgrids,
tick label style={font=\Large}, 
legend cell align={left},
legend style={at={(0.02,0.15)},anchor=west},
axis line style=thick,]
\addplot [ultra thick, ngreen]
table {figures/ex2/M2Lam1a.txt};
\addplot [ultra thick, hotmagenta]
table {figures/ex2/M2Lam19a.txt};
\addplot [ultra thick, azure]
table {figures/ex2/M2LamBIa.txt};
\addplot [ultra thick, Brown]
table {figures/ex2/M2LamUNIa.txt};
\end{axis}
\end{tikzpicture}
\begin{tikzpicture}[scale=0.545]
\definecolor{darkgray176}{RGB}{176,176,176}
\begin{axis}[height=8cm, width=14cm, legend columns=1 
cell align={left}, xmin=0, xmax=325, ymin=-90, ymax=-40,
xtick distance=50,
ytick distance=10,
x grid style={darkgray176},xmajorgrids,
y grid style={darkgray176},ymajorgrids,
tick label style={font=\Large}, 
legend cell align={left},
legend style={at={(0.02,0.15)},anchor=west},
axis line style=thick,]
\addplot [ultra thick, ngreen]
table {figures/ex2/M2Lam1o.txt};
\addplot [ultra thick, hotmagenta]
table {figures/ex2/M2Lam19o.txt};
\addplot [ultra thick, azure]
table {figures/ex2/M2LamBIo.txt};
\addplot [ultra thick, Brown]
table {figures/ex2/M2LamUNIo.txt};
\end{axis}
\end{tikzpicture}
\caption{Experiment of Section~\ref{sec:62} using $\mathsf{M}=2$. 
Normalized error $20\log(\|x_{\nS}-x_{\infty}\|/
\|x_{0}-x_{\infty}\|)$ (dB) versus execution time (s)
on a single processor machine for various relaxation strategies.
{\color{ngreen} Green}: \eqref{e:ex-1}.
{\color{hotmagenta} Magenta}: \eqref{e:ex0}. 
{\color{azure} Blue}: \eqref{e:ex1}.
{\color{Brown} Brown}: \eqref{e:ex2}.
Left: Average over ten runs.
Right: A single run.
}
\label{fig:ex2.2}
\end{figure}

\medskip
\noindent
{\bfseries Acknowledgment.} The authors thank Dr. Minh N. B\`ui for
bringing to their attention a gap in the proof of an earlier
version of Theorem~\ref{t:21}.

\end{document}